\documentclass{article}
\usepackage{cite}

\usepackage{amsthm}
\usepackage{array}
\usepackage{booktabs}
\usepackage{amsmath}
\usepackage{amssymb}
\usepackage{xcolor}
\usepackage{graphics}
\usepackage{float}
\usepackage[margin=2cm]{geometry}
\usepackage{ulem}
\usepackage{graphicx}
\usepackage{adjustbox}
\usepackage{appendix}
\usepackage{comment}
\usepackage{authblk}
\usepackage{algorithm}
\usepackage{algpseudocode}
\usepackage{multirow}
\usepackage{url}

\usepackage{chngcntr}
\counterwithout{table}{section}

\theoremstyle{definition}
\newtheorem{thm}{Theorem}
\newtheorem{lem}{Lemma}

\newtheorem{cor}{Corollary}

\newtheorem{dfn}{Definition}
\newtheorem{rem}{Remark}
\newtheorem{ex}{Example}

\numberwithin{equation}{section}

\title{Computer-Assisted Search for Differential Equations Corresponding to Optimization Methods and Their Convergence Rates}
\author[1]{Atsushi Tabei}
\author[2]{Ken'ichiro Tanaka}

\affil[1]{Department of Mathematical Informatics, Graduate School of Information Science and Technology, The University of Tokyo, 7-3-1, Hongo, Bunkyo-ku, Tokyo, 113-8656, Japan\\
\texttt{atsushi-tabei2001@g.ecc.u-tokyo.ac.jp}}
\affil[2]{Department of Mathematical and Computing Science, School of Computing, Institute of Science Tokyo, 2-12-1, Ookayama, Meguro-ku, Tokyo, 152-8550, Japan\\
\texttt{kenichiro@comp.isct.ac.jp}}
\date{\today}

\makeatletter
\@removefromreset{table}{section}
\@removefromreset{table}{subsection}
\@removefromreset{table}{subsubsection}

\makeatother

\begin{document}
\maketitle

\begin{abstract}
    Let $f:\mathbb{R}^n \to \mathbb{R}$ be a continuously differentiable convex function with its minimizer denoted by $x_*$ and optimal value $f_* = f(x_*)$. 
    Optimization algorithms such as the gradient descent method can often be interpreted in the continuous-time limit as differential equations known as continuous dynamical systems. 
    Analyzing the convergence rate of $f(x) - f_*$ in such systems often relies on constructing appropriate Lyapunov functions. 
    However, these Lyapunov functions have been designed through heuristic reasoning rather than a systematic framework.
    Several studies have addressed this issue. In particular, Suh, Roh, and Ryu (2022) proposed a constructive approach that involves introducing dilated coordinates and applying integration by parts.
    Although this method significantly improves the process of designing Lyapunov functions, it still involves arbitrary choices among many possible options, and thus retains a heuristic nature in identifying Lyapunov functions that yield the best convergence rates.
    In this study, we propose a systematic framework for exploring these choices computationally. 
    More precisely, we propose a brute-force approach using symbolic computation by computer algebra systems to explore every possibility. 
    By formulating the design of Lyapunov functions for continuous dynamical systems as an optimization problem, we aim to optimize the Lyapunov function itself. 
    As a result, our framework successfully reproduces many previously reported results and, in several cases, discovers new convergence rates that have not been shown in the existing studies.
\end{abstract}

\tableofcontents
\clearpage

\section{Introduction}

We consider the unconstrained convex optimization problem
\begin{align}
    \min_{x \in \mathbb{R}^{n}} f(x),
\end{align}
where $f : \mathbb{R}^{n} \to \mathbb{R}$ is a convex function and a minimizer $x_\ast \in \mathbb{R}^{n}$ exists.  
Optimization algorithms such as the gradient descent method and its variants are the most fundamental and widely used in various fields. 
The goal of numerical optimization is to generate a sequence of iterates
\[
    x_{0}, x_{1}, \ldots, x_{k} \to x_{\ast} \quad (k \to \infty),
\]
that converges to the minimizer. 
In general, the update rule can be expressed in the recursive form
\begin{align}
    x_{k+1} = G(x_{k}, x_{k-1}, \ldots, x_{0}),
\end{align}
where the function $G$ specifies how each new iterate is generated.

The simplest instance of such iterative schemes is the gradient descent method:
\begin{align}
    x_{k+1} = x_{k} - h_{k} \nabla f(x_{k})
    \qquad
    (h_{k} > 0).
    \label{eq:plainGD}
\end{align}
Besides this, a variety of other methods have been proposed.  
In recent years, especially in the field of machine learning, 
first-order methods that rely solely on the gradient $\nabla f$ have attracted great attention 
\cite{drori2014performance, 
diakonikolas2019approximate, 
ito2021nearly, 
kim2016optimized, 
kim2021optimizing, 
luo2022differential}.
This is mainly due to their low computational cost in large-scale settings.  
In contrast, second-order methods make use of the Hessian $\nabla^{2} f$; the Newton method is a representative example.

From a practical standpoint, algorithms that achieve faster convergence 
in terms of $x_{k} \to x_{\ast}$ or $f(x_{k}) \to f(x_{\ast})$ are preferable.  
For smooth convex functions $f$, the simple gradient descent~\eqref{eq:plainGD} only guarantees
\[
    f(x_{k}) - f(x_{\ast}) = \mathrm{O}(1/k)
\]
when the step size $h_{k}$ is appropriately chosen \cite[Theorem 2.1.14]{nesterov2018lectures}. 
To improve the convergence rate while remaining within the class of first-order methods, 
various acceleration techniques have been developed, 
among which Nesterov’s accelerated gradient (NAG) method
\cite{nesterov547701288method}
is a celebrated example. Refer to \cite{d2021acceleration} for details on acceleration techniques.
% Meanwhile, research efforts have also been directed toward accelerating second-order methods
% \textcolor{blue}{[CITE]}.

\subsection{Continuous-Time Modeling of Optimization Algorithms}
\label{sec:cont_model_of_opt}

Given the diversity of optimization algorithms, it is natural to seek unifying principles that guide the design of effective ones. 
In the field of continuous optimization, it has been known for decades that such algorithms can be analyzed through continuous dynamical systems obtained as the limit of those algorithms when the step size tends to $0$ 
\cite{saupe1988discrete,
POLYAK19641}.
Convergence analysis of various optimization methods has been conducted via continuous dynamical systems, and as a result, correspondences between the convergence rates of optimization methods and those of continuous dynamical systems have been established
\cite{
su2016differential, 
wilson2021lyapunov, 
schneider2020dynamical, 
saupe1988discrete, 
wibisono2016variational, 
krichene2015mirror,
sanz2021connections,
fazlyab2018analysis,
diakonikolas2021generalized,
franca2018admm,
defazio2019curved,
luo2022differential,
siegel2019accelerated,
ALVAREZ2002747,
sonntag2024fast,
pmlr-v97-muehlebach19a,
muehlebach2021optimization,
pmlr-v202-kim23y,
NEURIPS2023_c7074114,
ascher2019discrete}. 
This is useful as we can predict the convergence rate of an optimization method from the analysis of a relatively tractable dynamical system. 
Furthermore, we can gain inspiration for new optimization methods by discretizing continuous dynamical systems that exhibit favorable convergence rates. 

More precisely, we can get a fruitful perspective by studying 
continuous dynamical systems described by 
ordinary differential equations (ODEs) 
\begin{equation}
    \label{Base_Form_of_ODE}
    F(x,\dot{x},\ddot{x},\nabla f(x), \nabla^2 f(x) \dot{x}, \ldots)=0
\end{equation}
derived from continuous-time limit of optimization algorithms.  
Because such ODEs often describe the asymptotic behavior of the algorithms, this perspective provides a powerful tool for analyzing their convergence. 
Typical correspondences between the algorithms and ODEs include:
\begin{align*}
    \text{Gradient descent} &\ \leftrightarrow\ \dot{x} + \nabla f(x) = 0, \\
    \text{Nesterov (convex case)} &\ \leftrightarrow\ \ddot{x} + \frac{3}{t}\dot{x} + \nabla f(x) = 0, \\
    \text{Nesterov ($\mu$-strongly convex case)} &\ \leftrightarrow\ \ddot{x} + 2\sqrt{\mu}\dot{x} + \nabla f(x) = 0, \\
    \text{Damped Newton} &\ \leftrightarrow\ \nabla^{2} f(x)\dot{x} + \nabla f(x) = 0.
\end{align*}
For example, the recursive formula~\eqref{eq:plainGD} of the gradient descent can become the ODE 
\(
\dot{x} + \nabla f(x) = 0
\)
of the gradient flow if we regard $h_{k}$ as a time step $\Delta t$ and take the limit of 
\(
(x_{k+1} - x_{k})/h_{k}
\)
as $\Delta t \to 0$. 
Furthermore, we can consider the reverse direction of the correspondence; 
discretizing the ODE can yield the recursive formula. 
Such discretization can be viewed as constructing a numerical integrator for an ODE. 
Therefore we can expect that study of such ODEs helps construction of effective optimization algorithms, 
although the aforementioned correspondence is not always one-to-one or trivial 
\cite{ushiyama2023unified,
shi2022understanding}. 

\subsection{Lyapunov Analysis for ODEs}
\label{sec:LyapAnalODEs}

From this perspective, it becomes appealing to search for ODEs whose solutions $x(t)$ yield rapid convergence of $f(x(t)) - f(x_{\ast})$.  
Once such an ODE is identified, an appropriate time discretization that preserves its convergence rate may lead to the development of accelerated optimization algorithms.  
Hence, exploring ODEs with fast convergence properties is a promising approach.

A common analytical tool for quantifying the convergence rate of an ODE is the Lyapunov function (or energy function)
\cite{su2016differential, 
wilson2021lyapunov, 
attouch2016fast, 
attouch2017asymptotic,
attouch2018combining, 
suh2022continuous, 
TomoyaKamijima2024,
sanz2021connections,
bansal2019potential,
karimi2016unifiedconvergenceboundconjugate,
franca2018admm,
du2022lyapunov,
luo2022differential,
siegel2019accelerated}.  
Here a Lyapunov function refers to a function $\mathcal{E}(t)$ associated with an ODE that is monotonically non-increasing, whose precise definition in this context is given by Definition~\ref{dfn:Lyap} later.
When such a Lyapunov function satisfies
\begin{equation}
    \mathcal{E}(t) \geq \mathrm{e}^{\gamma(t)} (f(x(t))-f(x_*))
    \label{condition:Lyap}
\end{equation}
for $t \geq t_{0}$, 
it follows that
\begin{equation}
    \mathrm{e}^{\gamma(t)} (f(x(t))-f(x_*)) \leq \mathcal{E}(t) \leq \mathcal{E}(t_{0}) = \mathrm{const.}
\end{equation}
for $t \geq t_{0}$. 
Hence the convergence rate of $f(x(t))-f(x_*)$ can be bounded as
\begin{equation}
    f(x(t))-f(x_*) = \mathrm{O} \left( \mathrm{e}^{-\gamma(t)} \right)
    \qquad (t \to \infty).
\end{equation}
Therefore it is crucial to find Lyapunov functions satisfying the condition \eqref{condition:Lyap}, especially those for which $\gamma(t)$ grows as rapidly as possible. 

In many studies, however, 
Lyapunov functions have been discovered in a heuristic or ad-hoc manner. 
To address this issue,
several researchers recently study systematic finding of Lyapunov functions
\cite{suh2022continuous, 
TomoyaKamijima2024, 
moucer2023systematic, 
upadhyaya2025automated, 
upadhyaya2025autolyap,
taylor2018lyapunovfirst,
wibisono2016variational
}. 
In particular, Suh, Roh, and Ryu~(2022)~\cite{suh2022continuous} proposed a method to systematically construct Lyapunov functions based on derivation of conservative quantities.
While their method provides a systematic procedure, it still involves arbitrary choices in the derivation. 
Therefore a certain amount of trial and error by humans is necessary for finding Lyapunov functions giving good convergence rates. 

Lyapunov functions can also be used for finding ODEs that achieve fast convergence rate. Kamijima et al.~(2024)~\cite{TomoyaKamijima2024} considered ODEs with the Hessian $\nabla^{2} f$ and undetermined parameters.
That is, they considered a class of systems encompassing multiple instances and attempted to optimize the ODEs so that they achieve as fast convergence rate as possible. This extension made it possible not only to evaluate the convergence rate of specific dynamical systems but also to determine which choices of parameters yield systems with favorable convergence rates. 
However, as shown in Section~\ref{sec:ext_Kamijima}, the new terms including the Hessian matrix do not improve the convergence rate.

\subsection{Motivation and Research Objective}

To address the issues mentioned in Section~\ref{sec:LyapAnalODEs}, 
we aim to
\begin{enumerate}
    \item automate the discovery of Lyapunov functions, and
    \item identify ODEs whose solutions exhibit as fast convergence as possible.
\end{enumerate}
To this end, we propose a systematic approach to find good Lyapunov functions based on the derivation of conservative quantities by Suh, Roh, and Ryu~(2022)~\cite{suh2022continuous}. Moreover, by incorporating the method of Kamijima et al.~(2024)~\cite{TomoyaKamijima2024}, we also optimize ODEs with undetermined coefficients so that they provide good convergence rates. 
More precisely, 
we propose a brute-force approach using symbolic computation by computer algebra systems to explore every possibility for deriving conservative quantities. 
As a result, we obtain the results summarized in Table~\ref{Summary_of_Results}. Note that $\mu,L,k,r$, and $\alpha$ are constants independent of time $t$.

\begin{rem}
This direction is closely related to recent work such as \texttt{AutoLyap} by Upadhyaya et al.~(2025)~
\cite{upadhyaya2025autolyap},  
which seeks to computationally automate the search for Lyapunov functions and stability certificates. They perform numerical computations on discrete dynamical systems, whereas we perform symbolic operations on continuous dynamical systems in this paper.
\end{rem}

\begin{table}[H]
    \centering
    \caption{Summary of the main results of this paper. The best given parameters for each class of dynamical systems are substituted in the table. }
    \label{Summary_of_Results}
    \renewcommand{\arraystretch}{1.5}
    \setlength{\tabcolsep}{8pt}
    \begin{tabular}{|>{\centering\arraybackslash}m{1.4cm}
                  |>{\centering\arraybackslash}m{4.3cm}
                  |>{\centering\arraybackslash}m{1.5cm}
                  |>{\centering\arraybackslash}m{3.5cm}
                  |>{\centering\arraybackslash}m{3.8cm}|}
    \hline
    Theorem & Continuous dynamical system & Objective function & Convergence rate & Remark \\
    \hline
    \ref{Dumped_Newton} & $\nabla^2 f \dot{x} + \nabla f = 0$ & Convex & $\mathrm{O}\left(\mathrm{e}^{-t}\right)$ & Same as \cite{TomoyaKamijima2024} \\
    \hline
    \ref{First_Order} & $\dot{x} - \frac{1}{L} \nabla^2 f \dot{x} + \nabla f = 0$ & Strongly convex, Smooth & $\mathrm{O}\left(\mathrm{e}^{-\frac{\mu}{1-\frac{\mu}{L}}t}\right)$ & Novel convergence rate \\
    \hline
    \ref{Gradient_Descent} & $\dot{x} + \nabla f = 0$ & Strongly convex & $\mathrm{O}\left(\mathrm{e}^{-2\mu t}\right)$ & Same as \cite{attouch1996dynamical} \\
    \hline
    \ref{SC-NAG} & $\ddot{x} + 2\sqrt{\mu} \dot{x} + \nabla f = 0$ & Strongly convex & $\mathrm{O}\left(\mathrm{e}^{-\sqrt{\mu}t}\right)$ & Same as \cite{TomoyaKamijima2024} \\
    \hline
    \ref{Second_Order} & $\ddot{x} + \sqrt{\mu} \dot{x} + \frac{1}{\sqrt{\mu}} \nabla^2 f \dot{x} + \nabla f = 0$ & Strongly convex & $\mathrm{O}\left(\mathrm{e}^{-\sqrt{\mu}t}\right)$ & Achieves the same convergence rate as SC-NAG using a Hessian-based system \\
    \hline
    \ref{NAG_Convex} & $\ddot{x} + \frac{3}{t} \dot{x} + \nabla f = 0$ & Convex & $\mathrm{O}\left(\frac{1}{t^2} \right)$ & Same as \cite{su2016differential} \\
    \hline
    \ref{NAG_strong_log} & $\ddot{x} + \frac{r}{t} \dot{x} + \nabla f = 0 ~ (r>3)$ & Strongly convex & $\mathrm{O}\left(\frac{1}{t^{\frac{1}{2}r+\frac{1}{2}}} \right)$ & Strictly weaker than \cite{su2016differential} \\
    \hline
    \ref{NAG_strong_exp} & $\ddot{x} + \frac{4(k^2 + \mu)}{k^2} \frac{1}{t} \dot{x} + \nabla f = 0$ & Strongly convex & $t \leq \frac{2(k^2 + \mu)}{k^3}$: $\mathrm{O}\left(\mathrm{e}^{-kt} \right)$ & Guarantees exponential rate for NAG within a restricted range; Novel convergence rate with restart scheme \\
    \hline
    \ref{Generalized_NAG} & $\begin{array}{c}
    \ddot{x} + \frac{r}{t^{\alpha}}\dot{x}+\nabla f = 0 \\
    (r>0,~0<\alpha<1)
    \end{array}$ 
    & Strongly convex 
    & $\begin{array}{c}
    \mathrm{O}\!\left(\mathrm{e}^{-\left(\tfrac{2}{3} - \epsilon \right)\tfrac{r}{1-\alpha} t^{1-\alpha}}\right) \\
    (\epsilon>0)
    \end{array}$ 
    & Improved from \cite{cheng2025class} \\
    \hline
  \end{tabular}
\end{table}

\subsection{Organization of the Paper}

In Section~\ref{sec:ConvRatesODEsForOpt}, 
we present mathematical notation and terminology used in this paper. We then explain the connection between optimization problems and ODEs. This gives the fundamental motivation for this study. We also review some related work on convergence rate analysis. In particular, we detail the derivation of convergence rates using Lyapunov functions, which forms the core of our work. 
In Section~\ref{sec:ConstLyap}, 
we review the approaches of Suh, Roh, Ryu~(2022)~\cite{suh2022continuous} and Kamijima et al.~(2024)~\cite{TomoyaKamijima2024} and discuss their limitations. 

In Section~\ref{sec:our_method}, 
we introduce our proposed method to construct Lyapunov functions 
and show how our method addresses the limitations. We also give several notes on the implementation of symbolic computation to execute our method. 
Section~\ref{sec:Results} 
presents the results obtained by the proposed method. After giving an overview, we discuss each class of continuous dynamical systems presented in Table \ref{Summary_of_Results} in separate subsections. In each case, the system is given in a form that includes tunable constants in its differential equation, meaning that the analysis covers a broader scope than the corresponding result in Table~\ref{Summary_of_Results}.  
Section~\ref{sec:restart}
is on a restart scheme which makes use of one of the continuous dynamical systems found within our methods. 
Lastly, Section~\ref{sec:conclusion} 
concludes the paper by summarizing our findings and discussing future directions.  

\section{Convergence Rates of Continuous Dynamical Systems Corresponding to Optimization Problems}
\label{sec:ConvRatesODEsForOpt}

\subsection{Mathematical Preliminaries}
\subsubsection{Notation}
Let 
$x:\mathbb{R}_{\geq 0} \to \mathbb{R}^n$ and
$f: \mathbb{R}^n \to \mathbb{R}$ 
be twice differentiable functions. 
In addition, let 
$\gamma: \mathbb{R}_{\geq 0} \to \mathbb{R}$ be a function. 
The notation used in this paper is summarized in Table \ref{table_of_results}.
\begin{table}[H]
    \centering
    \caption{Notation used in this paper}
    \label{table_of_results}
    \renewcommand{\arraystretch}{1.5}
    \setlength{\tabcolsep}{8pt}
    \begin{tabular}{|>{\centering\arraybackslash}m{3cm}
                  |>{\centering\arraybackslash}m{6cm}|}
    \hline
    Symbol & Meaning \\
    \hline 
    $\mathbb{R}$ & The set of all real numbers \\
    \hline
    $\mathbb{R}_{\geq 0}$ & The set of all non-negative real numbers \\
    \hline
    $I_n$ & The $n \times n$ identity matrix \\
    \hline
    $O_n$ & The $n \times n$ zero matrix \\
    \hline
    $\| \cdot \|$ & Euclidean norm \\
    \hline
    $\langle \cdot, \cdot \rangle$ & Inner product in Euclidean space \\
    \hline
    $\nabla f $ & Gradient of the function $f$ \\
    \hline
    $\nabla^2 f$ & Hessian matrix of the function $f$ \\
    \hline
    $x_*$ & Optimal solution of the function $f$ \\
    \hline
    $f_*$ & Optimal value of the function $f$ (shorthand for $f(x_*)$) \\
    \hline
    $\dot{x}(t),~\dot{x}(s)$ & $\frac{\mathrm{d}}{\mathrm{d}t}x(t),~\frac{\mathrm{d}}{\mathrm{d}s}x(s)$ \\
    \hline
    $\ddot{x}(t),~\ddot{x}(s)$ & $\frac{\mathrm{d^2}}{\mathrm{d}t^2}x(t),~\frac{\mathrm{d^2}}{\mathrm{d}s^2}x(s)$ \\
    \hline
  \end{tabular}
\end{table}
  
For simplicity, 
we sometimes abbreviate 
$x(t)$, $f(x(t))$, $\nabla f(x(t))$, 
$\nabla^2 f(x(t))$, $\dot{x}(t)$, and $\ddot{x}(t)$
as 
$x$, $f$, $\nabla f$, $\nabla^2 f$, $\dot{x}$ and $\ddot{x}$, respectively, 
whenever it causes no confusion.  
In addition, 
we apply the same abbreviation to 
$x(s)$, $f(x(s))$, $\nabla f(x(s))$, $\nabla^2 f(x(s))$, $\dot{x}(s)$ and $\ddot{x}(s)$. 
Furthermore, we use positive real numbers $\mu$ and $L$
for expressing strong convexity and smoothness of $f$. 
If a function $f$ is both $\mu$-strongly convex and $L$-smooth (definitions given later), then $\mu \leq L$ holds.  
Finally, we use asymptotic order notation as follows.
For functions $g, h:\mathbb{R}_{\geq0}\to\mathbb{R}$, if there exist constants $T,C>0$ such that
\begin{equation}
    t \geq T \implies |g(t)| \leq C|h(t)|,
\end{equation}
then we write
\begin{equation}
    g(t) = \mathrm{O}(h(t)).
\end{equation}

\subsubsection{Terminology}
We introduce terminology used in this paper. We begin with the definition of a Lyapunov function.

\begin{dfn}[Lyapunov function]
\label{dfn:Lyap}
	For a continuous dynamical system governed by an ordinary differential equation, 
    a function $\hat{\mathcal{E}}: \mathbb{R}^{n} \times \mathbb{R}_{\geq 0} \to \mathbb{R}$ is a Lyapunov function  
    if $\hat{\mathcal{E}}(x(t), t)$ is monotonically non-increasing and bounded below for any solution trajectory $x(t)$ of that dynamical system.
    For simplicity, we use the abbreviation
    $\mathcal{E}(t) = \hat{\mathcal{E}}(x(t), t)$. 
    We also refer to this $\mathcal{E}$ as a Lyapunov function, even though this is an abuse of terminology.
\end{dfn}

In this paper, we employ Lyapunov functions with specific properties to guarantee convergence rates.
%, whose details are given in Section~3 of this chapter. 
For a function $f$ and a trajectory $x(t)$, if 
\begin{equation}
    f(x(t)) - f(x_*) = \mathrm{O}(g(t))
\end{equation}
holds, we call $\mathrm{O}(g(t))$ a convergence rate.
That is, we evaluate convergence rates in terms of asymptotic order with respect to $t$. For example, if the trajectory $x(t)$ following a certain dynamical system satisfies
\begin{equation}
    f(x(t)) - f(x_*) = \mathrm{O}\left(\mathrm{e}^{-t}\right),
\end{equation}
we say that the dynamical system achieves a convergence rate of $\mathrm{O}\left(\mathrm{e}^{-t}\right)$.

Next, we introduce classes of functions considered in this paper: smooth functions, convex functions, and strongly convex functions. 
In the following, we present their definitions and properties by following standard references such as Nesterov~(2013)~\cite{nesterov2013introductory} and Nesterov~(2018)~\cite{nesterov2018lectures}.

\begin{dfn}[Smooth function {\cite[Equation (1.2.8)]{nesterov2018lectures}}]
    For $L>0$, a differentiable function $f$ is said to be $L$-smooth if
    \begin{equation}	
	\left \| \nabla f(x) - \nabla f(y) \right \| \leq L \left \| x-y \right\|
    \end{equation}
    holds for all $x,y \in \mathbb{R}^n$.
\end{dfn}

\begin{dfn}[Convex function  {\cite[Definition 2.1.1]{nesterov2013introductory}}]
    A continuously differentiable function $f$ is convex if
    \begin{equation}
	f(y) \geq f(x) + \langle \nabla f(x), y-x \rangle
    \end{equation}
    holds for all $x,y \in \mathbb{R}^n$.
\end{dfn}

\begin{dfn}[Strongly convex function {\cite[Definition 2.1.3]{nesterov2018lectures}}]
    For $\mu >0$, a continuously differentiable function $f$ is $\mu$-strongly convex if
    \begin{equation}
	f(y) \geq f(x) + \langle \nabla f(x), y-x \rangle + \frac{\mu}{2} \left \| y-x \right\|^2
    \end{equation}
    holds for all $x,y \in \mathbb{R}^n$.
\end{dfn}

\begin{thm}[\!{\cite[Lemma 1.2.3]{nesterov2018lectures}}]
    A differentiable function $f$ is $L$-smooth if and only if
    \begin{equation}
	f(y) \leq f(x) + \langle \nabla f(x), y-x \rangle + \frac{L}{2} \left \| y-x \right\|^2
    \end{equation}
    holds for all $x,y \in \mathbb{R}^n$.
\end{thm}

\begin{thm}[\!{\cite[Lemma.1.2.2]{nesterov2018lectures}}]
    If $f$ is twice continuously differentiable, then $f$ is $L$-smooth if and only if
    \begin{equation}
	-L I_n \preceq\nabla^2 f(x) \preceq L I_n
    \end{equation}
    holds for all $x \in \mathbb{R}^n$.
\end{thm}

\begin{thm}[\!{\cite[Theorem 2.1.4]{nesterov2018lectures}}]
    If $f$ is twice continuously differentiable, then $f$ is convex if and only if
    \begin{equation}
	O_n \preceq \nabla^2 f(x)
    \end{equation}
    holds for all $x \in \mathbb{R}^n$.
\end{thm}

\begin{thm}[\!{\cite[Theorem 2.1.10]{nesterov2018lectures}}]
    A differentiable function $f$ is $\mu$-strongly convex if and only if
    \begin{equation}	
	\left \| \nabla f(x) - \nabla f(y) \right \| \geq \mu \left \| x-y \right\|
    \end{equation}
    holds for all $x,y \in \mathbb{R}^n$.
\end{thm}

\begin{thm}[\!{\cite[Theorem 2.1.11]{nesterov2018lectures}}]
    \label{Hessian}
    If $f$ is twice continuously differentiable, then $f$ is $\mu$-strongly convex if and only if
    \begin{equation}
	\mu I_n \preceq \nabla^2 f(x)
    \end{equation}
    holds for all $x \in \mathbb{R}^n$.
\end{thm}

\subsection{Continuous Dynamical Systems Corresponding to Optimization Methods}

As mentioned in Section~\ref{sec:cont_model_of_opt}, 
continuous dynamical systems given by ODEs are useful for analyzing optimization methods. 
In this subsection, we present several examples of such continuous dynamical systems. 

\begin{ex}[Continuous dynamical system corresponding to gradient descent]
    It has long been known that the gradient descent method corresponds to the continuous dynamical system
    \begin{equation}
        \dot{x} + \nabla f = 0
    \end{equation}
    (see, e.g., Attouch and Cominetti~(1996)~\cite{attouch1996dynamical}). Its convergence rate is given as follows; 
    if the objective function $f$ is merely convex, we have
    \begin{equation}
        f - f_* = \mathrm{O}\left( \frac{1}{t} \right)
    \end{equation}
    (Alvarez and Attouch~(2001)~\cite{alvarez2001convergence}),
    while if $f$ is $\mu$-strongly convex, we have
    \begin{equation}
        f - f_* = O\left( \mathrm{e}^{-2\mu t} \right)
    \end{equation}
    (Alvarez and Felipe~(2000)~\cite{alvarez2000minimizing}).
\end{ex}

Here we summarize the correspondence between the gradient descent and its associated continuous dynamical system. With step size $h$, the update rule of the gradient descent is
\begin{equation}
    x_{k+1} = x_k - h \nabla f (x_k).
\end{equation}
Taking the limit $h \to 0$, we obtain the corresponding continuous dynamical system
\begin{equation}
    \dot{x} + \nabla f = 0.
\end{equation}
Thus, optimization methods can be associated with specific continuous dynamical systems by taking the limit as the step size tends to zero. In what follows, we present other examples omitting the details of these correspondences for brevity.

\begin{ex}[Continuous dynamical system corresponding to Nesterov’s accelerated gradient method (NAG)]
    The continuous dynamical system corresponding to Nesterov’s accelerated gradient method (NAG) is
    \begin{equation}
        \ddot{x} + \frac{3}{t} \dot{x} + \nabla f = 0,
    \end{equation}
    whose convergence rate is
    \begin{equation}
        f - f_* = \mathrm{O}\left( \frac{1}{t^2} \right).
    \end{equation}
    See Theorem 7 in Su, Boyd, and Candès~(2016)~\cite{su2016differential} and Section 3.1 in Su, Roh, and Ryu~(2022)~\cite{suh2022continuous}.
\end{ex}

Moreover, an extended form of this dynamical system has also been studied:
\begin{equation}
    \ddot{x} + \frac{r}{t} \dot{x} + \nabla f = 0, 
\end{equation}
where $r \geq 3$. 
For this system, if the objective function $f$ is merely convex, the convergence rate remains
\begin{equation}
    f - f_* = \mathrm{O}\left( \frac{1}{t^2} \right).
\end{equation}
There are several studies focusing on improving the constant term of this rate (Su, Boyd, and Candès~(2016)~\cite{su2016differential}, Attouch et al.~(2018)~\cite{attouch2018fast}). 
If the objective function $f$ is $\mu$-strongly convex, it is shown that
\begin{equation}
    f - f_* = \mathrm{O}\left( \frac{1}{t^{\frac{2}{3}r}} \right).
\end{equation}
See Theorem 8 in Su, Boyd, and Candès~(2016)~\cite{su2016differential}, Theorem 3.3 in Attouch, Chbani, and Riahi~(2019)~\cite{attouch2019rate}, and Section 3.2 in Su, Roh, and Ryu~(2022)~\cite{suh2022continuous}. 

\begin{ex}[Continuous dynamical system corresponding to  Nesterov’s accelerated gradient method for strongly convex functions (SC-NAG)]
    For an $\mu$-strongly convex objective function $f$, the continuous dynamical system corresponding to the strongly convex version of Nesterov’s accelerated gradient method (SC-NAG) is
    \begin{equation}
        \ddot{x} + 2\sqrt{\mu}\dot{x} + \nabla f = 0,
    \end{equation}
    whose convergence rate is
    \begin{equation}
        f - f_* = \mathrm{O}\left( \mathrm{e}^{-\sqrt{\mu} t} \right).
    \end{equation}
    See Proposition 1 in Wilson et al.~(2021)~\cite{wilson2021lyapunov} and Section 3.4 in Su, Roh, and Ryu~(2022)~\cite{suh2022continuous}.
\end{ex}

\begin{ex}[General form of continuous dynamical systems corresponding to Nesterov’s accelerated gradient method]
    Cheng, Liu, and Shang~(2025)~\cite{cheng2025class} proposed 
    a generalized continuous dynamical system 
    \begin{equation}
        \ddot{x} + \frac{r}{t^\alpha} \dot{x} + \nabla f = 0
    \end{equation}
    that includes those corresponding to NAG and SC-NAG, 
    where $r>0$ and $0 \leq \alpha \leq 1$. 
    The case $r=3$ and $\alpha=1$ corresponds to NAG, while $r=2\sqrt{\mu}$ and $\alpha=0$ corresponds to SC-NAG. 
    In the case $0<\alpha<1$, 
    for $\mu$-strongly convex functions $f$, 
    the convergence rate is
    \begin{equation}
        f - f_* = \mathrm{O}\left(\mathrm{e}^{\frac{r}{2(1-\alpha)}t^{1-\alpha}}\right).
    \end{equation}
    See Theorem 1.2 in Cheng, Liu, and Shang~(2025)~\cite{cheng2025class}.
\end{ex}

\begin{ex}[Continuous dynamical system corresponding to the damped Newton method]
    The continuous dynamical system corresponding to the damped Newton method is
    \begin{equation}
        \nabla^2 f \dot{x} + \nabla f = 0.
    \end{equation}
    See Section 2 in Saupe~(1988)~\cite{saupe1988discrete}. For convex objective functions $f$, its convergence rate is
    \begin{equation}
        f - f_* = \mathrm{O}\left( \mathrm{e}^{-t} \right).
    \end{equation}
    See Theorem 8 in Kamijima et al.~(2024)~\cite{TomoyaKamijima2024}.
\end{ex}

The above examples are summarized in Table \ref{List_of_Previous_Results}.

\begin{table}[H]
    \centering
    \caption{Summary of correspondences between optimization methods and continuous dynamical systems with their convergence rates in Examples 1–5.}
    \label{List_of_Previous_Results}
    \renewcommand{\arraystretch}{1.5}
    \setlength{\tabcolsep}{8pt}
    \begin{tabular}{|>{\centering\arraybackslash}m{4cm}
                  |>{\centering\arraybackslash}m{2.6cm}
                  |>{\centering\arraybackslash}m{3.6cm}
                  |>{\centering\arraybackslash}m{4cm}|}
    \hline
    Optimization method & Property of objective function & Corresponding dynamical system & Convergence rate \\
    \hline
    \multirow{2}{*}{Gradient descent} & Convex & \multirow{2}{*}{$\dot{x} + \nabla f = 0$} & $f - f_* = \mathrm{O}\left( \frac{1}{t} \right)$ \\
    \cline{2-2}\cline{4-4}
     & $\mu$-strongly convex & & $f - f_* = \mathrm{O}\left( \mathrm{e}^{-2\mu t} \right)$ \\
    \hline
    NAG & Convex & $\ddot{x} + \frac{3}{t} \dot{x} + \nabla f = 0$ & $f - f_* = \mathrm{O}\left( \frac{1}{t^2} \right)$ \\
    \hline
    \multirow{2}{*}{Extended NAG} & Convex & \multirow{2}{*}{$\ddot{x} + \frac{r}{t} \dot{x} + \nabla f = 0$} & $f - f_* = \mathrm{O}\left( \frac{1}{t^2} \right)$ \\
    \cline{2-2}\cline{4-4}
     & $\mu$-strongly convex & & $f - f_* = \mathrm{O}\left( \frac{1}{t^{\frac{2}{3}r}} \right)$ \\
    \hline
    SC-NAG & $\mu$-strongly convex & $\ddot{x} + 2\sqrt{\mu} \dot{x} + \nabla f = 0$ & $f - f_* = \mathrm{O}\left( \mathrm{e}^{-\sqrt{\mu} t} \right)$ \\
    \hline
    Generalized NAG & $\mu$-strongly convex & $\ddot{x} + \frac{r}{t^\alpha}\dot{x} + \nabla f = 0$ & $f - f_* = \mathrm{O}\left(\mathrm{e}^{\frac{r}{2(1-\alpha)}t^{1-\alpha}}\right)$ \\
    \hline
    Damped Newton method & Convex & $\nabla^2 f \dot{x} + \nabla f = 0$ & $f - f_* = \mathrm{O}\left( \mathrm{e}^{-t} \right)$ \\
    \hline
    \end{tabular}
\end{table}

% \subsection{Deriving Convergence Rates via Lyapunov Functions}

% \textcolor{blue}{
% Lyapunov functions in Definition~\ref{dfn:Lyap}
% are well-known tools for obtaining
% the convergence rates of continuous dynamical systems. 
% For a given continuous dynamical system, suppose that there exists a Lyapunov function $\mathcal{E}(t)$ such that
% \begin{equation}
%     \label{condition:Lyap}
%     \mathcal{E}(t) \geq \mathrm{e}^{\gamma (t)} (f-f_*)
% \end{equation}
% holds for any $t \geq t_0$. Since $\mathcal{E}(t)$ is monotonically non-increasing, we obtain
% \begin{equation}
% 	\mathrm{e}^{\gamma (t)} (f-f_*) \leq \mathcal{E}(t) \leq \mathcal{E}(t_{0})
% 	\label{Lyapunov}
% \end{equation}
% for $t \geq t_{0}$, 
% which implies
% \begin{equation}
% 	f-f_* = \mathrm{O}(\mathrm{e}^{-\gamma(t)}).
% 	\label{order}
% \end{equation}
% However, in conventional approaches, such Lyapunov functions have been constructed in an ad-hoc manner for each individual dynamical system. 
% In this paper, 
% we propose a method to systematically construct such convenient Lyapunov functions by symbolic computation.
% }

\section{Construction of Lyapunov Functions}
\label{sec:ConstLyap}

We use the framework of 
Suh, Roh, and Ryu~\cite{suh2022continuous} 
that provides a method for constructing useful Lyapunov functions via conservation laws. 
Furthermore, 
we incorporate the extension of 
the method by Kamijima et al.~(2024)~\cite{TomoyaKamijima2024}, 
which makes it possible to deal with wider classes of dynamical systems.  
Therefore
we describe the method of 
Suh, Roh, and Ryu in Section~\ref{sec:SuhRohRyu_method}, 
summarize the work of Kamijima et al.\  
in Section~\ref{sec:ext_Kamijima}. 

% and propose our method based on them
% in Section~\ref{sec:our_method}. Then, we introduce the implementation details in Section~\ref{sec:implementation_details}, and summarize our algorithm in Section~\ref{sec:algorithm}.

\subsection{Method for Constructing Lyapunov Functions by Suh, Roh, and Ryu~(2022)~\cite{suh2022continuous}}
\label{sec:SuhRohRyu_method}

Suppose that a continuous dynamical system is given by the ODE in~\eqref{Base_Form_of_ODE}:
\begin{equation}
    \label{Base_Form_of_ODE_re}
    F\left(x(s),\dot{x}(s),\ddot{x}(s),\nabla f(x(s)), \nabla^2 f(x(s)) \dot{x}(s), \ldots \right)=0.
\end{equation}
For this ODE,
we prepare a function 
$\gamma: \mathbb{R}_{\geq 0} \to \mathbb{R}$
that is assumed to provide its convergence rate.
If this function is appropriate, 
we can derive a Lyapunov function by using their method 
that consists of the following four steps:
\begin{enumerate}
    \item For a solution $x$ of the ODE, define
        \begin{equation}
            w(s) = \mathrm{e}^{\gamma(s)} \left(x(s)-x_*\right).
            \label{eq:def_func_w}
        \end{equation}
    \item Take the inner product of both sides of the ODE and
        \begin{equation}
            \dot{w}(s) = \mathrm{e}^{\gamma(s)} \left(\dot{x}(s) + \dot{\gamma}(s) (x(s)-x_*)\right).
                    \label{dot_w}
        \end{equation}
    \item Integrate the above inner product in the left-hand side over the interval $[t_{0}, t]$ $(t \geq t_{0})$, and apply integration by parts to certain terms.
    \item Move only the constant term (the conserved quantity) to the right-hand side. 
    A Lyapunov function can be given by
    the expression obtained by removing the monotone non-decreasing terms from the left-hand side.  
\end{enumerate}

Specifically, if the result of the above procedure can be expressed as
\begin{equation}
	\label{PQ}
    p(t) + \mathrm{e}^{\gamma(t)} \left(f(x(t))-f_*\right)  + \int_{t_0}^t q(s) \, \mathrm{d}s = \mathrm{const.},
\end{equation}
then differentiation with respect to $t$ yields
% \begin{equation}
%     \dot{p}(t) + \frac{\mathrm{d}}{\mathrm{d}t} \left( \mathrm{e}^{\gamma(t)} \left(f(x(t))-f_*\right) \right)+ q(t) = 0,
% \end{equation}
% so that
\[
\dot{p}(t) + \frac{\mathrm{d}}{\mathrm{d}t} \left( \mathrm{e}^{\gamma(t)} \left(f(x(t))-f_*\right) \right) = - q(t).
\]
From this identity, 
if we define
\begin{equation}
	\mathcal{E}(t) = p(t) + \mathrm{e}^{\gamma(t)} \left(f(x(t))-f_*\right),
    \label{eq:def_general_E}
\end{equation}
then we have
\begin{equation}
    	\dot{\mathcal{E}}(t) = \dot{p}(t) + \frac{\mathrm{d}}{\mathrm{d}t} \left( \mathrm{e}^{\gamma(t)} \left(f(x(t))-f_*\right) \right) = - q(t).
        \label{eq:expr_dot_E}
\end{equation}
From~\eqref{eq:def_general_E} and \eqref{eq:expr_dot_E}, 
if 
\begin{align}
p(t) \geq 0 \quad \text{and} \quad q(t) \geq 0
\notag
\end{align}
are satisfied for any $t \geq t_0$, 
then $\mathcal{E}(t)$ is a desired Lyapunov function satisfying condition~\eqref{condition:Lyap}.
% because 
% \begin{gather}
%     \mathcal{E}(t) \geq \mathrm{e}^{\gamma(t)} \left(f(x(t))-f_*\right) \ \Leftrightarrow \ p(t) \geq 0,\\
%     \dot{\mathcal{E}}(t) = -q(t) \leq 0 \ \Leftrightarrow \ q(t) \geq 0,
% \end{gather}

According to Suh, Roh, and Ryu~(2022)~\cite{suh2022continuous}, 
we show a concrete example of the method by using the dynamical system
\begin{equation}
	\dot{x} + \nabla f = 0
\end{equation}
in the case of $\mu$-strongly convex objective function $f$.
Here we suppose that $t_{0}=0$
and omit argument $s$ inside integrals  and argument $t$ outside integrals for brevity.
\begin{enumerate}
    \item \label{item:step_def_w}
    Letting $\gamma(s) = \mu s$, we define
        \begin{equation}
            w(s) = \mathrm{e}^{\mu s} \left(x(s)-x_*\right).
        \end{equation}
    \item \label{item:step_inn_prod}
    Taking the inner product of both sides of the ODE $\dot{x} + \nabla f = 0$ and 
        \begin{equation}
            \label{parameter_dot_w}
            \dot{w}(s) = \mathrm{e}^{\mu s} \left(\mu (x(s)-x_*) + \dot{x}(s)\right),
        \end{equation}
        we have
        \begin{equation}
            \mathrm{e}^{\mu s} \langle \dot{x} + \nabla f , \mu (x - x_*) + \dot{x}  \rangle = 0.
        \end{equation}
    \item \label{item:step_integral}
    Integrating the above inner product over the interval $[0, t]$ $(t\geq 0)$, we have
    	\begin{equation}
		\int_0^t \left( \mathrm{e}^{\mu s} \langle \dot{x} + \nabla f , \mu (x - x_*) + \dot{x}  \rangle \right) \mathrm{d}s = 0, 
	\end{equation}
	which implies
	\begin{equation}
	\label{exampleGD}
	\begin{split}
		\int_0^t &\left( \mu \mathrm{e}^{\mu s} \langle x-x_*, \dot{x} \rangle + \mathrm{e}^{\mu s} \|\dot{x}\|^2 + \mu \mathrm{e}^{\mu s} \langle x-x_*, \nabla f \rangle + \mathrm{e}^{\mu s} \langle \nabla f, \dot{x} \rangle \right) \mathrm{d}s = 0.
	\end{split}
	\end{equation}
	We apply integration by parts
    to the first and fourth terms in 
    the left-hand side.
    For the first term, 
    letting $x_0 = x(0)$, we have
	\begin{equation}
	\begin{split}
		\int_0^t \left( \mu \mathrm{e}^{\mu s} \langle x-x_*, \dot{x} \rangle \right) \mathrm{d}s 
		&= \int_0^t \left( \mathrm{e}^{\mu s} \left( \frac{\mu}{2} \|x-x_*\|^2 \right)^{\prime} \right) \mathrm{d}s \\
            &= \frac{\mu}{2} \mathrm{e}^{\mu t} \| x-x_*\|^2 - \frac{\mu}{2} \|x_0 - x_*\|^2 - \int_0^t \left( \frac{\mu^2}{2} \mathrm{e}^{\mu s} \|x-x_*\|^2 \right)\mathrm{d}s.
	\end{split}
	\end{equation}
	For the fourth term, 
    letting $f_{0} = f(x(0))$, we obtain
	\begin{equation}
	\begin{split}
		\int_0^t \mathrm{e}^{\mu s} \langle \nabla f, \dot{x} \rangle \mathrm{d}s &= \int_0^t \mathrm{e}^{\mu s} \left( f - f_* \right)^{\prime} \mathrm{d}s \\
        &= \mathrm{e}^{\mu t} \left( f - f_* \right) - \left(f_0 - f_* \right) - \int_0^t \mu \mathrm{e}^{\mu s} \left( f - f_* \right) \mathrm{d}s.
	\end{split}
	\end{equation}
    It follows from these and equation~\eqref{exampleGD} that 
	\begin{equation}
    \label{exampleGD2}
	\begin{split}
		&\frac{\mu}{2} \mathrm{e}^{\mu t} \| x-x_*\|^2 - \frac{\mu}{2} \|x_0 - x_*\|^2 - \int_0^t \left( \frac{\mu^2}{2} \mathrm{e}^{\mu s} \|x-x_*\|^2 \right)\mathrm{d}s \\
            & + \mathrm{e}^{\mu t} \left( f - f_* \right) - \left(f_0 - f_* \right) - \int_0^t \mu \mathrm{e}^{\mu s} \left( f - f_* \right) \mathrm{d}s \\
            &+ \int_0^t \left( \mathrm{e}^{\mu s} \|\dot{x}\|^2 + \mu \mathrm{e}^{\mu s} \langle x-x_*, \nabla f \rangle \right)\mathrm{d}s = 0. 
	\end{split}
	\end{equation}
    \item \label{item:step_moving_terms}
    Moving the constant terms in 
    \eqref{exampleGD2} to the right-hand side, 
    we have
	\begin{equation}
	\begin{split}
		&\frac{\mu}{2} \mathrm{e}^{\mu t} \| x-x_*\|^2 + \mathrm{e}^{\mu t} \left( f - f_* \right) + \int_0^t \left( \mathrm{e}^{\mu s} \|\dot{x}\|^2 + \mu \mathrm{e}^{\mu s} \left( f_* - f - \langle \nabla f, x_* - x \rangle - \frac{\mu}{2} \|x-x_*\|^2 \right) \right)\mathrm{d}s \\
            &= \frac{\mu}{2} \|x_0 - x_*\|^2 + \left(f_0 - f_* \right).
	\end{split}
	\end{equation}
	% Thus, we have
	% \begin{equation}
	% \begin{split}
	% 	&\frac{\mu}{2} \mathrm{e}^{\mu t} \| x-x_*\|^2 + \mathrm{e}^{\mu t} \left( f - f_* \right) + \int_0^t \left( \mathrm{e}^{\mu s} \|\dot{x}\|^2 + \mu \mathrm{e}^{\mu s} \left( f_* - f - \langle \nabla f, x_* - x \rangle - \frac{\mu}{2} \|x-x_*\|^2 \right) \right)\mathrm{d}s \\
 %            &= \mathrm{const.}
	% \end{split}
	% \end{equation}
    This equation corresponds 
    to that in \eqref{PQ} with  
    \begin{gather}
    	% \gamma(t) = \mu t, \\
	    p(t) = \frac{\mu}{2} \mathrm{e}^{\mu t} \| x-x_*\|^2, \label{eq:ex_of_func_p} \\
	    q(s) = \mathrm{e}^{\mu s} \|\dot{x}\|^2 + \mu \mathrm{e}^{\mu s} \left( f_* - f - \langle \nabla f, x_* - x \rangle - \frac{\mu}{2} \|x-x_*\|^2 \right). \label{eq:ex_of_func_q}
    \end{gather}
    Trivially $p(t) \geq 0$ holds.  
    The inequality $q(s) \geq 0$ also holds because $f$ is $\mu$-strongly convex. Therefore
    \begin{equation}
    	\mathcal{E}(t) = p(t) + \mathrm{e}^{\gamma(t)} \left(f(x(t))-f_*\right)
    \end{equation}
    is a Lyapunov function satisfying condition~\eqref{condition:Lyap}, and the convergence rate
    \begin{equation}
    	f(x(t)) - f(x_*) = \mathrm{O}\left( \mathrm{e}^{-\mu t} \right)
    \end{equation}
    is derived. 
\end{enumerate}

In Step 3, we apply integration by parts
to the first and fourth terms in 
the left-hand side of \eqref{exampleGD}. 
However, this approach is not necessarily the only option.
For example, 
we can apply integration by parts to the second or third terms.
Among such options, 
we need to choose one deriving non-negative $p(t)$ and $q(t)$ that provide $\gamma(t)$ increasing as fast as possible. 
At present, 
the choice of where to apply integration by parts depends on human judgment 
without any theoretical framework.

\subsection{Extension of the Method by Kamijima et al.~(2024)~\cite{TomoyaKamijima2024}}
\label{sec:ext_Kamijima}

In Kamijima et al.~(2024)~\cite{TomoyaKamijima2024}, the method of Suh, Roh, and Ryu~(2022)~\cite{suh2022continuous} was applied to three classes of continuous dynamical systems:
\begin{gather}
	\nabla^2 f \dot{x} + \nabla f = 0, \\
	\dot{x} + c_2(t) \nabla^2 f \dot{x} + \nabla f = 0, \label{eq:dot_x_c2} \\
	\ddot{x} + c_1(t) \dot{x} + c_2(t) \nabla^2 f \dot{x} + \nabla f = 0. \label{eq:ddot_x_c1_c2}
\end{gather}
Here, we point out two major contributions in Kamijima et al.~(2024)~\cite{TomoyaKamijima2024}.
First, unlike the systems considered in Suh, Roh, and Ryu~(2022)~\cite{suh2022continuous}, these dynamical systems include terms involving the Hessian matrix.  
Second, in the second and third classes of systems, the ordinary differential equations contain time-dependent functions $c_1(t)$ and $c_2(t)$.
As mentioned in Section~\ref{sec:LyapAnalODEs}, 
each of them should be regarded not as a single system but as a class of systems encompassing multiple instances.  
% This extension made it possible not only to evaluate the convergence rate of specific dynamical systems but also to determine which choices of parameters yield systems with favorable convergence rates. 
The results obtained for each class are summarized in Table~\ref{Results_of_Kamijima}. 

\begin{table}[H]
    \centering
    \caption{Summary of results in Kamijima et al.~(2024)~\cite{TomoyaKamijima2024}}
    \label{Results_of_Kamijima}
    \renewcommand{\arraystretch}{1.5}
    \setlength{\tabcolsep}{8pt}
    \begin{tabular}{|>{\centering\arraybackslash}m{5.5cm}
                  |>{\centering\arraybackslash}m{2cm}
                  |>{\centering\arraybackslash}m{4cm}
                  |>{\centering\arraybackslash}m{2.7cm}|}
    \hline
    Continuous dynamical system & Property of objective function & Best parameter setting & Convergence rate \\
    \hline
    $\nabla^2 f \dot{x} + \nabla f = 0$ & Convex & - & $\mathrm{O} \left( \mathrm{e}^{-t} \right)$ \\
    \hline
    \multirow{2}{*}{$\dot{x} + c_2(t) \nabla^2 f \dot{x} + \nabla f = 0$}
      & Convex & $c_2(t) = 0$ & $\mathrm{O} \left( \frac{1}{t} \right)$ \\
    \cline{2-4}
      & $\mu$-strongly convex & $c_2(t) = 0$ & $\mathrm{O} \left( \mathrm{e}^{- \mu t} \right)$ \\
    \hline
    \multirow{2}{*}{$\ddot{x} + c_1(t) \dot{x} + c_2(t) \nabla^2 f \dot{x} + \nabla f = 0$}
      & Convex & $c_1(t) = \tfrac{3}{t}, ~ c_2(t) = 0$ & $\mathrm{O} \left( \tfrac{1}{t^2} \right)$ \\
    \cline{2-4}
      & $\mu$-strongly convex & $c_1(t) = 2 \sqrt{\mu}, ~ c_2(t) = 0$ & $\mathrm{O} \left( \mathrm{e}^{- \sqrt{\mu} t} \right)$ \\
    \hline
    \end{tabular}
\end{table}

As can be seen in Table \ref{Results_of_Kamijima}, the continuous dynamical system models corresponding to gradient descent, Nesterov’s accelerated gradient method, and its strongly convex variant achieve the best convergence rates within their respective classes.  
In addition, the convergence rate of the continuous dynamical system corresponding to the damped Newton method is also provided. Since $c_2(t)=0$ holds for all cases, the new terms including the Hessian matrix in Kamijima et al.~(2024)~\cite{TomoyaKamijima2024} do not improve the convergence rate. 
However, 
it is unclear whether these results are caused by limitations of each class of systems or by the inability to perform an exhaustive search for Lyapunov functions. 

\section{Proposed Method}
\label{sec:our_method}

As pointed out at the end of Section~\ref{sec:SuhRohRyu_method}, 
the framework of Suh, Roh, and Ryu (2022) \cite{suh2022continuous}
involves arbitrariness in the process of applying integration by parts.
In addition, 
we cannot be sure whether the framework can determine the optimal convergence rate of dynamical systems, as mentioned at the end of Section~\ref{sec:ext_Kamijima}. 

To address these issues, we propose a computer-assisted method of searching for Lyapunov functions that provides the fastest possible convergence rate for each class of dynamical systems. 
For this purpose, 
we introduce the following brute-force procedure.
\begin{enumerate}
    \item \label{item:mat_P_Q}
    To automate the judgment of the non-negativity of $p$ and $q$ in~\eqref{PQ}, we represent them as quadratic forms given by symmetric matrices $P$ and $Q$, respectively. 
    In particular, the integral like \eqref{exampleGD} before applying the integration-by-parts operations can be expressed by an initial pair $(P^{0}, Q^{0})$ of the matrices. 
    Then the positive semi-definiteness of $P$ and $Q$ is a sufficient condition for the non-negativity of $p$ and $q$. 
    \item \label{item:oper_int_by_parts}
    Since the integration-by-parts operations update $p$ and $q$, we can describe the rules how each operation updates the corresponding matrices $P$ and $Q$. Therefore we list the rules exhaustively. 
    \item \label{item:oper_conv_ineq}
    Furthermore, we need to apply several inequalities expressing the smoothness or (strong) convexity of $f$, as illustrated by \eqref{eq:ex_of_func_q}.
    In fact, we can describe the application of such inequalities as 
    update rules of the matrices $P$ and $Q$ by introducing some parameters. 
    Therefore we also list such rules exhaustively. 
    \item \label{item:appl_rules}
    We list all possible combinations of the update rules above. Then we 
    apply them to the initial pair $(P^{0}, Q^{0})$ of the matrices $P$ and $Q$ by symbolic computation on a computer.
    Thus we obtain all possible pairs of matrices $P$ and $Q$. 
    \item \label{item:analysis_P_Q}
    For each pairs, we maximize $\gamma(t)$ subject to $P \succeq O$ and $Q \succeq O$. We also do this by symbolic computation on a computer. If the dynamical system contains undetermined parameters 
    like \eqref{eq:dot_x_c2} and \eqref{eq:ddot_x_c1_c2}, 
    they are determined by this optimization. 
    Finally, we choose a pair $(P,Q)$ that achieves the maximum $\gamma(t)$.
\end{enumerate}

In the following, we detail several parts of the above procedure. 
In Section~\ref{sec:mat_rep_P_Q}, 
we show matrix representations of $p$ and $q$ (Step~\ref{item:mat_P_Q}). 
In Section~\ref{sec:initial_P_Q}, 
we show how we determine the initial pair of the matrices (Step~\ref{item:mat_P_Q}).
In Section~\ref{sec:operations_for_P_Q}, 
we list all the operations mentioned in 
Steps~\ref{item:oper_int_by_parts} and~\ref{item:oper_conv_ineq} above. 
In Section~\ref{sec:analysis_of_P_Q}, 
we detail the optimization of Step~\ref{item:analysis_P_Q}.

\subsection{Matrix Representations of $p$ and $q$}
\label{sec:mat_rep_P_Q}

We suppose that the function $p(t)$ 
%constituting the Lyapunov function $\mathcal{E}(t)$ obtained by the proposed method 
is expressed in terms of a quadratic form of the three vectors
\begin{align}
    v_1(t) &= x-x_*, \\
    v_2(t) &= \nabla f, \\
    v_3(t) &= \dot{x}, 
\end{align}
% while its derivative, $\dot{\mathcal{E}}(t) = - q(t)$, 
and the function $q(t)$ is expressed in terms of a quadratic form of the five vectors
\begin{align}
    v_1(t) &= x-x_*, \\
    v_2(t) &= \nabla f, \\
    v_3(t) &= \dot{x}, \\
    v_4(t) &= \nabla^2 f \dot{x}, \\
    v_5(t) &= \ddot{x}.
\end{align}
Accordingly, by introducing a $3\times 3$ matrix $P$ and a $5\times 5$ matrix $Q$, they can be expressed as
\begin{align}
    p(t) &= \mathrm{e}^{\gamma} \begin{pmatrix}
            v_1^T & v_2^T & v_3^T
        \end{pmatrix} (P \otimes I_n) \begin{pmatrix}
            v_1\\
            v_2\\
            v_3
        \end{pmatrix}, \\
    q(t) &= \mathrm{e}^{\gamma} \begin{pmatrix}
            v_1^T & v_2^T & v_3^T & v_4^T & v_5^T
        \end{pmatrix} (Q \otimes I_n) \begin{pmatrix}
            v_1\\
            v_2\\
            v_3\\
            v_4\\
            v_5
        \end{pmatrix},
\end{align}
where $v_i$ denotes $v_i(t)$ for brevity.

The condition for $\mathcal{E}(t)$ to serve as a Lyapunov function satisfying \eqref{condition:Lyap} is that
\begin{gather}
    p(t) = \mathrm{e}^{\gamma}\begin{pmatrix}
            v_1^T & v_2^T & v_3^T
        \end{pmatrix} (P \otimes I_n) \begin{pmatrix}
            v_1\\
            v_2\\
            v_3
        \end{pmatrix} \geq 0, \\
    q(t) = \mathrm{e}^{\gamma}\begin{pmatrix}
            v_1^T & v_2^T & v_3^T & v_4^T & v_5^T
        \end{pmatrix} (Q \otimes I_n) \begin{pmatrix}
            v_1\\
            v_2\\
            v_3\\
            v_4\\
            v_5
        \end{pmatrix} \geq 0
\end{gather}
hold for all $t$.
A sufficient condition for this is
\begin{align}
    P \otimes I_n \succeq O_{3n}, \\
    Q \otimes I_n \succeq O_{5n}.
\end{align}
Since the identity matrix $I_n$ is positive semi-definite, this sufficient condition reduces to
\begin{align}
    P \succeq O_{3}, \\
    Q \succeq O_{5}.
\end{align}

Henceforth, we search for pairs $P$ and $Q$ that are positive semi-definite and provides rapidly growing $\gamma(t)$. Since the pair $(P, Q)$ is updated by the operations described later, we denote the pair after the $k$-th operation as $(P^k, Q^k)$. Each operation is defined to preserve the symmetry of $P^k$ and $Q^k$. 

\subsection{Determination of $P^0$ and $Q^0$}
\label{sec:initial_P_Q}

In this subsection, we describe the procedure for determining the initial matrices $P^0$ and $Q^0$. 
% Following the approach of Suh et al.~\cite{suh2022continuous}, the proposed method first takes the inner product of both sides of the ODE in equation \eqref{Base_Form_of_ODE} with the parameter $\dot{w}$ in equation \eqref{parameter_dot_w} and integrates it over the interval $[t_0, t]$ ($t\geq 0$). 
Since the left-hand side of the differential equation \eqref{Base_Form_of_ODE_re}
and the function $\dot{w}$ in \eqref{dot_w}
can be represented as linear combinations of the vectors $v_1,v_2,v_3,v_4,v_5$, they can be expressed using symmetric matrices $P^0$ and $Q^0$. At this stage, no integration by parts has been applied, and the entire inner product remains inside the integral in Step~\ref{item:step_integral} in Section~\ref{sec:SuhRohRyu_method}. 
Thus the initial matrix $P^0$ is the zero matrix.  

As an example, consider the first-order dynamical system
\begin{equation}
    \dot{x} + b\nabla^2 f\dot{x} + \nabla f = 0,
\end{equation}
and let us determine the initial matrices $P^0$ and $Q^0$.
Taking
\(
    \dot{w} = \mathrm{e}^{\gamma} \left( \dot{x} + \dot{\gamma} (x-x_*) \right)
\)
and forming the inner product with both sides of the above dynamical system, 
we have 
\begin{equation}
    \mathrm{e}^{\gamma} \left\langle \dot{x} + b\nabla^2 f\dot{x} + \nabla f, \dot{x} + \dot{\gamma} (x-x_*) \right\rangle = 0.
\end{equation}
Expanding this expression gives
\begin{equation}
    \mathrm{e}^{\gamma} \left( \|\dot{x}\|^2 + \dot{\gamma} \langle x-x_*, \dot{x} \rangle + b \langle \nabla^2 f \dot{x}, \dot{x} \rangle + b \dot{\gamma} \langle \nabla^2 f \dot{x}, x-x_* \rangle + \langle \nabla f , \dot{x} \rangle + \dot{\gamma} \langle \nabla f, x-x_* \rangle \right) = 0,
\end{equation}
whose left-hand side is written in the form
\begin{equation}
    \mathrm{e}^{\gamma} \begin{pmatrix}
            v_1^T & v_2^T & v_3^T & v_4^T & v_5^T
        \end{pmatrix} 
        \left( \frac{1}{2}
        \begin{pmatrix}
            0 & \dot{\gamma} & \dot{\gamma} & b\dot{\gamma} & 0 \\
            \dot{\gamma} & 0 & 1 & 0 & 0 \\
            \dot{\gamma} & 1 & 2 & b & 0 \\
            b\dot{\gamma} & 0 & b & 0 & 0 \\
            0 & 0 & 0 & 0 & 0
        \end{pmatrix} \otimes I_n \right) 
        \begin{pmatrix}
            v_1\\
            v_2\\
            v_3\\
            v_4\\
            v_5
        \end{pmatrix}.
\end{equation}
Accordingly, the initial pair of the matrices for this dynamical system is given by
\begin{equation}
    P^0 = O_3, \quad Q^0 = \frac{1}{2}
        \begin{pmatrix}
            0 & \dot{\gamma} & \dot{\gamma} & b\dot{\gamma} & 0 \\
            \dot{\gamma} & 0 & 1 & 0 & 0 \\
            \dot{\gamma} & 1 & 2 & b & 0 \\
            b\dot{\gamma} & 0 & b & 0 & 0 \\
            0 & 0 & 0 & 0 & 0
        \end{pmatrix}.
\end{equation}

\subsection{Operations on Matrices $P$ and $Q$ and Their Combinations}
\label{sec:operations_for_P_Q}

% The matrices $P^0$ and $Q^0$ are further modified by operations such as integration by parts to build suitable Lyapunov functions. 
% However, there is a degree of freedom in choosing which terms to integrate and in what order. 
% Similarly, Kamijima et al.~(2024)~\cite{TomoyaKamijima2024} proposed operations that exploit properties of the objective function $f(x)$
% but again there is freedom in where these properties are used.  

We systematically track how the pair $(P,Q)$ is changed by the operations mentioned in Steps~\ref{item:oper_int_by_parts} and~\ref{item:oper_conv_ineq} at the beginning of Section~\ref{sec:our_method}.
They eliminate specific off-diagonal entries of $Q$ and change others of $P$ and $Q$.
Indeed, there are 13 operations classified into the following two types.
\begin{itemize}
    \item Operations of integration by parts (Step~\ref{item:oper_int_by_parts}). 
    They consist of Operations 
    A1, B1--B3, C1, and D1--D4, 
    which are classified into Groups A, B, C, and D. 
    Operation A1 is used just once at the beginning for extracting the term $\mathrm{e}^{\gamma} (f - f_{\ast})$ as explained below.
    The operations in Groups B, C, and D are related to the specific entries of $P$ and $Q$ shown in the following expressions, where the underline $\uline{\quad }$ indicates elements that will be eliminated by a certain operation in each group.
    \begin{itemize}
        \item Group B:
            \begin{equation}
    P= \begin{pmatrix}
        P_{11} & \phantom{P_{11}} & P_{13} \\
         &  & \\
        P_{13} & & P_{33} 
    \end{pmatrix}
    ,\quad 
    Q=
    \begin{pmatrix}
        Q_{11} & \phantom{Q_{11}} & \uline{Q_{13}} & \phantom{Q_{11}} & \uline{Q_{15}} \\
         & & & & \\
        \uline{Q_{13}} & & Q_{33} & & \uline{Q_{35}} \\
         & & & & \\
        \uline{Q_{15}} & & \uline{Q_{35}} & &  
    \end{pmatrix}.
    \label{eq:comp_of_P_Q_by_B}
    \end{equation}

        \item Group C:
            \begin{equation}
    P= \begin{pmatrix}
         \phantom{P_{11}} & \phantom{P_{11}} & \phantom{P_{11}} \\
         & P_{22} &  \\
         & &  
    \end{pmatrix}
    ,\quad 
    Q=
    \begin{pmatrix}
         \phantom{Q_{11}} & \phantom{Q_{11}} & \phantom{Q_{11}} & \phantom{Q_{11}} & \phantom{Q_{11}} \\
         & Q_{22} &  & \uline{Q_{24}} &  \\
         & & & & \\
         & \uline{Q_{24}} &  & & \\
         & & & & 
    \end{pmatrix}.
    \label{eq:comp_of_P_Q_by_C}
    \end{equation}

        \item Group D:
    \begin{equation}
    P= \begin{pmatrix}
         & P_{12} & \\
        P_{12} & & P_{23} \\
         & P_{23} & 
    \end{pmatrix}
    ,\quad 
    Q=
    \begin{pmatrix}
         & Q_{12} & & \uline{Q_{14}} & \\
        Q_{12} & & \uline{Q_{23}} & & \uline{Q_{25}} \\
         & \uline{Q_{23}} & & \uline{Q_{34}} & \\
        \uline{Q_{14}} & & \uline{Q_{34}} & & \\
         & \uline{Q_{25}} & & &  
    \end{pmatrix}.
    \label{eq:comp_of_P_Q_by_D}
    \end{equation}
    \end{itemize}
    
    \item Operations of applying inequalities for smoothness and convexity (Step~\ref{item:oper_conv_ineq}). 
    They consist of Operations E1 and F1.
    Operation E1 is related to inequalities for $\nabla f$ and Operation F1 is related to those for $\nabla^{2}f$.
    The operations in Groups E and F are related to the specific entries of $P$ and $Q$ shown in the following expressions, where the underline $\uline{\quad }$ indicates elements that will be eliminated by a certain operation in each group.
    \begin{itemize}
        \item Group E:
    \begin{equation}
    P= \begin{pmatrix}
         P_{11} & &  \\
         \phantom{P_{11}} & \phantom{P_{22}} & \phantom{P_{11}} \\
         & &  
    \end{pmatrix}
    ,\quad 
    Q=
    \begin{pmatrix}
        Q_{11} &  &  & \uline{Q_{14}} & \\
         \phantom{Q_{11}} & \phantom{Q_{22}} & \phantom{Q_{11}} & \phantom{Q_{11}} & \phantom{Q_{11}} \\
         & & & & \\
         \uline{Q_{14}} & &  & & \\
         & & & & 
    \end{pmatrix}.
    \label{eq:comp_of_P_Q_by_E}
    \end{equation}

    \item Group F:
        \begin{equation}
    P= \begin{pmatrix}
          & &  \\
         \phantom{P_{11}} & \phantom{P_{11}} & \phantom{P_{11}} \\
         & &  
    \end{pmatrix}
    ,\quad 
    Q=
    \begin{pmatrix}
        \phantom{Q_{11}} & \phantom{Q_{11}} & \phantom{Q_{11}} & \phantom{Q_{11}} & \phantom{Q_{11}} \\
         & & & & \\
         & & Q_{33} & \uline{Q_{34}} & \\
         & & \uline{Q_{34}} & & \\
         & & & & 
    \end{pmatrix}.
    \label{eq:comp_of_P_Q_by_F}
    \end{equation}

    \end{itemize}
\end{itemize}

\begin{rem}
We note that the underlined entries become zero due to a certain operation, but they may become non-zero again due to subsequent other operations.
For example, 
Operation D3 eliminates $Q_{23}$ and adds a certain term to $Q_{14}$, whereas 
Operation E1 eliminates $Q_{14}$ and adds a certain term to $Q_{11}$. 
Therefore applying Operations E1 and D3 in this order does not necessarily result in a zero value for $Q_{14}$. 
\end{rem}

In this subsection, 
for readability, 
we describe only three of the operations and how they can be systematically combined. 
We provide detailed discussions on all the operations and their combinations in Appendices~\ref{sec:list_of_oper} and~\ref{sec:list_of_combi}. 
In the following, 
the $(i,j)$ entry of $P^{k}$ and $Q^{k}$ are denoted as $P_{ij}^k$ and $Q_{ij}^k$, respectively.
In addition, 
we use ``const." to represent constants appearing in integration by parts throughout this paper.

% \bigskip
% \noindent
\subsubsection{Operations of Integration by Parts}
% \bigskip

% First, we illustrate matrix transformations induced by integration by parts, focusing on basic operations for constructing the desired Lyapunov function. 
We show Operation A1 that creates
\begin{equation}
    \mathrm{e}^\gamma (f-f_*)
\end{equation}
outside of the integral. This operation is always performed first, so it is applied to $(P^0,Q^0)$ to obtain $(P^1,Q^1)$.

\begin{itemize}
    \item Operation A1: Extracting $\mathrm{e}^\gamma(f-f_*)$
        \begin{equation}
            \begin{split}
                &\int_{t_0}^t \mathrm{e}^\gamma \langle v_2, v_3 \rangle \mathrm{d}s \\
                &= \int_{t_0}^t \mathrm{e}^\gamma \langle \nabla f, \dot{x} \rangle \mathrm{d}s \\
                &= \int_{t_0}^t \mathrm{e}^\gamma (f-f_*)^{\prime} \mathrm{d}s \\
                &= \mathrm{e}^\gamma (f-f_*) - \int_{t_0}^t \mathrm{e}^\gamma (f - f_*) \mathrm{d}s + \mathrm{const.}\\
                &= \mathrm{e}^\gamma(f-f_*) + \int_{t_0}^t \mathrm{e}^\gamma \cdot \dot{\gamma} (f_* - f - \langle \nabla f , x_* - x \rangle) \mathrm{d}s\\
                &\phantom{=} - \int_{t_0}^t \mathrm{e}^\gamma \cdot \dot{\gamma} \langle \nabla f, x-x_* \rangle \mathrm{d}s + \mathrm{const.}
            \end{split}
        \end{equation}
        Here, when $f(x)$ is ($\mu$-strongly) convex and $L$-smooth, introducing a time-dependent parameter $\lambda \in [\mu, L]$ gives
        \begin{equation}
            f_* - f - \langle \nabla f, x_* - x \rangle = \frac{\lambda}{2} \|x-x_*\|^2,
        \end{equation}
        which leads to
        \begin{equation}
            \begin{split}
                &\int_{t_0}^t \mathrm{e}^\gamma \langle v_2, v_3 \rangle \mathrm{d}s \\
                &= \mathrm{e}^\gamma (f-f_*) + \int_{t_0}^t \mathrm{e}^\gamma \cdot \frac{\lambda}{2} \dot{\gamma} \|v_1\|^2 \mathrm{d}s - \int_{t_0}^t \mathrm{e}^\gamma \cdot \dot{\gamma} \langle v_1, v_2 \rangle \mathrm{d}s.
            \end{split}
        \end{equation}
        Hence, applying Operation A1 to $(P^0,Q^0)$ yields
        \begin{gather}
            P^1 = P^0, \\
            Q^1 = Q^0 - \frac{1}{2} \begin{pmatrix}
                - \lambda \dot{\gamma} & \dot{\gamma}  & 0 & 0 & 0 \\
                \dot{\gamma} & 0 & 1 & 0 & 0 \\
                0 & 1 & 0 & 0 & 0 \\
                0 & 0 & 0 & 0 & 0 \\
                0 & 0 & 0 & 0 & 0
            \end{pmatrix}.
        \end{gather}
\end{itemize}

Similarly, 
eight other integration-by-parts operations can be performed: B1, B2, B3, C1, D1, D2, D3, and D4. 
Unlike Operation A1, 
all of these operations are expressed only by the entries of $P$ and $Q$ and can be applied multiple times. 
Here we show operation B1:

\begin{itemize}
    \item Operation B1: Integration by parts of $\langle v_3,v_5 \rangle$
        \begin{equation}
            \begin{split}
                &\int_{t_0}^t \mathrm{e}^\gamma \cdot 2 Q_{35}^k \langle v_3, v_5 \rangle \mathrm{d}s \\
                &= \int_{t_0}^t \mathrm{e}^\gamma \cdot Q_{35}^k (2\langle \dot{x}, \ddot{x} \rangle) \mathrm{d}s \\
                &= \int_{t_0}^t \mathrm{e}^\gamma \cdot Q_{35}^k ( \|\dot{x}\|^2)^{\prime} \mathrm{d}s \\
                &= \mathrm{e}^\gamma \cdot Q_{35}^k \|\dot{x}\|^2 - \int_{t_0}^t \mathrm{e}^\gamma \cdot (\dot{\gamma} Q_{35}^k + \dot{Q_{35}^k}) \|\dot{x}\|^2 \mathrm{d}s + \mathrm{const.} \\
                &= \mathrm{e}^\gamma \cdot Q_{35}^k \|v_3\|^2 - \int_{t_0}^t \mathrm{e}^\gamma \cdot (\dot{\gamma} Q_{35}^k + \dot{Q_{35}^k}) \|v_3\|^2 \mathrm{d}s + \mathrm{const.},
            \end{split}
        \end{equation}
        resulting in
        \begin{gather}
            P^{k+1} = P^k + \begin{pmatrix}
                0 & 0 & 0 \\
                0 & 0 & 0 \\
                0 & 0 & Q_{35}^k
            \end{pmatrix}, \\
            Q^{k+1} = Q^k - \begin{pmatrix}
                0 & 0 & 0 & 0 & 0 \\
                0 & 0 & 0 & 0 & 0 \\
                0 & 0 & \dot{\gamma} Q_{35}^k + \dot{Q_{35}^k} & 0 & Q_{35}^k \\
                0 & 0 & 0 & 0 & 0 \\
                0 & 0 & Q_{35}^k & 0 & 0
            \end{pmatrix}.
        \end{gather}
\end{itemize}

In this way, 
Operation B1 depends on the $(3,5)$-entry $Q_{35}^k$ of $Q^k$. 
Once operation B1 is applied, $Q_{35}^{k+1} = 0$ holds. 
Therefore, 
Operation B1 is an idempotent operation. 
As shown later, 
the other operations in Group B are also idempotent and 
this property becomes a key when considering combinations of operations. The remaining seven operations based on integration by parts are described in 
Appendices 
\ref{sec:oper_Group_B}, 
\ref{sec:oper_Group_C}, and
\ref{sec:oper_Group_D}.

% \bigskip
% \noindent
\subsubsection{Operations of Applying Inequalities for Smoothness and Convexity}
% \bigskip

We show operations of applying inequalities for smoothness and convexity of the objective function $f$, 
which are introduced by Kamijima et al.~(2024)~\cite{TomoyaKamijima2024}. 
Here, 
we show a representative operation, Operation E1.

\begin{itemize}
    \item Operation E1: Transfers $\langle v_1, v_4 \rangle$ to $\|x-x_*\|^2$ using the strong convexity and smoothness properties
        \begin{equation}
            \begin{split}
                &\int_{t_0}^t \mathrm{e}^\gamma \cdot 2Q_{14}^k \langle v_1, v_4 \rangle \, \mathrm{d}s \\
                &= \int_{t_0}^t \mathrm{e}^\gamma \cdot 2Q_{14}^k \langle x-x_*, \nabla^2 f \dot{x} \rangle \, \mathrm{d}s \\
                &= \int_{t_0}^t \mathrm{e}^\gamma \cdot 2Q_{14}^k (\langle x-x_*, \nabla f \rangle)^\prime \, \mathrm{d}s - \int_{t_0}^t \mathrm{e}^\gamma \cdot 2Q_{14}^k \langle \nabla f, \dot{x} \rangle \, \mathrm{d}s \\
                &= \mathrm{e}^\gamma \cdot 2Q_{14}^k \langle x-x_*, \nabla f \rangle - \int_{t_0}^t \mathrm{e}^\gamma \cdot 2(\dot{\gamma}Q_{14}^k + \dot{Q_{14}^k}) \langle x-x_*, \nabla f \rangle \, \mathrm{d}s \\
                &\phantom{=} - \int_{t_0}^t \mathrm{e}^\gamma \cdot 2Q_{14}^k (f-f_*)^\prime \, \mathrm{d}s + \mathrm{const.}\\
                &= \mathrm{e}^\gamma \cdot 2Q_{14}^k \langle x-x_*, \nabla f \rangle - \int_{t_0}^t \mathrm{e}^\gamma \cdot 2(\dot{\gamma}Q_{14}^k + \dot{Q_{14}^k}) \langle x-x_*, \nabla f \rangle \, \mathrm{d}s \\
                &\phantom{=} - \mathrm{e}^\gamma \cdot 2Q_{14}^k (f-f_*) + \int_{t_0}^t \mathrm{e}^\gamma \cdot 2(\dot{\gamma}Q_{14}^k + \dot{Q_{14}^k}) (f-f_*) \, \mathrm{d}s + \mathrm{const.}\\
                &= \mathrm{e}^\gamma \cdot 2Q_{14}^k (f_* - f - \langle \nabla f, x_*-x \rangle) \\
                &\phantom{=} - \int_{t_0}^t \mathrm{e}^\gamma \cdot 2(\dot{\gamma}Q_{14}^k + \dot{Q_{14}^k}) (f_* - f - \langle \nabla f, x_*-x \rangle) \, \mathrm{d}s + \mathrm{const.}.
            \end{split}
        \end{equation}

Here, 
if the objective function $f(x)$ is ($\mu$-strongly) convex and $L$-smooth, 
we use the parameter $\lambda \in [\mu, L]$ which is the same parameter that appears in Operation A1:
\begin{equation}
    f_* - f - \langle \nabla f, x_* - x \rangle = \frac{\lambda}{2} \|x-x_*\|^2.
\end{equation}
Then we have
\begin{equation}
    \begin{split}
        &\int_{t_0}^t \mathrm{e}^\gamma \cdot 2Q_{14}^k \langle v_1, v_4 \rangle \, \mathrm{d}s \\
        &= \mathrm{e}^\gamma \cdot \lambda Q_{14}^k \|x-x_*\|^2 - \int_{t_0}^t \mathrm{e}^\gamma \cdot \lambda (\dot{\gamma}Q_{14}^k + \dot{Q_{14}^k}) \|x-x_*\|^2 \, \mathrm{d}s + \mathrm{const.}\\
        &= \mathrm{e}^\gamma \cdot \lambda Q_{14}^k \|v_1\|^2 - \int_{t_0}^t \mathrm{e}^\gamma \cdot \lambda (\dot{\gamma}Q_{14}^k + \dot{Q_{14}^k}) \|v_1\|^2 \, \mathrm{d}s+ \mathrm{const.},
    \end{split}
\end{equation}
which results in
\begin{gather}
    P^{k+1} = P^k + \begin{pmatrix}
        \lambda Q_{14}^k & 0 & 0 \\
        0 & 0 & 0 \\
        0 & 0 & 0
    \end{pmatrix},\\
    Q^{k+1} = Q^k - \begin{pmatrix}
        \lambda (\dot{\gamma}Q_{14}^k+\dot{Q_{14}^k}) & 0 & 0 & Q_{14}^k & 0 \\
        0 & 0 & 0 & 0 & 0 \\
        0 & 0 & 0 & 0 & 0 \\
        Q_{14}^k & 0 & 0 & 0 & 0 \\
        0 & 0 & 0 & 0 & 0
    \end{pmatrix}.
\end{gather}
\end{itemize}

The other operations in this category are described in Sections~\ref{sec:oper_Group_E} and~\ref{sec:oper_Group_F}. 
Their notable feature is that they involve parameters (like $\lambda$) that vary within the range $[\mu, L]$. 
Indeed, we use parameter $\theta$ in the range $[\mu, L]$ that satisfy
\begin{equation}
    \langle \nabla^2 f \dot{x}, \dot{x} \rangle = \theta \|\dot{x}\|^2
\end{equation}
for the other operations in this category 
in the following manner.
\begin{itemize}
\item Operation E1: the parameter $\lambda \in [\mu, L]$
\item Operation F1: the parameter $\theta \in [\mu, L]$
\end{itemize}
We discuss how to deal with these parameters in Section~\ref{sec:param_in_mu_L}. 

% \bigskip
% \noindent
\subsubsection{Combination of the Operations}
% \bigskip

% Combining the above, there are a total of 13 operations that involve either integration by parts or the use of the objective function’s properties. The details of each operation are provided in the appendix. 

We aim to find pairs of positive semi-definite matrices $P$ and $Q$ among those generated by combinations of the above 13 operations. 
% Suh, Roh, and Ryu~(2022)~\cite{suh2022continuous} and Kamijima et al.~(2024)~\cite{TomoyaKamijima2024} performed this process heuristically. 
% In contrast, 
To this end, 
we derive and examine all possible pairs of matrices $P$ and $Q$ that can be generated by combining the 13 operations in arbitrary order.
At first glance, it may seem that the number of their possible combinations is infinite, since each of the 13 operations besides Operation A1 can be applied any number of times and in any order, and the operations are not necessarily commutative. 
In reality, however, the number of such possibilities is finite. 
We outline the overall idea to show this fact below and detail it in Appendix~\ref{sec:list_of_combi}.

% First, the 13 operations can be broadly classified into six groups: groups A to F.

First, we consider possible combinations of the operations within each group. As previously explained, Operation A1 is always applied exactly once at the beginning. Therefore there is only one possible case for Group A. 
The operations in Group B are idempotent as mentioned earlier, and they set specific entries of the matrix $Q$ to zero. 
Therefore possible combinations are limited to the following ten cases:
\begin{enumerate}
\item Do nothing
\item B1
\item B2
\item B1→B2
\item B3
\item B1→B3
\item B2→B3
\item B1→B2→B3
\item B3→B2
\item B1→B3→B2
\end{enumerate}
By investigating possible combinations of the operations in this way for each group, we can conclude that their numbers are finite as summarized below:
\begin{enumerate}
\item Group A: 1 combination
\item Group B: 10 combinations
\item Group C: 2 combinations
\item Group D: 13 combinations
\item Group E: 2 combinations
\item Group F: 2 combinations
\end{enumerate}

Next, 
we consider combinations of the operations across different groups. 
Since Operation A1 is applied only once at the beginning, we do not need to consider Group A further. 
The operations in Groups B, C, and D act on the entries of $P$ and $Q$ different from those of the other group as shown in~\eqref{eq:comp_of_P_Q_by_B}, \eqref{eq:comp_of_P_Q_by_C}, and~\eqref{eq:comp_of_P_Q_by_D}.
Therefore they commute with all the operations in the other group. 
Furthermore, 
as can be seen from~\eqref{eq:comp_of_P_Q_by_B}, \eqref{eq:comp_of_P_Q_by_C}, \eqref{eq:comp_of_P_Q_by_E}, and \eqref{eq:comp_of_P_Q_by_F}, 
the entries eliminated by Groups B and C are different from those eliminated by Groups E and F.
Accordingly, 
it suffices to consider combinations among the operations in Groups D, E, and F. 
Since the operations in Groups E and F are also commutative, it is enough to consider their relationships with Group D. 

The operations in Group D are not necessarily commutative with those in Groups E and F. However, as shown in Appendix~\ref{sec:comb_acoss_groups}, it suffices to consider the case that we apply the operations in Group D and those of Groups E and F in order, and then apply those of Group D again. 

In summary, 
all possible combinations of the operations are given as follows:
\begin{itemize}
\item Group A: 1 combination
\item Group B: 10 combinations
\item Group C: 2 combinations
\item First application of Group D: 13 combinations
\item Groups E, F, and second application of Group D: 7 combinations
\item Third application of Group D: 13 combinations
\end{itemize}
Since each stage is independent, the total number of the possible combinations is $1 \times 10 \times 2 \times 13 \times 7 \times 13 = 23660$, which is finite. We then have only to perform the semi-definite analysis on these at most 23660 possible pairs of matrices $P$ and $Q$.

\subsection{Analysis of the Obtained Matrix Pairs}
\label{sec:analysis_of_P_Q}

For the obtained matrix pairs $(P,Q)$, we consider the positive semidefinite condition. When this semidefiniteness is guaranteed, the convergence rate derived via a Lyapunov function can be expressed as
\begin{equation}
    f(x) - f(x_*) = \mathrm{O}\left(\mathrm{e}^{-\gamma (t)}\right).
\end{equation}
Since faster convergence is preferable, a faster growth of $\gamma(t)$ is desirable. From this observation, the key question becomes: 

\begin{center}
    which $\gamma(t)$ exhibits the fastest growth among those for which $P$ and $Q$ are positive semidefinite?
\end{center}

To address this problem, we maximize $\gamma$ subject to the semi-definite constraints for each of the 23660 matrix pairs $(P,Q)$, and identify one yielding the fastest growth of $\gamma$. 
Hence, we consider the optimization problem
\begin{equation}
    \mathop{\text{``max''}} \ \gamma(t) \quad \mathrm{s.t.} \quad P \succeq O_3, \quad Q \succeq O_5,
    \label{eq:max_gamma_st_P_Q_PSD}
\end{equation}
where ``max'' indicates the fastest function. 
The entries of the matrices $P$ and $Q$ under consideration include:
\begin{itemize}
    \item the time variable $t$,
    \item functions such as $c_{1}(t)$ and $c_{2}(t)$ appearing in the continuous-time dynamical systems in~\eqref{eq:dot_x_c2} and~\eqref{eq:ddot_x_c1_c2},
    \item parameters $\lambda, \theta \in [\mu, L]$ introduced by the operations E1, F1, and
    \item the objective function $\gamma(t)$.
\end{itemize}
Here we discuss how these components are treated. Regarding the time $t$, we need to guarantee positive semi-definiteness of $P$ and $Q$ for all $t>0$ (or for all $t\ge T$ for some sufficiently large $T>0$). Next, 
in the same manner as the approach of Kamijima et al.~(2024)~\cite{TomoyaKamijima2024}, 
the functions $c_{1}(t)$ and $c_{2}(t)$ are also subject to optimization and can be freely selected.
Finally, 
the constraints in \eqref{eq:max_gamma_st_P_Q_PSD} 
must be satisfied for any parameters $\lambda, \theta \in [\mu, L]$. 
Therefore the problem in \eqref{eq:max_gamma_st_P_Q_PSD} is refined as follows: 
\begin{equation}
    \mathop{\text{``max''}}_{c_{1}(t), c_{2}(t)} \ \gamma(t) 
    \quad \mathrm{s.t.} \quad 
    \forall t>0,~ \forall (\lambda, \theta) \in [\mu,L]^2,~ P \succeq O_3, \ Q \succeq O_5.
    \label{eq:max_gamma_st_P_Q_PSD_refined}
\end{equation}
If this problem is feasible, 
its solution gives the fastest convergence rate achievable by the Lyapunov function corresponding to a given matrix pair $(P,Q)$. 
By evaluating the rate for each possible $(P,Q)$ pair, the overall fastest convergence rate can be determined.

\subsection{Implementation Details}
\label{sec:implementation_details}

We use Wolfram Mathematica to perform the operations for the matrices and derive the convergence rates by symbolic computation. 
For implementation,
we need three considerations described below. 

\subsubsection{Implementation of ``$\max$''}
First, we discuss the implementation of the ``$\max$'' operation in \eqref{eq:max_gamma_st_P_Q_PSD_refined} finding the fastest-growing $\gamma(t)$. For example, if there are candidates $\gamma(t) = \log t$ and $\gamma(t) = t$, the latter should be selected. However, to the best of our knowledge, there is no command in Wolfram Mathematica that can directly produce such an output. Therefore, we assume a specific form for $\gamma(t)$ and maximize its coefficient. 
Based on previous studies, 
we assume the plural forms
\[
\gamma(t) = kt, \quad 
\gamma(t) = k\log t, \quad 
\gamma(t) = k \frac{r}{1-\alpha}t^{1-\alpha}  
\]
for a continuous-time dynamical system and maximize the parameter $k$, 
where the last is considered for the dynamical system for the generalized NAG with a fixed parameter $\alpha$ discussed in Section~\ref{sec:generalized_NAG}. 
Similarly, for the functions $c_{1}(t)$ and $c_{2}(t)$ that are free parameters for maximization, we assume specific forms using parameters $a,b, r \in \mathbb{R}$ and regard them as free parameters in $\mathbb{R}$.

\subsubsection{Treatment of Parameters $(\lambda, \theta) \in [\mu,L]^2$}
\label{sec:param_in_mu_L}

Next, we consider the three parameters that cannot be freely selected. These parameters appear in the following entries as a result of the relevant operations:
\begin{itemize}
    \item Parameter $\lambda$: appears in $P_{11}, Q_{11}$
    \item Parameter $\theta$: appears in $Q_{33}$
\end{itemize}
From these, it is clear that each parameter only affects the diagonal elements (see Appendices~\ref{sec:oper_Group_E} and~\ref{sec:oper_Group_F} for details). Moreover, the diagonal elements are monotone functions of each parameter. Therefore it suffices to consider the cases that each parameter is $\mu$ or $L$. Then, regarding  
\begin{equation}
    P = P(\lambda, \theta) \quad \text{and} \quad Q = Q(\lambda, \theta)
\end{equation}
as functions of $\lambda$ and $\theta$,
we have only to check the positive semi-definiteness of the following matrices:
\begin{gather}
    P(\mu, \mu),~ Q(\mu, \mu), \\
    P(\mu, L),~ Q(\mu, L), \\
    P(L, \mu),~ Q(L, \mu), \\
    P(L, L),~ Q(L, L).
\end{gather}

\subsubsection{Reduction of Computational Load via Grouping of Matrix Pairs}

Finally, we describe an implementation technique to reduce computational cost. As mentioned earlier, there are at most 23660 matrix pairs $(P,Q)$.  However, all of them are not necessarily distinct. Moreover, the process of ``maximizing'' $\gamma(t)$ for a given matrix pair is far more computationally intensive than generating the matrices via the operations. Therefore we adopt the following procedure:
\begin{enumerate}
    \item Compute up to 23660 matrix pairs $(P,Q)$.
    \item Group together matrix pairs that are identical in both $P$ and $Q$.
    \item For each resulting group, ``maximize'' $\gamma(t)$.
\end{enumerate}
This approach significantly reduces the computational cost. 

\subsection{Summary of the Proposed Algorithm}
\label{sec:algorithm}

A pseudo-code for our proposed algorithm is shown in Algorithm \ref{alg:lyapunov_construction}. 
The Mathematica codes implementing this algorithm can be found at \url{https://github.com/kentanakadpp/symb_comp_Lyap}. 

\begin{algorithm}[H]
\caption{Procedure for Constructing and Evaluating Lyapunov Function Candidates}
\label{alg:lyapunov_construction}
\begin{algorithmic}[1]
\State \textbf{Input:} A given continuous dynamical system
\State \textbf{Output:} The best combination of matrices yielding the fastest convergence rate
\State Generate initial matrices $P^{0}$ and $Q^{0}$ from the given dynamical system.
\State Apply up to 23660 possible combinations of operations to $P^{0}$ and $Q^{0}$ to obtain candidate pairs $(P, Q)$.
\State Group identical pairs $(P, Q)$ into the same category.
\State For each group of $(P, Q)$ pairs, derive the corresponding convergence rate.
\State Identify and inspect the group that achieves the largest convergence rate.
\end{algorithmic}
\end{algorithm}

\begin{rem}[Cases in which the machine computation does not terminate]
In Step~6 of Algorithm \ref{alg:lyapunov_construction}, there were several cases in which the convergence rate could not be obtained through machine computation.
For such cases, we terminated the computation after a fixed period of time, and inspected the corresponding pairs of matrices manually to derive the convergence rates.
The results presented in the subsequent sections of this paper include both those for which the convergence rate was explicitly obtained by machine computation and those that were manually derived in this manner.
\end{rem}

\section{Results}
\label{sec:Results}

\subsection{Summary of Results}

The results obtained in this study are summarized in Table~\ref{tab:results_of_this_paper}. Here, $a, b, r$, and $\alpha$ are constants independent of time $t$.

\begin{table}[H]
    \centering
    \caption{Summary of Results in This Study}
    \label{tab:results_of_this_paper}
    \renewcommand{\arraystretch}{1.5}
    \setlength{\tabcolsep}{8pt}
    \begin{tabular}{|>{\centering\arraybackslash}m{4.8cm}
                  |>{\centering\arraybackslash}m{2.7cm}
                  |>{\centering\arraybackslash}m{3.5cm}
                  |>{\centering\arraybackslash}m{3.5cm}|}
    \hline
    Class of Continuous-time Dynamical Systems & Properties of Objective Function & Optimal Parameter Settings & Convergence Rate \\
    \hline
    $\nabla^2 f \dot{x} + \nabla f = 0$ & Convex & - & $\mathrm{O} \left( \mathrm{e}^{-t} \right)$ \\
    \hline
    \multirow{2}{*}{$\dot{x} + b \nabla^2 f \dot{x} + \nabla f = 0$}
      & $\mu$-strongly convex, $L$-smooth & $b = -\frac{1}{L}$ & $\mathrm{O} \left( \mathrm{e}^{- \frac{\mu}{1-\frac{\mu}{L}} t} \right)$ \\
      \cline{2-4}
      & $\mu$-strongly convex, $L$-smooth & $b = 0$ & $\mathrm{O} \left( \mathrm{e}^{- 2 \mu t} \right)$ \\
    \hline
    \multirow{2}{*}{$\ddot{x} + a \dot{x} + b \nabla^2 f \dot{x} + \nabla f = 0$}
      & $\mu$-strongly convex & $a = 2\sqrt{\mu} , ~ b = 0$ & $\mathrm{O} \left( \mathrm{e}^{-\sqrt{\mu}t} \right)$ \\
      \cline{2-4}
      & $\mu$-strongly convex & $a = \sqrt{\mu} , ~ b = \frac{1}{\sqrt{\mu}}$ & $\mathrm{O} \left( \mathrm{e}^{-\sqrt{\mu}t} \right)$ \\
    \hline
    \multirow{3}{*}{$\ddot{x} + \frac{r}{t} \dot{x} + \nabla f = 0$}
      & Convex & $r = 3$ & $\mathrm{O} \left( \frac{1}{t^2} \right)$ \\
      \cline{2-4}
      & $\mu$-strongly convex & $r > 3$ & $\mathrm{O} \left( \frac{1}{t^{\frac{1}{2}r+\frac{1}{2}}} \right)$ \\
      \cline{2-4}
      & $\mu$-strongly convex & $r = \frac{4(k^2 + \mu)}{k^2}$ & $\mathrm{O} \left( \mathrm{e}^{-kt} \right)$ (valid only for $t \leq \frac{2(k^2 + \mu)}{k^3}$) \\
    \hline
    $\ddot{x} + \frac{r}{t^\alpha}\dot{x} + \nabla f = 0$ 
      & $\mu$-strongly convex & $r>0,~0<\alpha<1$ & $\mathrm{O}\left( \mathrm{e}^{-\left(\frac{2}{3}-\epsilon\right)\frac{r}{1-\alpha}t^{1-\alpha}} \right)$ (for any positive constant $\epsilon$) \\
    \hline
  \end{tabular}
\end{table}

In the following sections, we detail the results and evaluations obtained by applying the proposed method to each class of continuous-time dynamical systems.

\subsection{Continuous Dynamical System $\nabla^2 f \dot{x} + \nabla f = 0$}
The dynamical system
\begin{equation}
    \nabla^2 f \dot{x} + \nabla f = 0
\end{equation}
corresponds to the discretized damped Newton method. Applying the proposed method, the initial matrices $P^0$ and $Q^0$ are given by
\begin{equation}
    P^0 = O_3,~ Q^0 = \begin{pmatrix}
        0 & \frac{\dot{\gamma}}{2} & 0 & \frac{\dot{\gamma}}{2} & 0 \\
        \frac{\dot{\gamma}}{2} & 0 & \frac{1}{2} & 0 & 0 \\
        0 & \frac{1}{2} & 0 & \frac{1}{2} & 0 \\
        \frac{\dot{\gamma}}{2} & 0 & \frac{1}{2} & 0 & 0 \\
        0 & 0 & 0 & 0 & 0 \\
    \end{pmatrix}.
\end{equation}

By performing an exhaustive search using the proposed method, a total of 21 pairs of matrices $(P, Q)$ are obtained. For these matrices, the objective function is assumed to be simply convex (corresponding to $\mu = 0$), and the convergence rate is assumed to be linear:
\begin{equation}
    f - f_* = \mathrm{O} \left( \mathrm{e}^{-kt} \right), ~ \gamma = kt,
\end{equation}
where $k$ is a positive constant. Among the 21 pairs, only one pair of matrices guarantees a convergence rate with $k>0$:
\begin{equation}
    P = \begin{pmatrix}
        \frac{\lambda}{2} \dot{\gamma} & 0 & 0 \\
        0 & 0 & 0 \\
        0 & 0 & 0
    \end{pmatrix}, ~ Q = \begin{pmatrix}
        \frac{\lambda}{2} \left( \dot{\gamma} - \dot{\gamma}^2 -\ddot{\gamma} \right) & 0 & 0 & 0 & 0 \\
        0 & 0 & 0 & 0 & 0 \\
        0 & 0 & \theta & 0 & 0 \\
        0 & 0 & 0 & 0 & 0 \\
        0 & 0 & 0 & 0 & 0
    \end{pmatrix},
\end{equation}
which we denote as $P^{DN}$ and $Q^{DN}$, respectively. The maximum achievable value of $k$ in this case is $k=1$, yielding the following convergence rate.

\begin{thm}[Convergence Rate of $\nabla^2 f \dot{x} + \nabla f = 0$]
    \label{Dumped_Newton}
    For convex functions $f$, the continuous-time dynamical system $\nabla^2 f \dot{x} + \nabla f = 0$ exhibits the convergence rate
    \begin{equation}
        f(x(t)) - f_* = \mathrm{O} \left( \mathrm{e}^{-t} \right).
    \end{equation}
\end{thm}

\begin{proof}
Substituting $\gamma = kt$ into the matrices $P^{DN}$ and $Q^{DN}$ yields
\begin{equation}
    \label{Lyapunov_Dumped_Newton}
    P^{DN} = \begin{pmatrix}
        \frac{\lambda k}{2} & 0 & 0 \\
        0 & 0 & 0 \\
        0 & 0 & 0
    \end{pmatrix}, ~ Q^{DN} = \begin{pmatrix}
        \frac{\lambda}{2} \left( k - k^2 \right) & 0 & 0 & 0 & 0 \\
        0 & 0 & 0 & 0 & 0 \\
        0 & 0 & \theta & 0 & 0 \\
        0 & 0 & 0 & 0 & 0 \\
        0 & 0 & 0 & 0 & 0
    \end{pmatrix}.
\end{equation}
The necessary and sufficient conditions for these matrices to be positive semi-definite are
\begin{gather}
    \frac{\lambda k}{2} \geq 0, \\
    \frac{\lambda}{2} \left( k - k^2 \right) \geq 0, \\
    \theta \geq 0.
\end{gather}
Since $f$ is assumed to be convex, $\lambda \geq 0$ and $\theta \geq 0$ hold. Therefore these conditions reduce to
\begin{gather}
    k \geq 0, \\
    k - k^2 \geq 0,
\end{gather}
which is satisfied for
\begin{equation}
    0 \leq k \leq 1.
\end{equation}
Thus, choosing $k=1$ ensures positive semi-definiteness, guaranteeing the convergence rate
\begin{equation}
    f(x(t)) - f_* = \mathrm{O} \left( \mathrm{e}^{-t} \right).
\end{equation}
\end{proof}

The convergence rate obtained here is identical to that reported by Kamijima et al.~(2024)~\cite{TomoyaKamijima2024}. However, unlike their heuristic approach, the present result is obtained through an exhaustive exploration of the operations, which demonstrates the advantage of the proposed method. Details of the convergence rate proof using the explicit representation of the Lyapunov function are provided in Appendix~\ref{sec:pr_thm_d_Newton}.

\subsection{Continuous Dynamical System $\dot{x} + b \nabla^2 f \dot{x} + \nabla f = 0$}

For the continuous dynamical system
\begin{equation}
    \dot{x} + b \nabla^2 f \dot{x} + \nabla f = 0,
\end{equation}
when applying the proposed method, the initial matrices $P^0, Q^0$ are given by
\begin{equation}
    P^0 = O_3,~ Q^0 = \begin{pmatrix}
        0 & \frac{\dot{\gamma}}{2} & \frac{\dot{\gamma}}{2} & \frac{b\dot{\gamma}}{2} & 0 \\
        \frac{\dot{\gamma}}{2} & 0 & \frac{1}{2} & 0 & 0 \\
        \frac{\dot{\gamma}}{2} & \frac{1}{2} & 1 & \frac{b}{2} & 0 \\
        \frac{b\dot{\gamma}}{2} & 0 & \frac{b}{2} & 0 & 0 \\
        0 & 0 & 0 & 0 & 0 \\
    \end{pmatrix}.
\end{equation}

By performing exhaustive search with the proposed method, a total of 42 pairs of matrices $P,Q$ are obtained.  
Here, the objective function is assumed to be $\mu$-strongly convex and $L$-smooth ($0 < \mu < L$), and the convergence rate is assumed to be first-order,
\begin{equation}
    f - f_* = \mathrm{O} \left( \mathrm{e}^{-kt} \right), ~ \gamma = kt,
\end{equation}
where $k$ is a positive constant. Under this assumption, all 42 pairs guarantee a convergence rate with $k > 0$, which are categorized as follows:
\begin{itemize}
    \item $k=\frac{\mu}{1-\frac{\mu}{L}}$: 1 pair (obtained only under the assumption of $\mu$-strong convexity and $L$-smoothness ($0 < \mu < L$), with $b=-\frac{1}{L}$),
    \item $k=2\mu$: 21 pairs (The same result holds even if only $\mu$-strong convexity is assumed. All coincide when $b=0$),
    \item $k=\mu$: 20 pairs (the same result holds even if only $\mu$-strong convexity is assumed).
\end{itemize}

Among these, we focus on the pair that gives $k=\frac{\mu}{1-\frac{\mu}{L}}$, and the pair after substituting $b=0$ that gives $k=2\mu$.  
First, the pair that yields $k=\frac{\mu}{1-\frac{\mu}{L}}$ is
\begin{equation}
    P = \begin{pmatrix}
        \frac{1}{2}(1+b\lambda)\dot{\gamma} & 0 & 0 \\
        0 & 0 & 0 \\
        0 & 0 & 0
    \end{pmatrix}, ~ Q = \begin{pmatrix}
        \frac{1}{2} \left( \lambda \dot{\gamma} - (1+b\lambda)\dot{\gamma}^2 - (1+b\lambda)\ddot{\gamma} \right) & 0 & 0 & 0 & 0 \\
        0 & 0 & 0 & 0 & 0 \\
        0 & 0 & 1+b\theta & 0 & 0 \\
        0 & 0 & 0 & 0 & 0 \\
        0 & 0 & 0 & 0 & 0
    \end{pmatrix},
\end{equation}
which we refer to as $P^{FO},Q^{FO}$.  
When $b=-\frac{1}{L}$, the maximum value $k=\frac{\mu}{1-\frac{\mu}{L}}$ is obtained.  
The convergence rate guaranteed by this pair is given by the following theorem.

\begin{thm}[Convergence rate of the continuous dynamical system $\dot{x} -\frac{1}{L} \nabla^2 f \dot{x} + \nabla f = 0$]
    \label{First_Order}
     For an $\mu$-strongly convex and $L$-smooth function $f$ ($0 < \mu < L$) the continuous dynamical system $\dot{x} -\frac{1}{L} \nabla^2 f \dot{x} + \nabla f = 0$ admits, the convergence rate
    \begin{equation}
        f(x(t)) - f_* = \mathrm{O} \left( \mathrm{e}^{- \frac{\mu}{1-\frac{\mu}{L}} t} \right).
    \end{equation}
\end{thm}

\begin{proof}
    Substituting $\gamma = kt, b = -\frac{1}{L}$ into $P^{FO},~ Q^{FO}$ yields
    \begin{equation}
        \label{Lyapunov_First_Order}
        P^{FO} = \begin{pmatrix}
        \frac{1}{2}\left(1 - \frac{\lambda}{L}\right)k & 0 & 0 \\
        0 & 0 & 0 \\
        0 & 0 & 0
        \end{pmatrix}, ~ Q^{FO} = \begin{pmatrix}
        \frac{1}{2} \left( \lambda k - \left(1 - \frac{\lambda}{L}\right)k^2  \right) & 0 & 0 & 0 & 0 \\
        0 & 0 & 0 & 0 & 0 \\
        0 & 0 & \left(1 - \frac{\theta}{L}\right) & 0 & 0 \\
        0 & 0 & 0 & 0 & 0 \\
        0 & 0 & 0 & 0 & 0
        \end{pmatrix}.
    \end{equation}
    The necessary and sufficient conditions for these matrices to be positive semi-definite are
    \begin{gather}
        \frac{1}{2}\left(1 - \frac{\lambda}{L}\right)k \geq 0, \\
        \frac{1}{2} \left( \lambda k - \left(1 - \frac{\lambda}{L}\right)k^2  \right) \geq 0, \\
        1 - \frac{\theta}{L} \geq 0.
    \end{gather}
    Since the objective function $f$ is assumed to be $\mu$-strongly convex and $L$-smooth ($0 < \mu < L$), it follows that $\mu \leq \lambda \leq L$ and $\mu \leq \theta \leq L$. Thus the conditions reduce to
    \begin{gather}
        k \geq 0, \\
        \mu k - \left(1 - \frac{\mu}{L}\right)k^2 \geq 0.
    \end{gather}
    The solutions are
    \begin{equation}
        0 \leq k \leq \frac{\mu}{1-\frac{\mu}{L}},
    \end{equation}
    and therefore, for $k=\frac{\mu}{1-\frac{\mu}{L}}$, positive semidefiniteness holds and the convergence rate
    \begin{equation}
        f(x(t)) - f_* = \mathrm{O} \left( \mathrm{e}^{-\frac{\mu}{1-\frac{\mu}{L}}t} \right)
    \end{equation}
    is guaranteed.
\end{proof}

When $\frac{1}{2} < \frac{\mu}{L} < 1$ holds, the convergence rate obtained here surpasses that of the gradient flow (corresponding to gradient descent)
\begin{equation}
    f(x(t)) - f_* = \mathrm{O} \left( \mathrm{e}^{-2\mu t} \right),
\end{equation}
which will be discussed later.
This suggests the possibility of accelerating
corresponding optimization methods
by incorporating the Hessian. 
The detailed proof of the convergence rate using Lyapunov functions derived from the matrix representation is given in Appendix~\ref{sec:pr_thm_1st_order}.

% We also examine the parameter $b=-\frac{1}{L}$ obtained here.  
% The necessary and sufficient conditions for $P^{FO},~Q^{FO}$ to be positive semidefinite are the following three inequalities:
% \begin{gather}
%     \frac{1}{2}(1+b\lambda)\dot{\gamma} \geq 0, \\
%     \frac{1}{2} \left( \lambda \dot{\gamma} - (1+b\lambda)\dot{\gamma}^2 - (1+b\lambda)\ddot{\gamma} \right) \geq 0, \\
%     1 + b\theta \geq 0.
% \end{gather}
% Since these must hold for all $\lambda \in [\mu, L]$ and $\theta \in [\mu, L]$, the first and third conditions reduce, under the assumption $\dot{\gamma}>0$, to
% \begin{equation}
%     b \geq - \frac{1}{L}.
% \end{equation}
% For the second condition, substituting $\gamma = kt$ ($k>0$ constant) gives
% \begin{equation}
%     \frac{1}{2}(\lambda k - (1+b\lambda) k^2) \geq 0,
% \end{equation}
% which, after multiplying both sides by $\frac{2}{k}$, becomes
% \begin{equation}
%     \lambda \geq (1+b\lambda) k.
% \end{equation}
% When $1+b\lambda = 0$, this holds automatically.  
% For $1+b\lambda > 0$, the maximum value of $k$ is achieved when $1+b\lambda$ is minimized. Therefore, minimizing $b$ subject to $b\geq -\frac{1}{L}$ yields $b=-\frac{1}{L}$, which preserves the positive semidefiniteness of $P^{FO},Q^{FO}$ while maximizing $k$.  
% Although this reasoning illustrates how to derive parameters that maximize $k$ under positive semidefiniteness conditions, we omit the derivation in subsequent examples due to its complexity.

Next, we turn to the matrices corresponding to $k=2\mu$.  
In this case, 
as long as $f$ is $\mu$-strongly convex ($0 < \mu$), a total of 21 pairs are found even without assuming $L$-smoothness of $f$. 
All of these required $b=0$ to ensure the positive semi-definiteness of $P,Q$, and they coincide by the substitution of $b=0$.  
We show them here:
\begin{equation}
    P = \begin{pmatrix}
        0 & 0 & 0 \\
        0 & 0 & 0 \\
        0 & 0 & 0
    \end{pmatrix}, ~ Q = \begin{pmatrix}
        \frac{\lambda}{2}\dot{\gamma} & 0 & \frac{\dot{\gamma}}{2} & 0 & 0 \\
        0 & 0 & 0 & 0 & 0 \\
        \frac{\dot{\gamma}}{2} & 0 & 1 & 0 & 0 \\
        0 & 0 & 0 & 0 & 0 \\
        0 & 0 & 0 & 0 & 0
    \end{pmatrix},
\end{equation}
which we denote as $P^{GD},Q^{GD}$.  The convergence rate guaranteed by this pair is given in the following theorem.

\begin{thm}[Convergence rate of the continuous dynamical system $\dot{x}+ \nabla f = 0$]
    \label{Gradient_Descent}
    For an $\mu$-strongly convex function $f$ ($0 < \mu$), the continuous dynamical system $\dot{x}+ \nabla f = 0$ admits the convergence rate
    \begin{equation}
        f(x(t)) - f_* = \mathrm{O} \left( \mathrm{e}^{- 2\mu t} \right).
    \end{equation}
\end{thm}

\begin{proof}
    Substituting $\gamma = kt$ into $P^{GD},~ Q^{GD}$ yields
    \begin{equation}
        \label{Lyapunov_Gradient_Descent}
        P^{GD} = \begin{pmatrix}
        0 & 0 & 0 \\
        0 & 0 & 0 \\
        0 & 0 & 0
        \end{pmatrix}, ~ Q^{GD} = \begin{pmatrix}
        \frac{\lambda k}{2} & 0 & \frac{k}{2} & 0 & 0 \\
        0 & 0 & 0 & 0 & 0 \\
        \frac{k}{2} & 0 & 1 & 0 & 0 \\
        0 & 0 & 0 & 0 & 0 \\
        0 & 0 & 0 & 0 & 0
        \end{pmatrix}.
    \end{equation}
    The necessary and sufficient conditions for these matrices to be positive semi-definite are
    \begin{gather}
        \frac{\lambda k}{2} \geq 0, \\
        \frac{\lambda k}{2} \cdot 1 - \left( \frac{k}{2} \right)^2 \geq 0.
    \end{gather}
    Since the objective function $f$ is assumed to be $\mu$-strongly convex ($0 < \mu$), it follows that $\mu \leq \lambda$, and thus the conditions reduce to
    \begin{gather}
        k \geq 0, \\
        \frac{1}{4} \left( 2\mu k - k^2 \right) \geq 0.
    \end{gather}
    The solutions are
    \begin{equation}
        0 \leq k \leq 2 \mu,
    \end{equation}
    and therefore, for $k=2\mu$, positive semi-definiteness holds and the convergence rate
    \begin{equation}
        f(x(t)) - f_* = \mathrm{O} \left( \mathrm{e}^{-2\mu t} \right)
    \end{equation}
    is guaranteed.
\end{proof}

This corresponds to the gradient flow, which is the continuous dynamical system associated with gradient descent.  
Although this convergence rate is already known, it strictly improves upon the rate obtained by Kamijima et al.~(2024)~\cite{TomoyaKamijima2024},
\begin{equation}
    f(x(t)) - f_* = \mathrm{O} \left( \mathrm{e}^{-\mu t} \right),
\end{equation}
demonstrating that our proposed method can reveal previously overlooked Lyapunov functions.  
The detailed proof of the convergence rate using Lyapunov functions derived from the matrix representation is given in Appendix~\ref{sec:pr_thm_GD}.

\subsection{Continuous Dynamical System $\ddot{x} + a \dot{x}+b \nabla^2 f \dot{x} + \nabla f = 0$}

For the system
\begin{equation}
    \ddot{x} + a\dot{x}+ b \nabla^2 f \dot{x} + \nabla f = 0,
\end{equation}
when the proposed method is applied, the initial matrices $P^0, Q^0$ are given by
\begin{equation}
    P^0 = O_3,~ Q^0 = \begin{pmatrix}
        0 & \frac{\dot{\gamma}}{2} & \frac{a\dot{\gamma}}{2} & \frac{b\dot{\gamma}}{2} & \frac{\dot{\gamma}}{2} \\
        \frac{\dot{\gamma}}{2} & 0 & \frac{1}{2} & 0 & 0 \\
        \frac{a\dot{\gamma}}{2} & \frac{1}{2} & a & \frac{b}{2} & \frac{1}{2} \\
        \frac{b\dot{\gamma}}{2} & 0 & \frac{b}{2} & 0 & 0 \\
        \frac{\dot{\gamma}}{2} & 0 & \frac{1}{2} & 0 & 0 \\
    \end{pmatrix}.
\end{equation}
Applying the proposed method with exhaustive search then yields a total of 210 pairs of matrices $(P,Q)$.  
For these, assuming that the objective function is $\mu$-strongly convex ($0 < \mu$), the convergence rate is assumed to be first-order,
\begin{equation}
    f - f_* = \mathrm{O} \left( \mathrm{e}^{-kt} \right), ~ \gamma = kt,
\end{equation}
where $k$ is a positive constant. Among the 210 pairs, 43 pairs are found to guarantee convergence with $k>0$. Their breakdown is as follows:
\begin{itemize}
    \item $k=\sqrt{\mu}$: 22 pairs
    \item $k=(2-\sqrt{2})\sqrt{\mu}$: 21 pairs
\end{itemize}

Since $\sqrt{\mu} > (2-\sqrt{2})\sqrt{\mu}$, we focus on the 22 pairs that yield $k=\sqrt{\mu}$. These 22 pairs can be further divided into two groups:
\begin{itemize}
    \item 21 pairs in which the condition $b=0$ is required for $P,Q$ to be positive semidefinite, and which yield $k=\sqrt{\mu}$ when $a=2\sqrt{\mu}, ~b=0$.
    \item 1 pair that yields $k=\sqrt{\mu}$ when $a=\sqrt{\mu}, ~b=\frac{1}{\sqrt{\mu}}$.
\end{itemize}
Although both cases yield $k=\sqrt{\mu}$, they correspond to different continuous dynamical systems, and thus we present them separately.

We first discuss the first group. For $\mu$-strongly convex ($0<\mu$) objective functions, 21 pairs were found in total, all of which yield $k=\sqrt{\mu}$ under $a=2\sqrt{\mu},~b=0$. In every case, the condition $b=0$ was necessary for $P,Q$ to be positive semidefinite, and substituting $b=0$ resulted in the same pair of matrices. Therefore, we present here the matrices after substituting $b=0$:
\begin{equation}
    P = \frac{1}{2} \begin{pmatrix}
        a \dot{\gamma} - \dot{\gamma}^2 - \ddot{\gamma} & 0 & \dot{\gamma} \\
        0 & 0 & 0 \\
        \dot{\gamma} & 0 & 1
    \end{pmatrix}, ~ Q = \frac{1}{2} \begin{pmatrix}
        -a \dot{\gamma}^2 + \dot{\gamma}^3 - a \ddot{\gamma} + \lambda \dot{\gamma} + 3 \dot{\gamma}\ddot{\gamma} + \dddot{\gamma} & 0 & 0 & 0 & 0 \\
        0 & 0 & 0 & 0 & 0 \\
        0 & 0 & 2a - 3\dot{\gamma} & 0 & 0 \\
        0 & 0 & 0 & 0 & 0 \\
        0 & 0 & 0 & 0 & 0
    \end{pmatrix}.
\end{equation}
We hereafter denote these as $P^{SCNAG}, Q^{SCNAG}$.  
The convergence rate guaranteed by this pair is given by the following theorem.

\begin{thm}[Convergence rate of the continuous dynamical system $\ddot{x} + 2\sqrt{\mu} \dot{x} +\nabla f = 0$]
    \label{SC-NAG}
    For an $\mu$-strongly convex ($0<\mu$) function $f$, the continuous dynamical system
    \begin{equation}
        \ddot{x} + 2\sqrt{\mu} \dot{x} +\nabla f = 0
    \end{equation}
    has the convergence rate
    \begin{equation}
        f(x(t)) - f_* = \mathrm{O} \left( \mathrm{e}^{- \sqrt{\mu} t} \right).
    \end{equation}
\end{thm}

\begin{proof}
    Substituting $\gamma=kt$ and $a=2\sqrt{\mu}$ into the pair $P^{SCNAG}, Q^{SCNAG}$ yields
    \begin{equation}
        \label{Lyapunov_SC-NAG}
        P^{SCNAG} = \frac{1}{2} \begin{pmatrix}
        2\sqrt{\mu} k - k^2 & 0 & k \\
        0 & 0 & 0 \\
        k & 0 & 1
        \end{pmatrix}, ~ Q^{SCNAG} = \frac{1}{2} \begin{pmatrix}
        -2\sqrt{\mu} k^2 + k^3 + \lambda k & 0 & 0 & 0 & 0 \\
        0 & 0 & 0 & 0 & 0 \\
        0 & 0 & 4\sqrt{\mu} - 3k & 0 & 0 \\
        0 & 0 & 0 & 0 & 0 \\
        0 & 0 & 0 & 0 & 0
        \end{pmatrix}.
    \end{equation}
    The necessary and sufficient conditions for these matrices to be positive semi-definite are
    \begin{gather}
        2\sqrt{\mu} k - k^2 \geq 0 \\
        (2\sqrt{\mu} k - k^2) \cdot 1 - k^2 \geq 0 \\
        -2\sqrt{\mu} k^2 + k^3 + \lambda k \geq 0 \\
        4\sqrt{\mu} - 3k \geq 0.
    \end{gather}
    These simplify to
    \begin{gather}
        k (2\sqrt{\mu} - k) \geq 0 \\
        2k(\sqrt{\mu}-k) \geq 0 \\
        k(k^2 - 2\sqrt{\mu}k + \lambda) \geq 0 \\
        -3\left( k - \frac{4}{3}\sqrt{\mu} \right) \geq 0.
    \end{gather}
    Combining these conditions, we obtain
    \begin{gather}
        0 \leq k \leq \sqrt{\mu} \\
        (k - \sqrt{\mu})^2 + \lambda - \mu \geq 0.
    \end{gather}
    Since $f$ is assumed to be $\mu$-strongly convex ($0<\mu$), we have $\mu \leq \lambda$. Thus, the feasible condition is
    \begin{equation}
        0 \leq k \leq \sqrt{\mu}.
    \end{equation}
    Therefore, positive semi-definiteness is satisfied at $k=\sqrt{\mu}$, guaranteeing the convergence rate
    \begin{equation}
        f(x(t)) - f_* = \mathrm{O} \left( \mathrm{e}^{-\sqrt{\mu} t} \right).
    \end{equation}
\end{proof}

The continuous dynamical system obtained here corresponds to the strongly convex version of Nesterov’s accelerated gradient method (SC-NAG), and the convergence rate also coincides with that of the SC-NAG. While the dynamical system and its rate are not novel discoveries themselves, this illustrates that the proposed method can systematically guarantee the already-known continuous dynamical systems and its rate. The detailed proofs via Lyapunov functions written in matrix form are provided in Appendix~\ref{sec:pr_thm_SC-NAG}.

Next, we turn to the other group. This case does not include the condition $b=0$ for $P,Q$ to be positive semi-definite, and yields $k=\sqrt{\mu}$ under $a=\sqrt{\mu},~b=\frac{1}{\sqrt{\mu}}$. There is only one such pair: 
\begin{equation}
    P = \frac{1}{2} \begin{pmatrix}
        b \lambda \dot{\gamma} & 0 & \dot{\gamma} \\
        0 & 0 & 0 \\
        \dot{\gamma} & 0 & 1
    \end{pmatrix}, ~ Q = \frac{1}{2} \begin{pmatrix}
        \lambda (\dot{\gamma} - b \dot{\gamma}^2 - b \ddot{\gamma}) & 0 & a \dot{\gamma} - \dot{\gamma}^2 - \ddot{\gamma} & 0 & 0 \\
        0 & 0 & 0 & 0 & 0 \\
        a \dot{\gamma} - \dot{\gamma}^2 - \ddot{\gamma} & 0 & 2a + 2b\theta - 3 \dot{\gamma} & 0 & 0 \\
        0 & 0 & 0 & 0 & 0 \\
        0 & 0 & 0 & 0 & 0
    \end{pmatrix}.
\end{equation}
We hereafter denote these by $P^{SO}, Q^{SO}$. The convergence rate guaranteed by this pair is given by the following theorem.

\begin{thm}[Convergence rate of the continuous dynamical system $\ddot{x} + \sqrt{\mu} \dot{x} + \frac{1}{\sqrt{\mu}} \nabla^2 f \dot{x} + \nabla f = 0$]
    \label{Second_Order}
    For an $\mu$-strongly convex ($0<\mu$) function $f$, the continuous dynamical system $\ddot{x} + \sqrt{\mu} \dot{x} + \frac{1}{\sqrt{\mu}} \nabla^2 f \dot{x} + \nabla f = 0$ has the convergence rate
    \begin{equation}
        f(x(t)) - f_* = \mathrm{O} \left( \mathrm{e}^{- \sqrt{\mu} t} \right).
    \end{equation}
\end{thm}

\begin{proof}
    Substituting $\gamma=kt$, $a=\sqrt{\mu}$, and $b=\frac{1}{\sqrt{\mu}}$ into the pair $P^{SO}, Q^{SO}$ yields
    \begin{equation}
        \label{Lyapunov_Second_Order}
        P^{SO} = \frac{1}{2} \begin{pmatrix}
        \frac{\lambda}{\sqrt{\mu}}k & 0 & k \\
        0 & 0 & 0 \\
        k & 0 & 1
        \end{pmatrix}, ~ Q^{SO} = \frac{1}{2} \begin{pmatrix}
        \lambda (k - \frac{1}{\sqrt{\mu}}k^2) & 0 & \sqrt{\mu}k - k^2 & 0 & 0 \\
        0 & 0 & 0 & 0 & 0 \\
        \sqrt{\mu}k - k^2 & 0 & 2\sqrt{\mu} + 2\frac{\theta}{\sqrt{\mu}} - 3 k & 0 & 0 \\
        0 & 0 & 0 & 0 & 0 \\
        0 & 0 & 0 & 0 & 0
        \end{pmatrix}.
    \end{equation}
    The necessary and sufficient conditions for these matrices to be positive semidefinite are
    \begin{gather}
        \frac{\lambda}{\sqrt{\mu}}k \geq 0 \\
        \frac{\lambda}{\sqrt{\mu}}k \cdot 1 - k^2 \geq 0 \\
        \lambda (k - \frac{1}{\sqrt{\mu}}k^2) \geq 0 \\
        2\sqrt{\mu} + 2\frac{\theta}{\sqrt{\mu}} - 3k \geq 0 \\
        \left( \lambda (k - \frac{1}{\sqrt{\mu}}k^2) \right)\left( 2\sqrt{\mu} + 2\frac{\theta}{\sqrt{\mu}} - 3k \right) - (\sqrt{\mu}k - k^2)^2 \geq 0.
    \end{gather}
    Simplifying them, we obtain
    \begin{gather}
        k \geq 0 \\
        k\left(\frac{\lambda}{\sqrt{\mu}} - k\right) \geq 0 \\
        \frac{\lambda}{\sqrt{\mu}}k(\sqrt{\mu} - k) \geq 0 \\
        2\sqrt{\mu} + 2\frac{\theta}{\sqrt{\mu}} - 3k \geq 0 \\
        k(\sqrt{\mu} - k)\left(\frac{\lambda}{\sqrt{\mu}} \left(2\sqrt{\mu} + 2\frac{\theta}{\sqrt{\mu}} - 3k \right) - \sqrt{\mu}k + k^2\right) \geq 0.
    \end{gather}
    Since $f$ is assumed to be $\mu$-strongly convex ($0<\mu$), we have $\mu \leq \lambda, ~\mu \leq \theta$, hence $\frac{\lambda}{\sqrt{\mu}} \geq \sqrt{\mu}, ~\frac{\theta}{\sqrt{\mu}} \geq \sqrt{\mu}$. Therefore, the conditions reduce to
    \begin{gather}
        0 \leq k \leq \sqrt{\mu} \\
        2\sqrt{\mu} + 2\frac{\theta}{\sqrt{\mu}} - 3k \geq 0 \\
        k(\sqrt{\mu} - k)\left(\frac{\lambda}{\sqrt{\mu}}\left(2\sqrt{\mu} + 2\frac{\theta}{\sqrt{\mu}} - 3k\right) - \sqrt{\mu}k + k^2\right) \geq 0.
    \end{gather}
    Therefore, positive semidefiniteness is satisfied at $k=\sqrt{\mu}$, guaranteeing the convergence rate
    \begin{equation}
        f(x(t)) - f_* = \mathrm{O} \left( \mathrm{e}^{-\sqrt{\mu} t} \right).
    \end{equation}
\end{proof}

The convergence rate obtained here coincides with that of the strongly convex version of Nesterov’s accelerated gradient method (SC-NAG), but the continuous dynamical system itself is a novel one. This system corresponds to a second-order optimization method and thus deserves further investigation. In this way, even when the convergence rate itself is not truly improved, the proposed method allows us to discover new continuous dynamical systems that yield favorable convergence rates. The detailed proofs via Lyapunov functions written in matrix form are provided in Appendix~\ref{sec:pr_thm_Second_Order}.

\subsection{Continuous Dynamical System $\ddot{x} + \frac{r}{t} \dot{x} + \nabla f = 0$}

For the continuous dynamical system
\begin{equation}
    \ddot{x} + \frac{r}{t} \dot{x} + \nabla f = 0,
\end{equation}
applying the proposed method yields initial matrices $P^0, Q^0$ of the form
\begin{equation}
    P^0 = O_3,~ Q^0 = \begin{pmatrix}
        0 & \frac{\dot{\gamma}}{2} & \frac{r}{2t} & 0 & \frac{\dot{\gamma}}{2} \\
        \frac{\dot{\gamma}}{2} & 0 & \frac{1}{2} & 0 & 0 \\
        \frac{r}{2t} & \frac{1}{2} & \frac{r}{t} & 0 & \frac{1}{2} \\
        0 & 0 & 0 & 0 & 0 \\
        \frac{\dot{\gamma}}{2} & 0 & \frac{1}{2} & 0 & 0 \\
    \end{pmatrix}.
\end{equation}
Performing the exhaustive search in the proposed framework produced a total of 10 distinct candidate pairs $(P,Q)$. First, assuming that the objective function is merely convex (corresponding to $\mu = 0$) and that the convergence rate is polynomial,
\begin{equation}
    f - f_* = \mathrm{O} \left( \frac{1}{t^k} \right), ~ \gamma = k \log t,
\end{equation}
with $k>0$ a positive constant, we found that two out of the 10 pairs guarantee some $k>0$. Both of these pairs yield $k=2$ when $r=3$. We present one such pair:
\begin{equation}
    \label{PNAG_QNAG}
    P = \frac{1}{2} \begin{pmatrix}
        \frac{r\dot{\gamma} - \dot{\gamma}^2 t- \ddot{\gamma} t}{t} & 0 & \dot{\gamma} \\
        0 & 0 & 0 \\
        \dot{\gamma} & 0 & 1
    \end{pmatrix}, ~ Q = \frac{1}{2} \begin{pmatrix}
        \frac{- r \dot{\gamma}^2 t+ {\dot{\gamma}^3 t^2 + \dot{\gamma} \left( r + \lambda t^2 + 3\ddot{\gamma} t^2\right)}+ \left( -r \ddot{\gamma} + \dddot{\gamma} t \right)t}{t^2} & 0 & 0 & 0 & 0 \\
        0 & 0 & 0 & 0 & 0 \\
        0 & 0 & \frac{2r}{t} - 3 \dot{\gamma} & 0 & 0 \\
        0 & 0 & 0 & 0 & 0 \\
        0 & 0 & 0 & 0 & 0 \\
    \end{pmatrix} \\
\end{equation}
We denote these matrices by $P^{NAG},Q^{NAG}$. Since these yield $k=2$ at $r=3$, the resulting convergence rate is stated in the following theorem.

\begin{thm}[Convergence rate of the continuous dynamical system $\ddot{x} + \frac{3}{t} \dot{x} + \nabla f = 0$]
    \label{NAG_Convex}
    The continuous dynamical system $\ddot{x} + \frac{3}{t} \dot{x} + \nabla f = 0$ satisfies the convergence rate
    \begin{equation}
        f(x(t)) - f_* = \mathrm{O} \left( \frac{1}{t^2} \right)
    \end{equation}
    for convex functions $f$.
\end{thm}

\begin{proof}
    Substituting $\gamma = k \log t$ and $r=3$ into the matrices $P^{NAG},~ Q^{NAG}$ yields
    \begin{equation}
        \label{Lyapunov_NAG_Convex}
        P^{NAG} = \frac{1}{2} \begin{pmatrix}
        \frac{4k - k^2}{t^2} & 0 & \frac{k}{t} \\
        0 & 0 & 0 \\
        \frac{k}{t} & 0 & 1
    \end{pmatrix}, ~ Q^{NAG} = \frac{1}{2} \begin{pmatrix}
        \frac{k^3 - 6k^2 + 8k}{t^3} + \frac{\lambda k}{t} & 0 & 0 & 0 & 0 \\
        0 & 0 & 0 & 0 & 0 \\
        0 & 0 & \frac{6 - 3k}{t} & 0 & 0 \\
        0 & 0 & 0 & 0 & 0 \\
        0 & 0 & 0 & 0 & 0 \\
    \end{pmatrix} 
    \end{equation}
    The necessary and sufficient conditions for these matrices to be positive semidefinite are
    \begin{gather}
        \frac{4k - k^2}{t^2} \geq 0 \\
        \frac{4k - k^2}{t^2} \cdot 1 - \left( \frac{k}{t} \right)^2 \geq 0 \\
        \frac{k^3 - 6k^2 + 8k}{t^3} + \frac{\lambda k}{t} \geq 0 \\
        \frac{6 - 3k}{t} \geq 0
    \end{gather}
    Rearranging these yields
    \begin{gather}
        \frac{k(4-k)}{t} \geq 0 \label{NAG_log_condition1} \\
        \frac{2k(2-k)}{t^2} \geq 0 \label{NAG_log_condition2} \\
        \frac{k(k-2)(k-4)}{t^3} + \frac{\lambda k}{t} \geq 0 \label{NAG_log_condition3} \\
        \frac{3(2-k)}{t}\geq 0 \label{NAG_log_condition4}
    \end{gather}
    Since $f$ is assumed convex, $\lambda \ge 0$ holds. For $t>0$, inequalities \eqref{NAG_log_condition1}, \eqref{NAG_log_condition2}, \eqref{NAG_log_condition3}, \eqref{NAG_log_condition4} together imply
    \begin{equation}
        0 \leq k \leq 2.
    \end{equation}
    Hence $k=2$ satisfies the positive semidefiniteness conditions, and therefore
    \begin{equation}
        f(x(t)) - f_* = \mathrm{O} \left( \frac{1}{t^2} \right)
    \end{equation}
    is guaranteed.
\end{proof}

The theorem above recovers the convergence rate for the continuous dynamical system corresponding to Nesterov's accelerated gradient method (NAG). Although the convergence rate itself is already known, reproducing it through the proposed method is meaningful. The detailed Lyapunov-function-based proof using the matrix representation is included in Appendix~\ref{sec:pr_thm_NAG_Convex}.

Next, we narrow the class of objective functions to $\mu$-strongly convex ones without changing the assumed form of the convergence rate. For the same two candidate pairs mentioned above, we found that under certain conditions they can yield $k>2$. Concretely, using the same matrices $P^{NAG}$ and $Q^{NAG}$, we can verify that 
\begin{gather}
    r = 1 + 2 \sqrt{1+\mu T^2} \quad \text{and} \\
    \gamma(t) = \left( 1 + \sqrt{1+\mu T^2} \right) \log t
\end{gather}
guarantee their positive semi-definiteness for $t \geq T$, 
where $T>0$ is a certain positive number. 
If positive semi-definiteness holds for all $t\ge T$, then we have
\begin{equation}
    \mathrm{e}^{\gamma}(f-f_*) \leq \mathcal{E}(t) \leq \mathcal{E}(T) = \mathrm{const.}
\end{equation}
by the same argument as in the unconstrained-in-time cases.
Thus the usual convergence estimate follows. Moreover, since 
\begin{equation}
    r = 1 + 2 \sqrt{1+\mu T^2} > 3
\end{equation}
and $r$ is monotonically increasing in $T$, it is convenient to fix $r$ and express the convergence rate in terms of $r$. This leads to the following theorem.

\begin{thm}[Convergence rate of the continuous dynamical system $\ddot{x}+ \frac{r}{t} \dot{x} + \nabla f = 0$]
    \label{NAG_strong_log}
    For an $\mu$-strongly convex function $f~(\mu > 0)$ , the continuous dynamical system $\ddot{x}+ \frac{r}{t} \dot{x} + \nabla f = 0 ~ (r>3)$ satisfies, the convergence rate
    \begin{equation}
        f(x(t)) - f_* = \mathrm{O} \left( \frac{1}{t^{\frac{1}{2}r + \frac{1}{2}}} \right).
    \end{equation}
\end{thm}
\begin{proof}
    Substituting $r = 1 + 2 \sqrt{1+\mu T^2}, ~\gamma = k \log t$ into $P^{NAG}$ and $Q^{NAG}$ yields
    \begin{gather}
    \label{Lyapunov_NAG_strong_log}
    P^{NAG} = \frac{1}{2} \begin{pmatrix}
        \frac{k(2+2 \sqrt{1+\mu T^2}-k)}{t^2} & 0 & \frac{k}{t} \\
        0 & 0 & 0 \\
        \frac{k}{t} & 0 & 1
    \end{pmatrix}, \\~ Q^{NAG} = \frac{1}{2} \begin{pmatrix}
        \frac{k(k^2 -(4+ 2\sqrt{1+\mu T^2}) k + (4 + 4 \sqrt{1+\mu T^2} + \lambda t^2))}{t^3} & 0 & 0 & 0 & 0 \\
        0 & 0 & 0 & 0 & 0 \\
        0 & 0 & \frac{2 + 4 \sqrt{1+\mu T^2}-3k}{t} & 0 & 0 \\
        0 & 0 & 0 & 0 & 0 \\
        0 & 0 & 0 & 0 & 0 \\
    \end{pmatrix}. \\
    \end{gather}
    The necessary and sufficient conditions for positive semi-definiteness become
    \begin{gather}
        \frac{k(2+2 \sqrt{1+\mu T^2}-k)}{t^2} \geq 0, \\
        \frac{k(2+2 \sqrt{1+\mu T^2}-k)}{t^2} \cdot 1 - \left( \frac{k}{t} \right)^2 \geq 0, \\
        \frac{k(k^2 -(4+ 2\sqrt{1+\mu T^2}) k + (4 + 4 \sqrt{1+\mu T^2} + \lambda t^2))}{t^3} \geq 0, \\
        \frac{2 + 4 \sqrt{1+\mu T^2}-3k}{t} \geq 0,
    \end{gather}
    which can be rewritten as
    \begin{gather}
        \frac{k(2+2 \sqrt{1+\mu T^2}-k)}{t^2} \geq 0, \\
        \frac{2k(1+\sqrt{1+\mu T^2}-k)}{t^2} \geq 0, \\
        \frac{k}{t^3}\left(\left(k-\left(1+\sqrt{1+\mu T^2}\right)\right)\left(k-\left(3+\sqrt{1+\mu T^2}\right)\right)+\lambda t^2 - \mu T^2\right) \geq 0, \\
        \frac{3}{t} \left( \frac{2 + 4\sqrt{1+\mu T^2}}{3} - k \right) \geq 0.
    \end{gather}
    Since $f$ is $\mu$-strongly convex, $\lambda \geq \mu$ holds. In addition, we have $\sqrt{1+\mu T^2} \ge 1$. Combining these observations, the conditions reduce to
    \begin{equation}
        0 \leq k \leq 1 + \sqrt{1+\mu T^2}.
    \end{equation}
    Hence $k = 1 + \sqrt{1+\mu T^2}$ satisfies the positive semi-definiteness conditions, yielding the convergence rate
    \begin{equation}
        f(x(t)) - f_* = \mathrm{O} \left(\frac{1}{t^{1 + \sqrt{1+\mu T^2}}} \right).
    \end{equation}
    Observing that
    \begin{equation}
        \begin{split}
            k &= 1 + \sqrt{1+\mu T^2} \\
            &= \frac{1}{2} \left( 1 + 2\sqrt{1+ \mu T^2} \right) + \frac{1}{2} \\
            &= \frac{1}{2} r + \frac{1}{2},
        \end{split}
    \end{equation}
    and noting that $r$ increases with $T$, we conclude that for any $r>3$ the convergence rate
    \begin{equation}
        f(x(t)) - f_* = \mathrm{O} \left(\frac{1}{t^{\frac{1}{2}r+\frac{1}{2}}} \right)
    \end{equation}
    holds.
\end{proof}
A proof based on a Lyapunov function for the present results is provided in Appendix~\ref{sec:pr_thm_NAG_strong_log}.
The result of this theorem is a weaker statement than the known result in Su, Boyd, and Candès (2016)~\cite{su2016differential}, which established the rate
\begin{equation}
    f(x(t)) - f_* = \mathrm{O} \left(\frac{1}{t^{\frac{2}{3}r}} \right)
\end{equation}
for the same system ($r>3$).
We found that this rate cannot be reproduced by the current form of our method. However, by incorporating parts of the techniques used in Su, Boyd, and Candès (2016) into our method, we can derive an extension achieving the same rate from our result above. This extension is presented in Appendix~\ref{sec:appl_SBC_to_ours}.

Finally, keeping the strong-convexity assumption on the objective function $f$ and assuming exponential (linear) convergence 
\begin{equation}
    f - f_* = \mathrm{O}\left( \mathrm{e}^{-kt} \right), \quad \gamma = k t
\end{equation}
with $k>0$, we again find that 2 out of the 10 candidate pairs yield conditional $k>0$. 
Among those two, the pair $(P^{NAG}$ and $Q^{NAG})$ for the matrices in \eqref{PNAG_QNAG} is the most tractable. 
% and also performed well in numerical experiments. 
Therefore we focus on it. 
For these matrices, 
\begin{gather}
    r = \frac{4(k^2 + \mu)}{k^2} \quad \text{and} \quad
    T = \frac{2(k^2 + \mu)}{k^3}
\end{gather}
ensure positive semi-definiteness for any $k>0$ on the time interval
\begin{equation}
    0 < t \leq T.
\end{equation}
Thus the result shows linear convergence up to a finite time $T$, 
The conclusion is summarized in the following theorem.

\begin{thm}[Convergence rate of the continuous dynamical system $\ddot{x} + \frac{4(k^2 + \mu)}{k^2 t} \dot{x} + \nabla f = 0$]
    \label{NAG_strong_exp}
    For any $k>0$ and a $\mu$-strongly convex $f$, the continuous dynamical system $ \ddot{x} + \frac{4(k^2 + \mu)}{k^2 t} \dot{x} + \nabla f = 0 $ has the convergence rate
    \begin{equation}
        f(x(t)) - f_* = \mathrm{O}\left(\mathrm{e}^{-kt}\right)
    \end{equation}
    on the time interval $0 < t \leq T = \frac{2(k^2 + \mu)}{k^3}$.
\end{thm}
\begin{proof}
    Substituting $r = \frac{4(k^2 + \mu)}{k^2}, \gamma = kt$ into $P^{NAG}, Q^{NAG}$ gives
    \begin{equation}
        \label{Lyapunov_NAG_strong_exp}
        P^{NAG} = \frac{1}{2} \begin{pmatrix}
            \frac{4(k^2+\mu)}{kt} - k^2 & 0 & k \\
            0 & 0 & 0 \\
            k & 0 & 1
        \end{pmatrix},~ Q^{NAG} = \begin{pmatrix}
            \frac{k^2(k^2+\lambda)t^2 - 4k(k^2+\mu)t+4(k^2+\mu)}{2kt^2} & 0 & 0 & 0 & 0 \\
            0 & 0 & 0 & 0 & 0 \\
            0 & 0 & \frac{8(k^2+\mu)}{k^2 t} - 3k & 0 & 0 \\
            0 & 0 & 0 & 0 & 0 \\
            0 & 0 & 0 & 0 & 0
        \end{pmatrix}.
    \end{equation}
    The necessary and sufficient conditions for positive semidefiniteness are
    \begin{gather}
        \frac{4(k^2+\mu)}{kt} - k^2 \geq 0, \\
        \left( \frac{4(k^2+\mu)}{kt} - k^2 \right) \cdot 1 - k^2 \geq 0, \\
        \frac{k^2(k^2+\lambda)t^2 - 4k(k^2+\mu)t+4(k^2+\mu)}{2kt^2} \geq 0, \\
        \frac{8(k^2+\mu)}{k^2 t} -3k \geq 0,
    \end{gather}
    which can be rearranged into
    \begin{gather}
        \frac{4(k^2+\mu)}{k} \left(\frac{1}{t} - \frac{k^3}{4(k^2+\mu)}\right) \geq 0, \\
        \frac{4(k^2+\mu)}{k} \left(\frac{1}{t} - \frac{k^3}{2(k^2+\mu)}\right) \geq 0, \\
        \frac{k^2+\mu}{2kt^2} \left( kt -2 \right)^2 + \frac{k}{2} \left( \lambda - \mu \right) \geq 0, \\
        \frac{8(k^2+\mu)}{k^2} \left(\frac{1}{t} - \frac{3k^3}{8(k^2+\mu)}\right) \geq 0.
    \end{gather}
    Since $f$ is $\mu$-strongly convex $\lambda \geq \mu$ holds. Then, combining the inequalities yields
    \begin{equation}
        \frac{1}{t} \geq \frac{k^3}{2(k^2 + \mu)}.
    \end{equation}
    Therefore, for $t$ satisfying
    \begin{equation}
        0 < t \leq T = \frac{2(k^2 + \mu)}{k^3},
    \end{equation}
    the positive semi-definiteness holds and the rate
    \begin{equation}
        f(x(t)) - f_* = \mathrm{O}\left(\mathrm{e}^{-kt}\right)
    \end{equation}
    is guaranteed.
\end{proof}

A proof of the convergence rate based on a Lyapunov function
% , derived by explicitly expanding the matrix representation, 
is included in Appendix~\ref{sec:pr_thm_NAG_strong_exp}.
The above theorem
% , although limited to a specific time interval, 
is useful because the dynamical system corresponding to the standard NAG 
can achieve limited linear convergence for strongly convex objective functions. 
% not only for the continuous dynamical system corresponding to the strongly convex version of Nesterov’s accelerated gradient method, 
% but also for 
Indeed, by combining this theorem with a restart scheme, we can establish linear convergence for the entire time $t \geq 0$.
We show this fact in Section~\ref{sec:restart}. 

\begin{rem}[Incorporating a Hessian term into the NAG-type system]
We consider the continuous dynamical system
\begin{equation}
\ddot{x} + \frac{r}{t}\dot{x} + b \nabla^2 f \dot{x} + \nabla f = 0
\end{equation}
for $b \neq 0$. 
When we analyze its convergence rate using the proposed method, the best possible linear convergence is obtained when
\begin{equation}
r = 0, \quad b = 2\sqrt{\frac{L}{\mu(2L-\mu)}},
\end{equation}
leading to
\begin{equation}
f(x(t)) - f(x_*) = \mathrm{O}\left(\mathrm{e}^{-\sqrt{\frac{\mu L}{2L-\mu}}t}\right).
\end{equation}
This rate coincides with that in Theorem~\ref{SC-NAG} only when $\mu = L$, and is strictly inferior to it when $\mu < L$. Moreover, since the optimal case occurs when $r = 0$, it follows that incorporating the Hessian term into the standard NAG-type dynamical system does not contribute to improving the convergence rate, at least under our proposed method.
\end{rem}

\subsection{Continuous Dynamical System $\ddot{x} + \frac{r}{t^\alpha}\dot{x} + \nabla f = 0$}
\label{sec:generalized_NAG}

The continuous dynamical system
\begin{equation}
    \ddot{x} + \frac{r}{t^\alpha}\dot{x} + \nabla f = 0 \quad (r>0, ~0\leq\alpha\leq1)
\end{equation}
was introduced by Cheng, Liu, and Shang (2025)~\cite{cheng2025class} as a class of continuous dynamical systems including those corresponding to NAG and SC-NAG. Among this class, we focus on the case $0<\alpha<1$.  
Applying the proposed method to this continuous dynamical system, the initial matrices $P^0,~Q^0$ are given by
\begin{equation}
    P^0 = O_3, ~ Q^0 = \begin{pmatrix}
        0 & \frac{\dot{\gamma}}{2} & \frac{r\dot{\gamma}}{2t^\alpha} & 0 & \frac{\dot{\gamma}}{2} \\
        \frac{\dot{\gamma}}{2} & 0 & \frac{1}{2} & 0 & 0 \\
        \frac{r\dot{\gamma}}{2t^\alpha} & \frac{1}{2} & \frac{r}{t^\alpha} & 0 & \frac{1}{2} \\
        0 & 0 & 0 & 0 & 0 \\
        \frac{\dot{\gamma}}{2} & 0 & \frac{1}{2} & 0 & 0
    \end{pmatrix}
\end{equation}
Performing an exhaustive search using the proposed method yields a total of 10 pairs of matrices $(P,Q)$.  
For a $\mu$-strongly convex ($\mu>0$) objective function $f$, following Cheng, Liu, and Shang (2025)~\cite{cheng2025class}, we assume the convergence rate
\begin{equation}
    f - f_* = \mathrm{O}\left(\mathrm{e}^{k \frac{r}{1-\alpha}t^{1-\alpha}}\right), \quad \gamma = k \frac{r}{1-\alpha}t^{1-\alpha},
\end{equation}
where $k$ is a positive constant.  
Under this assumption, two out of the 10 pairs guarantee a convergence rate with $k>0$. One of these corresponds exactly to the Lyapunov function used in Cheng, Liu, and Shang (2025)~\cite{cheng2025class}, reproducing their result (Theorem 1.3.ii) with $k=\tfrac{1}{2}$. The other pair provides a strictly better rate, which we present below:
\begin{equation}
    P = \frac{1}{2}\begin{pmatrix}
        \frac{r\dot{\gamma}}{t^\alpha} & 0 & \dot{\gamma} \\
        0 & 0 & 0 \\
        \dot{\gamma} & 0 & 1
    \end{pmatrix},~ Q = \frac{1}{2}\begin{pmatrix}
        \frac{\left(r\alpha + \lambda t^{1+\alpha}\right)\dot{\gamma} - rt\dot{\gamma}^2-rt\ddot{\gamma}}{t^{1+\alpha}} & 0 & - \dot{\gamma}^2 - \ddot{\gamma} & 0 & 0 \\
        0 & 0 & 0 & 0 & 0 \\
        - \dot{\gamma}^2 - \ddot{\gamma} & 0 & \frac{2r}{t^\alpha} - 3 \dot{\gamma} & 0 & 0 \\
        0 & 0 & 0 & 0 & 0 \\
        0 & 0 & 0 & 0 & 0
    \end{pmatrix}.
\end{equation}
Hereafter, we denote these matrices as $P^{G\text{-}NAG}$ and $Q^{G\text{-}NAG}$.  
The convergence rate guaranteed by this pair is given in the following theorem:
\begin{thm}[Convergence Rate of the Continuous Dynamical System $\ddot{x} + \frac{r}{t^\alpha}\dot{x} + \nabla f = 0$]
    \label{Generalized_NAG}
    For the continuous dynamical system $\ddot{x} + \frac{r}{t^\alpha}\dot{x} + \nabla f = 0$ with a $\mu$-strongly convex function $f$, the convergence rate
    \begin{equation}
        f(x(t)) - f_* = \mathrm{O} \left( \mathrm{e}^{-\left(\frac{2}{3}-\epsilon\right)\frac{r}{1-\alpha}t^{1-\alpha}} \right)
    \end{equation}
    holds for any positive constant $\epsilon$.
\end{thm}

\begin{proof}
    Using a positive constant $k$, we substitute $\gamma = k \frac{r}{1-\alpha}t^{1-\alpha}$ into $P^{G\text{-}NAG}$ and $Q^{G\text{-}NAG}$ to obtain
    \begin{gather}
    \label{Lyapunov_Generalized_NAG}
    P^{G-NAG} = \frac{1}{2}\begin{pmatrix}
        r^2 k t^{-2\alpha} & 0 & rk t^{-\alpha} \\
        0 & 0 & 0 \\
        rk t^{-\alpha} & 0 & 1
    \end{pmatrix},\\ Q^{G-NAG} = \frac{1}{2}\begin{pmatrix}
        \lambda rk t^{-\alpha} - r^3k^2 t^{-3\alpha} + 2r^2\alpha k t^{-1-2\alpha} & 0 & - r^2 k^2 t^{-2\alpha} + r\alpha k t^{-1-\alpha} & 0 & 0 \\
        0 & 0 & 0 & 0 & 0 \\
        - r^2 k^2 t^{-2\alpha} + r\alpha k t^{-1-\alpha} & 0 & 2rt^{-\alpha} - 3rk t^{-\alpha} & 0 & 0 \\
        0 & 0 & 0 & 0 & 0 \\
        0 & 0 & 0 & 0 & 0
    \end{pmatrix}.
    \end{gather}
    The necessary and sufficient conditions for these matrices to be positive semi-definite are
    \begin{gather}
        r^2 k t^{-2\alpha} \geq 0, \\
        r^2 k t^{-2\alpha} \cdot 1 - \left( rk t^{-\alpha} \right)^2 \geq 0, \\
        \lambda rk t^{-\alpha} - r^3k^2 t^{-3\alpha} + 2r^2\alpha k t^{-1-2\alpha} \geq 0, \\
        2rt^{-\alpha} - 3rk t^{-\alpha} \geq 0, \\
        \left(\lambda rk t^{-\alpha} - r^3k^2 t^{-3\alpha} + 2r^2\alpha k t^{-1-2\alpha}\right) \cdot \left(2rt^{-\alpha} - 3rk t^{-\alpha}\right) - \left( - r^2 k^2 t^{-2\alpha} + r\alpha k t^{-1-\alpha} \right)^2 \geq 0,
    \end{gather}
    Simplifying these yields
    \begin{gather}
        r^2 k t^{-2\alpha} \geq 0, \\
        r^2 k (1-k) t^{-2\alpha} \geq 0, \\
        \lambda rk t^{-\alpha} - r^3k^2 t^{-3\alpha} + 2r^2 \alpha k t^{-1-2\alpha} \geq 0, \\
        r(2-3k) t^{-\alpha} \geq 0 \\
        \lambda r^2 k (2-3k) t^{-2\alpha} - r^4 k^2 (k-1)(k-2) t^{-4\alpha} + 2r^3 \alpha k(k-1)(k-2) t^{-1-3\alpha} - r^2 \alpha^2 k^2 t^{-2-2\alpha} \geq 0.
    \end{gather}
    If the coefficients of the leading terms in $t$ are positive on the left-hand sides above, these inequalities hold for $t \geq T$, where $T>0$ is a certain sufficiently large number. 
    The conditions for the coefficients are given by
    \begin{gather}
        r^2 k > 0, \\
        r^2 k ( 1 - k ) > 0, \\
        \lambda r k > 0, \\
        r ( 2 - 3k ) > 0, \\
        \lambda r^2 k ( 2 - 3k ) > 0.
    \end{gather}
    Because the objective function $f$ is assumed to be $\mu$-strongly convex, we have $\lambda \geq \mu > 0$. Since $r>0$, the above  conditions are satisfied when
    \begin{equation}
        0 < k < \frac{2}{3}. 
    \end{equation}
    If positive semi-definiteness holds for $t \geq T$, then it holds that
    \begin{equation}
        \mathrm{e}^{\gamma}(f-f_*) \leq \mathcal{E}(t) \leq \mathcal{E}(T) = \mathrm{const.}
    \end{equation}
    Thus, the convergence rate can be derived in the same way as if it held for all $t>0$. Therefore, since positive semi-definiteness holds for $0 < k < \frac{2}{3}$ with $t \geq T$, by taking $k=\frac{2}{3}-\epsilon$ for any positive constant $\epsilon$ (assumed small), 
    we can guarantee the convergence rate
    \begin{equation}
        f(x(t)) - f_* = \mathrm{O} \left( \mathrm{e}^{-\left(\frac{2}{3}-\epsilon\right)\frac{r}{1-\alpha}t^{1-\alpha}} \right).
    \end{equation}
\end{proof}
The theorem obtained here provides a convergence rate that strictly improves upon that of Cheng, Liu, and Shang (2025)~\cite{cheng2025class} (with $k=\tfrac{1}{2}$) for exactly the same continuous dynamical system. 
Therefore it is one of the contributions of the proposed method. The detailed proof of the convergence rate using the Lyapunov function with the expanded matrix representation is included in Appendix~\ref{sec:pr_thm_Generalized_NAG}. 
% In addition, the other pair of matrices which guarantees the convergence rate with $k>0$ coincides with the Lyapunov function used in Cheng, Liu, and Shang (2025)~\cite{cheng2025class}.

\section{Restart Scheme}
\label{sec:restart}

In this section, using Theorem~\ref{NAG_strong_exp}, we propose a restart scheme for strongly convex functions and analyze its convergence rate. Specifically, we consider the continuous dynamical system 
\begin{equation}
\ddot{x} + \frac{4(k^2+\mu)}{k^2t}\dot{x} + \nabla f = 0
\end{equation}
with a constant $k > 0$.
% For a positive constant $k$, 
We define another positive constant $l$ by
\begin{equation}
k = l\sqrt{\mu}.
\end{equation}
Then, the above continuous dynamical system can be rewritten as
\begin{equation}
\ddot{x} + \frac{4(l^2+1)}{l^2t}\dot{x} + \nabla f = 0,
\label{eq:ext_NAG_param_l}
\end{equation}
and Theorem~\ref{NAG_strong_exp} can be restated as follows.

\begin{cor}[Convergence rate of the continuous dynamical system $\ddot{x} + \frac{4(l^2 + \mu)}{l^2 t} \dot{x} + \nabla f = 0$]
\label{thm:restated_NAG_strong_exp}
For a $\mu$-strongly convex function $f$,
the continuous dynamical system $\ddot{x} + \frac{4(l^2 + 1)}{l^2 t} \dot{x} + \nabla f = 0~(l>0)$ achieves the convergence rate
\begin{equation}
f(x(t)) - f_* = \mathrm{O}\left(\mathrm{e}^{-l\sqrt{\mu}t}\right)
\end{equation}
for $0 < t \leq T = \frac{2(l^2 + 1)}{l^3 \sqrt{\mu}}$.
This assertion is derived from the fact that the function
\begin{equation}
\mathcal{E}(t) = l^2 \left( \frac{T}{t} - 1 \right) \mathrm{e}^{l\sqrt{\mu}t} \|x(t)-x_*\|^2 + \frac{1}{2} \mathrm{e}^{l\sqrt{\mu}t} \left\| \dot{x}(t) + l\sqrt{\mu} (x(t)-x_*)\right\|^2 + \mathrm{e}^{l\sqrt{\mu}t} (f(x(t))-f_*)
\end{equation}
is non-increasing. 
This function is the rewritten form of the Lyapunov function in \eqref{eq:E_of_NAG_strong_exp} 
used for showing Theorem~\ref{NAG_strong_exp} in Appendix~\ref{sec:pr_thm_NAG_strong_exp}. 
\end{cor}

We consider a restart scheme for the dynamical system in~\eqref{eq:ext_NAG_param_l}. 
Let $c > 1$ be a constant. 
First, 
we evolve $x$ in time from $t = \frac{T}{c}$ to $t = T$ according to the dynamical system. 
Next, 
after reaching $t=T$, we restart the evolution of $x$ from $t=\frac{T}{c}$ using the same dynamics. 
We refer to this entire process as a round and repeat this thereafter.
For this restart scheme, 
we can show the following theorem
by using Corollary~\ref{thm:restated_NAG_strong_exp}.

\begin{thm}[Convergence rate of the restart scheme for the continuous dynamical system $\ddot{x} + \frac{4(l^2+1)}{l^2t}\dot{x} + \nabla f = 0$]
\label{restart}
Let $l > 0$ and $c > 1$ be constants satisfying 
\begin{equation}
    \left(2(c-1)l^2 + 1\right) \mathrm{e}^{\left(-1+\frac{1}{c}\right)\frac{2(l^2+1)}{l^2}} 
    \leq 1.
    \label{eq:c_l_assumption}
\end{equation}
For a $\mu$-strongly convex objective function $f$ and a constant $c>1$, 
let $x$ be a solution trajectory of the aforementioned restart scheme based on the continuous dynamical system $\ddot{x} + \frac{4(l^2+1)}{l^2t}\dot{x} + \nabla f = 0$. Then, we have
\begin{equation}
f(x(t))-f_* + \frac{1}{2} \left\|\dot{x}(t) + l \sqrt{\mu} (x(t)-x_*)\right\|^2 \leq C_{c,l} h(c,l)^{\sqrt{\mu}t} \left( f(x_0)-f_* + \frac{1}{2} \left\|\dot{x}_0 + l \sqrt{\mu} (x_0-x_*)\right\|^2 \right),
\end{equation}
where
\begin{gather}
h(c,l) = \left(\left(2(c-1)l^2 + 1\right) \mathrm{e}^{\left(-1+\frac{1}{c}\right)\frac{2(l^2+1)}{l^2}} \right)^{\frac{cl^3}{2(c-1)(l^2+1)}}, \\
C_{c,l} = \left(2(c-1)l^2 + 1\right) \mathrm{e}^{-\left(-1+\frac{1}{c}\right)\frac{2(l^2+1)}{l^2}}.
\end{gather}
\end{thm}

\begin{proof}
Let $i$ be a positive integer. For the $i$-th round, define the Lyapunov function according to Corollary~\ref{thm:restated_NAG_strong_exp} as
\begin{equation}
\mathcal{E}_{i}(t) = l^2 \mu \left( \frac{T}{t} - 1 \right) \mathrm{e}^{l\sqrt{\mu}t} \|x(t)-x_*\|^2 + \frac{1}{2} \mathrm{e}^{l\sqrt{\mu}t} \left\| \dot{x}(t) + l\sqrt{\mu} (x(t)-x_*)\right\|^2 + \mathrm{e}^{l\sqrt{\mu}t} (f(x(t))-f_*).
\end{equation}
In the $i$-th round, we denote the position, velocity, and function value at time $t=\frac{T}{c}$ as
\begin{equation}
x_{i-1} = x\left(\frac{T}{c}\right), \ 
\dot{x}_{i-1} = \dot{x}\left(\frac{T}{c}\right), \  
f_{i-1} = f\left(x\left(\frac{T}{c}\right)\right),
\end{equation}
and those at time $t=T$ as
\begin{equation}
x_i = x(T), \ 
\dot{x}_i = \dot{x}(T), \ 
f_i = f(x(T)),
\end{equation}
respectively. 
Furthermore, for $j=0,1,\cdots$, 
we define
\begin{equation}
g_j =f_j - f_* + \frac{1}{2} \|\dot{x}_j + l \sqrt{\mu} (x_j - x_*)\|^2.
\end{equation}

First, we consider the $i$-th round. At $t=\frac{T}{c}$, we have
\begin{equation}
    \begin{split}
        \mathcal{E}_i\left(\frac{T}{c}\right) &= l^2 \mu (c-1) \mathrm{e}^{\frac{l\sqrt{\mu}T}{c}} \|x_{i-1}-x_*\|^2 + \frac{1}{2} \mathrm{e}^{\frac{l\sqrt{\mu}T}{c}} \|\dot{x}_{i-1}+l\sqrt{\mu}(x_{i-1}-x_*)\|^2 + \mathrm{e}^{\frac{l\sqrt{\mu}T}{c}} (f_{i-1}-f_*) \\
        &= l^2 \mu (c-1) \mathrm{e}^{\frac{l\sqrt{\mu}T}{c}} \|x_{i-1}-x_*\|^2 + \mathrm{e}^{\frac{l\sqrt{\mu}T}{c}} g_{i-1}.
    \end{split}
\end{equation}
Since $f$ is $\mu$-strongly convex, we have
\begin{equation}
    \begin{split}
        g_{i-1} &= f_{i-1} - f_* + \frac{1}{2}\|\dot{x}_{i-1}+l\sqrt{\mu}(x_{i-1}-x_*)\|^2 \\
        &\geq f_{i-1} - f_* \\
        &\geq \frac{\mu}{2} \|x_{i-1}-x_*\|^2,
    \end{split}
\end{equation}
which implies
\begin{equation}
    \|x_{i-1}-x_*\|^2 \leq \frac{2}{\mu} g_{i-1}.
\end{equation}
Hence, at $t=\frac{T}{c}$, we have
\begin{equation}
    \begin{split}
        \mathcal{E}_i\left(\frac{T}{c}\right) &= l^2 \mu (c-1) \mathrm{e}^{\frac{l\sqrt{\mu}T}{c}} \|x_{i-1}-x_*\|^2 + \mathrm{e}^{\frac{l\sqrt{\mu}T}{c}} g_{i-1} \\
        &\leq l^2 \mu (c-1) \mathrm{e}^{\frac{l\sqrt{\mu}T}{c}} \cdot \frac{2}{\mu} g_{i-1} + \mathrm{e}^{\frac{l\sqrt{\mu}T}{c}} g_{i-1} \\
        &= 2l^2 (c-1) \mathrm{e}^{\frac{l\sqrt{\mu}T}{c}}  g_{i-1} + \mathrm{e}^{\frac{l\sqrt{\mu}T}{c}} g_{i-1} \\
        &= (2l^2 (c-1) +1) \mathrm{e}^{\frac{l\sqrt{\mu}T}{c}}  g_{i-1}.
    \end{split}
    \label{eq:E_i_T_c}
\end{equation}
On the other hand, at $t=T$, the equality
\begin{equation}
    \begin{split}
        \mathcal{E}_i(T) &= \frac{1}{2} \mathrm{e}^{l\sqrt{\mu}T} \|\dot{x}_i + l\sqrt{\mu}(x_i-x_*)\|^2 + \mathrm{e}^{l\sqrt{\mu}T} (f_i-f_*) \\
        &= \mathrm{e}^{l\sqrt{\mu}T} g_i
    \end{split}
    \label{eq:E_i_T}
\end{equation}
holds.
By combining \eqref{eq:E_i_T_c}, \eqref{eq:E_i_T}, and Corollary~\ref{thm:restated_NAG_strong_exp}, we have
\begin{equation}
    \mathrm{e}^{l\sqrt{\mu}T}g_i = \mathcal{E}_i(T) \leq \mathcal{E}_i \left(\frac{T}{c}\right) \leq (2l^2 (c-1) +1) \mathrm{e}^{\frac{l\sqrt{\mu}T}{c}}  g_{i-1}, 
    \label{eq:key_ineq_for_eg_i}
\end{equation}
where the first inequality is owing to Corollary~\ref{thm:restated_NAG_strong_exp} showing that the function $\mathcal{E}_i(t)$ is non-increasing for $0<t\leq T$. 
Dividing both sides of \eqref{eq:key_ineq_for_eg_i} by $\mathrm{e}^{l\sqrt{\mu}T}$ and 
substituting $T=\frac{2(l^2+1)}{l^3\sqrt{\mu}}$ give
\begin{align}
    g_i & \leq (2l^2 (c-1) +1) \mathrm{e}^{ \left( -1 + \frac{1}{c}\right)l\sqrt{\mu}T}  g_{i-1}
    \notag \\
& \leq (2l^2 (c-1) +1) \mathrm{e}^{ \left( -1 + \frac{1}{c}\right)\frac{2(l^2+1)}{l^2}}  g_{i-1}.
\end{align}
Since this holds for any positive integer $i$, recursively applying this inequality gives
\begin{equation}
    g_i \leq \left( (2l^2 (c-1) +1) \mathrm{e}^{ \left( -1 + \frac{1}{c}\right)\frac{2(l^2+1)}{l^2}} \right)^i  g_{0}.
\end{equation}

Next, we consider the number of rounds that can be performed up to time $t$. The duration of one round is
\begin{equation}
    T - \frac{T}{c} = \frac{c-1}{c} T = \frac{2(c-1)(l^2+1)}{cl^3\sqrt{\mu}}.
\end{equation}
Therefore, the number of rounds completed by time $t$ is
\begin{equation}
    \left\lfloor \frac{t}{\frac{2(c-1)(l^2+1)}{cl^3\sqrt{\mu}}} \right\rfloor = \left\lfloor \frac{cl^3}{2(c-1)(l^2+1)}\sqrt{\mu}t \right\rfloor.
\end{equation}

Combining the above results, we obtain
\begin{equation}
\begin{split}
    g_i &\leq \left( (2l^2 (c-1) +1) \mathrm{e}^{ \left( -1 + \frac{1}{c}\right)\frac{2(l^2+1)}{l^2}} \right)^{\left\lfloor \frac{cl^3}{2(c-1)(l^2+1)}\sqrt{\mu}t\right\rfloor}  g_{0} \\
    &\leq \left( (2l^2 (c-1) +1) \mathrm{e}^{ \left( -1 + \frac{1}{c}\right)\frac{2(l^2+1)}{l^2}} \right)^{\frac{cl^3}{2(c-1)(l^2+1)}\sqrt{\mu}t -1}  g_{0},
\end{split}
\end{equation}
where the second inequality is owing to \eqref{eq:c_l_assumption}.
Thus we have the conclusion:
\begin{equation}
    f(x(t))-f_* + \frac{1}{2} \left\|\dot{x}(t) + l \sqrt{\mu} (x(t)-x_*)\right\|^2 \leq C_{c,l} h(c,l)^{\sqrt{\mu}t} \left( f(x_0)-f_* + \frac{1}{2} \left\|\dot{x}_0 + l \sqrt{\mu} (x_0-x_*)\right\|^2 \right).
\end{equation}
\end{proof}

From the above theorem, if there exists a pair $(c,l)$ such that $h(c,l)<1$, then the restart scheme achieves linear convergence. Although the minimum value of $h(c,l)$ has not yet been determined, we have found at least one example for which it is smaller than $1$:
\begin{gather}
h\left(2,\frac{1}{\sqrt{2}}\right) = 2^{\frac{1}{3\sqrt{2}}} \mathrm{e}^{-\frac{\sqrt{2}}{{2}}} \approx 0.580578, \\
C_{2,\frac{1}{\sqrt{2}}} = \frac{\mathrm{e}^3}{2}.
\end{gather}
Then, the restart scheme 
for $(c,l) = (2, \frac{1}{\sqrt{2}})$
achieves linear convergence.

\begin{cor}[An example of the restart scheme]
\label{restart_example}
Consider the restart scheme based on the continuous dynamical system
\begin{equation}
\ddot{x} + \frac{12}{t} \dot{x} + \nabla f = 0,
\end{equation}
which evolves the system from $t=3\sqrt{\frac{2}{\mu}}$ to $t=6\sqrt{\frac{2}{\mu}}$ and restart it after reaching the terminal time.
For a $\mu$-strongly convex objective function $f$, we have
\begin{equation}
f(x(t))-f_* + \frac{1}{2} \left\|\dot{x}(t) + l \sqrt{\mu} (x(t)-x_*)\right\|^2 \leq \frac{\mathrm{e}^3}{2} 0.580579^{\sqrt{\mu} t} \left( f(x_0)-f_* + \frac{1}{2} \left\|\dot{x}_0 + l \sqrt{\mu} (x_0-x_*)\right\|^2 \right).
\end{equation}
\end{cor}

From the above results, we conclude that introducing a restart scheme into the standard NAG-type continuous dynamical system yields linear convergence. While similar results are observed in Su, Boyd, and Candès~(2016)~\cite{su2016differential}, our approach demonstrates convergence under only the $\mu$-strong convexity assumption, without requiring $L$-smoothness. 

\section{Conclusion}
\label{sec:conclusion}

In this paper, 
we focused on the method of the Lyapunov functions, one of the proof techniques for convergence rates of continuous dynamical systems corresponding to optimization methods.
% for a particular class of objective functions. 
We worked on their systematic construction and derived convergence rates via symbolic computation on a computer. Our method in this study is an extension of Suh, Roh, Ryu~(2022)~\cite{suh2022continuous} and Kamijima et al.~(2024)~\cite{TomoyaKamijima2024}. It eliminates the arbitrariness in the choice of integration by parts observed in their construction of the Lyapunov functions by exhaustive search. At the same time,  introducing the matrix representation enables examining a wider range of candidates of the Lyapunov functions. Applying the proposed method to the six classes of continuous dynamical systems, we optimized their parameters and obtained their convergence rates.
This capability of optimizing parameters also constitutes a contribution of this work inherited from Kamijima et al.~(2024)~\cite{TomoyaKamijima2024}.

Among the obtained rates, Theorem~\ref{First_Order} gives the most remarkable one. Even without the second derivative term $\ddot{x}$, the introduction of the Hessian created the possibility of achieving better convergence rates than gradient flow corresponding to the gradient descent method. 
Additionally, 
the rate obtained in Theorem~\ref{Second_Order} suggests the potential for acceleration through the introduction of the Hessian, 
although the rate coincides with that of the dynamical system corresponding to the SC-NAG. 
Therefore this also represents a significant contribution of this work. 
Furthermore, 
Theorem~\ref{Generalized_NAG} provides better convergence rates than the prior study, which is also a noteworthy contribution. 
Also, 
Theorem~\ref{restart} shows that the restart scheme achieves a good convergence rate with a relatively weak assumption. 

There remain many future challenges. Throughout this paper, we have focused solely on the convergence rate of $f-f_*$. However, as in Upadhyaya et al.~(2025)~\cite{upadhyaya2025autolyap}, a possible direction of future work would be to evaluate convergence rates for various performance measures. 
For instance, 
suppose that we want to evaluate the convergence rate of
\begin{equation}
    f - f_* - \frac{\mu - \epsilon}{2} \|x-x_*\|^2.
    \label{eq:another_perf_meas}
\end{equation}
as in Ushiyama~(2025)~\cite{ushiyama2025sqrt2acceleratedfistacompositestrongly}. We can evaluate its rate by using our method if we modify the initial matrix $P$ as
\begin{equation}
    P^0 = \begin{pmatrix}
        \frac{\mu - \epsilon}{2} & 0 & 0 \\
        0 & 0 & 0 \\
        0 & 0 & 0
    \end{pmatrix}.
\end{equation} 
In addition, 
in the case that we ignore $f - f_{\ast}$ in \eqref{eq:another_perf_meas} and just intend to evaluate $\| x - x_{\ast} \|^{2}$, we have only to skip Operation A1. 
Many other extensions can also be handled by modifying the initial matrices or the combinations of operations.

Moreover, the semi-definite conditions on $P$ and $Q$ are merely sufficient conditions for establishing convergence rates. As shown in Appendix~\ref{sec:E_not_decrease_by_SBC}, 
Su, Boyd, Candes~(2016)~\cite{su2016differential} presents a technique to derive the convergence rates of some dynamical systems even when the matrices $P$ and $Q$ are not positive semi-definite. 
Their technique is based on the observation that the boundedness of a counterpart of the function $\mathcal{E}(t)$ is sufficient to derive the convergence rates. 
% because the vectors $v_1$, $v_2$, $v_3$, $v_4$, and $v_5$ are not entirely independent. 
Therefore another direction of future research can be to incorporate this technique into the proposed method. This could enable us to develop a systematic scheme capable of obtaining better convergence rates from the same function. 

From the perspective of broadening the search range, we can generalize the function
\begin{equation}
    w = \mathrm{e}^{\gamma}(x-x_*) 
    % ~ \dot{w} = \mathrm{e}^{\gamma}(\dot{\gamma}(x-x_*) + \dot{x}),
    \notag
\end{equation}
in~\eqref{eq:def_func_w} used at the beginning of our method. Since the form of this function is not mandatory, we can consider several alternatives like
\begin{gather}
    u = \mathrm{e}^{\gamma}(x-x_* + \delta\nabla f), 
    \notag \\
    % ~ \dot{u} = \mathrm{e}^{\gamma}(\dot{\gamma}(x-x_*) + (\dot{\gamma}\delta + \dot{\delta}) \nabla f + \dot{x} + \delta \nabla^2 f \dot{x}), \\
    y = \mathrm{e}^{\gamma}(x-x_* + \delta\nabla f + \epsilon \dot{x}), 
    % ~ \dot{y} = \mathrm{e}^{\gamma}(\dot{\gamma}(x-x_*) + (\dot{\gamma}\delta + \dot{\delta}) \nabla f + (1 + \dot{\gamma}\epsilon + \dot{\epsilon}) \dot{x} + \delta \nabla^2 f \dot{x} + \epsilon \ddot{x}),
    \notag
\end{gather}
using $\delta(t)$ and $\epsilon(t)$.
Indeed, we tested these functions for our method. However, no significant improvement has not been obtained yet. 
Therefore improvement by extending the function $w$ can be a topic for future work. 

Finally, another major task for future work
is to create a time-discrete version of our method. 
This is because our ultimate goal is to develop fast optimization algorithms with the help of continuous dynamical systems. 
To this end, 
we can consider two approaches.
The first approach is to propose a method of exhaustive search for discrete Lyapunov functions by symbolic computation. 
We hope that we can develop a similar procedure by replacing the proposed operations with their discrete counterparts. For example, integration by parts will be replaced with summation by parts. 
The second approach is discretization of continuous dynamical systems that preserves their convergence rates. 
Since not all discretizations preserve the rate, this is not a trivial task. 

\section*{Acknowledgments}

The authors would like to express their sincere gratitude to members of the Department of Mathematical Informatics, 
Graduate School of Information Science and Technology, 
The University of Tokyo for their helpful suggestions and encouragement. They are especially grateful to Mr.~Kansei Ushiyama, Professor Takayasu Matsuo, Assistant Professor Hiroyuki Miyoshi, and Mr. Tomoya Kamijima. Furthermore, the authors would like to thank Associate Professor Shun Sato of the Graduate School of Science, Tokyo Metropolitan University, 
for his insightful comments, especially regarding the restart scheme. 
The second author was supported by JSPS KAKENHI Grant Number JP24K00536 (Grant-in-Aid for Scientific Research (B)).

\bibliographystyle{plain}% BibTeX
\bibliography{reference}% BibTeX

\appendix

\section{List of Operations}
\label{sec:list_of_oper}

In this section, we first group and enumerate the possible operations. Each operation is defined so as not to destroy the symmetry of the matrices $P^{k}$ and $Q^{k}$. The $(i,j)$ entries of $P^k$ and $Q^k$ are denoted by $P_{ij}^k$ and $Q_{ij}^k$ respectively.

\subsection{Integration by Parts Involving $Q_{11},Q_{13},Q_{15},Q_{33},Q_{35}$}
\label{sec:oper_Group_B}

%From here, we introduce integration-by-parts operations. 
We show the operations that handle coefficients of terms expressed as inner products among $v_1 = x - x_*$, $v_3 = \dot{x}$, and $v_5 = \ddot{x}$.
\begin{itemize}
    \item Operation B1: Integration by parts for $\langle v_3,v_5 \rangle$
        \begin{equation}
            \begin{split}
                &\int_{t_0}^t \mathrm{e}^\gamma \cdot 2 Q_{35}^k \langle v_3, v_5 \rangle \mathrm{d}s \\
                &= \int_{t_0}^t \mathrm{e}^\gamma \cdot Q_{35}^k (2\langle \dot{x}, \ddot{x} \rangle) \mathrm{d}s \\
                &= \int_{t_0}^t \mathrm{e}^\gamma \cdot Q_{35}^k ( \|\dot{x}\|^2)^{\prime} \mathrm{d}s \\
                &= \mathrm{e}^\gamma \cdot Q_{35}^k \|\dot{x}\|^2 - \int_{t_0}^t \mathrm{e}^\gamma \cdot (\dot{\gamma} Q_{35}^k + \dot{Q_{35}^k}) \|\dot{x}\|^2 \mathrm{d}s +\mathrm{const.}\\
                &= \mathrm{e}^\gamma \cdot Q_{35}^k \|v_3\|^2 - \int_{t_0}^t \mathrm{e}^\gamma \cdot (\dot{\gamma} Q_{35}^k + \dot{Q_{35}^k}) \|v_3\|^2 \mathrm{d}s +\mathrm{const.}
            \end{split}
        \end{equation}
        Thus, applying Operation B1 to $(P_k,Q_k)$ yields
        \begin{gather}
            P^{k+1} = P^k + \begin{pmatrix}
                0 & 0 & 0 \\
                0 & 0 & 0 \\
                0 & 0 & Q_{35}^k
            \end{pmatrix},\\
            Q^{k+1} = Q^k - \begin{pmatrix}
                0 & 0 & 0 & 0 & 0 \\
                0 & 0 & 0 & 0 & 0 \\
                0 & 0 & \dot{\gamma} Q_{35}^k + \dot{Q_{35}^k} & 0 & Q_{35}^k \\
                0 & 0 & 0 & 0 & 0 \\
                0 & 0 & Q_{35}^k & 0 & 0
            \end{pmatrix}.
        \end{gather}
    \item Operation B2: Integration by parts for $\langle v_1,v_5 \rangle$
        \begin{equation}
            \begin{split}
                &\int_{t_0}^t \mathrm{e}^\gamma \cdot2Q_{15}^k \langle v_1, v_5 \rangle \mathrm{d}s \\
                &= \int_{t_0}^t \mathrm{e}^\gamma \cdot2Q_{15}^k \langle x-x_*, \ddot{x} \rangle \mathrm{d}s \\
                &= \int_{t_0}^t \mathrm{e}^\gamma \cdot 2Q_{15}^k (\langle x-x_*, \dot{x} \rangle)^\prime \mathrm{d}s - \int_{t_0}^t \mathrm{e}^\gamma \cdot 2Q_{15}^k \|\dot{x}\|^2 \mathrm{d}s \\
                &= \mathrm{e}^\gamma \cdot 2 Q_{15}^k \langle x-x_*, \dot{x} \rangle - \int_{t_0}^t \mathrm{e}^\gamma \cdot 2Q_{15}^k \|\dot{x}\|^2 \mathrm{d}s \\
                &\phantom{=} - \int_{t_0}^t \mathrm{e}^\gamma \cdot 2(\dot{\gamma}Q_{15}^k + \dot{Q_{15}^k}) \langle x-x_*, \dot{x} \rangle \mathrm{d}s +\mathrm{const.}\\
                &= \mathrm{e}^\gamma \cdot 2 Q_{15}^k \langle v_1, v_3 \rangle - \int_{t_0}^t \mathrm{e}^\gamma \cdot 2Q_{15}^k \|v_3\|^2 \mathrm{d}s \\
                &\phantom{=} - \int_{t_0}^t \mathrm{e}^\gamma \cdot 2(\dot{\gamma}Q_{15}^k + \dot{Q_{15}^k}) \langle v_1, v_3 \rangle \mathrm{d}s +\mathrm{const.}
            \end{split}
        \end{equation}
        Thus, applying Operation B2 to $(P_k,Q_k)$ yields
        \begin{gather}
            P^{k+1} = P^k + \begin{pmatrix}
                0 & 0 & Q_{15}^k \\
                0 & 0 & 0 \\
                Q_{15}^k & 0 & 0
            \end{pmatrix}, \\
            Q^{k+1} = Q^k - \begin{pmatrix}
                0 & 0 & \dot{\gamma} Q_{15}^k + \dot{Q_{15}^k} & 0 & Q_{15}^k \\
                0 & 0 & 0 & 0 & 0 \\
                \dot{\gamma} Q_{15}^k + \dot{Q_{15}^k} & 0 & 2Q_{15}^k & 0 & 0 \\
                0 & 0 & 0 & 0 & 0 \\
                Q_{15}^k & 0 & 0 & 0 & 0
            \end{pmatrix}.
        \end{gather}
    \item Operation B3: Integration by parts for $\langle v_1,v_3 \rangle$
        \begin{equation}
            \begin{split}
                &\int_{t_0}^t \mathrm{e}^\gamma \cdot 2Q_{13}^k \langle v_1, v_3 \rangle \mathrm{d}s \\
                &= \int_{t_0}^t \mathrm{e}^\gamma \cdot Q_{13}^k (2\langle x-x_*, \dot{x} \rangle)  \mathrm{d}s \\
                &= \int_{t_0}^t \mathrm{e}^\gamma \cdot Q_{13}^k (\|x-x_*\|^2)^\prime \mathrm{d}s \\
                &= \mathrm{e}^\gamma \cdot Q_{13}^k \|x-x_*\|^2 - \int_{t_0}^t \mathrm{e}^\gamma \cdot (\dot{\gamma} Q_{13}^k + \dot{Q_{13}^k})\|x-x_*\|^2 \mathrm{d}s +\mathrm{const.}\\
                &= \mathrm{e}^\gamma \cdot Q_{13}^k \|v_1\|^2 - \int_{t_0}^t \mathrm{e}^\gamma \cdot (\dot{\gamma} Q_{13}^k + \dot{Q_{13}^k})\|v_1\|^2 \mathrm{d}s +\mathrm{const.}
            \end{split}
        \end{equation}
        Thus, applying Operation B3 to $(P_k,Q_k)$ yields
        \begin{gather}
            P^{k+1} = P^k + \begin{pmatrix}
                Q_{13}^k & 0 & 0 \\
                0 & 0 & 0 \\
                0 & 0 & 0
            \end{pmatrix},\\
            Q^{k+1} = Q^k - \begin{pmatrix}
                \dot{\gamma}Q_{13}^k+\dot{Q_{13}^k} & 0 & Q_{13}^k & 0 & 0 \\
                0 & 0 & 0 & 0 & 0 \\
                Q_{13}^k & 0 & 0 & 0 & 0 \\
                0 & 0 & 0 & 0 & 0 \\
                0 & 0 & 0 & 0 & 0
            \end{pmatrix}.
        \end{gather}
\end{itemize}

\subsection{Integration by Parts Involving $Q_{22},Q_{24}$}
\label{sec:oper_Group_C}

We show the operations related to the coefficients of terms expressed by the inner product of $v_2=\nabla f$ and $v_4 = \nabla^2 f \dot{x}$.
\begin{itemize}
    \item Operation C1: Integration by parts for $\langle v_2, v_4 \rangle$
        \begin{equation}
            \begin{split}
                &\int_{t_0}^t \mathrm{e}^\gamma \cdot 2Q_{24}^k \langle v_2, v_4 \rangle \,\mathrm{d}s \\
                &= \int_{t_0}^t \mathrm{e}^\gamma \cdot Q_{24}^k \cdot (2 \langle \nabla f, \nabla^2 f \dot{x} \rangle ) \,\mathrm{d}s \\
                &= \int_{t_0}^t \mathrm{e}^\gamma \cdot Q_{24}^k \cdot (\|\nabla f\|^2)^{\prime} \,\mathrm{d}s \\
                &= \mathrm{e}^\gamma Q_{24}^k \|\nabla f\|^2 - \int_{t_0}^t \mathrm{e}^\gamma \cdot (\dot{\gamma} Q_{24}^k + \dot{Q_{24}^k}) \|\nabla f\|^2 \,\mathrm{d}s +\mathrm{const.}\\
                &= \mathrm{e}^\gamma Q_{24}^k \|v_2\|^2 - \int_{t_0}^t \mathrm{e}^\gamma \cdot (\dot{\gamma} Q_{24}^k + \dot{Q_{24}^k}) \|v_2\|^2 \,\mathrm{d}s +\mathrm{const.}
            \end{split}
        \end{equation}
 
        Therefore, applying Operation C1 to $(P^k,Q^k)$ gives
        \begin{gather}
            P^{k+1} = P^k + \begin{pmatrix}
                0 & 0 & 0 \\
                0 & Q_{24}^k & 0 \\
                0 & 0 & 0
            \end{pmatrix}, \\
            Q^{k+1} = Q^k - \begin{pmatrix}
                0 & 0 & 0 & 0 & 0 \\
                0 & \dot{\gamma}Q_{24}^k+\dot{Q_{24}^k} & 0 & Q_{24}^k & 0 \\
                0 & 0 & 0 & 0 & 0 \\
                0 & Q_{24}^k & 0 & 0 & 0 \\
                0 & 0 & 0 & 0 & 0
            \end{pmatrix}.
        \end{gather}
\end{itemize}

\subsection{Integration by Parts Involving $Q_{12},Q_{14},Q_{23},Q_{25},Q_{34}$}
\label{sec:oper_Group_D}

Here we show the most involved integration-by-parts operations, corresponding to the coefficients of the terms expressed by the inner products of $v_1=x-x_*$, $v_3=\dot{x}$, and $v_5=\ddot{x}$ with $v_2=\nabla f$ and $v_4 = \nabla^2 f \dot{x}$.
\begin{itemize}
    \item Operation D1: Integration by parts for $\langle v_3, v_4 \rangle$
        \begin{equation}
            \begin{split}
                &\int_{t_0}^t \mathrm{e}^\gamma \cdot 2Q_{34}^k \langle v_3, v_4 \rangle \,\mathrm{d}s \\
                &= \int_{t_0}^t \mathrm{e}^\gamma \cdot 2Q_{34}^k \langle \dot{x},\nabla^2 f \dot{x} \rangle \,\mathrm{d}s \\
                &= \int_{t_0}^t \mathrm{e}^\gamma \cdot 2Q_{34}^k (\langle \nabla f, \dot{x} \rangle)^\prime \,\mathrm{d}s - \int_{t_0}^t \mathrm{e}^\gamma \cdot 2Q_{34}^k \langle \nabla f, \ddot{x} \rangle \,\mathrm{d}s \\
                &= \mathrm{e}^\gamma \cdot 2Q_{34}^k \langle \nabla f, \dot{x} \rangle - \int_{t_0}^t \mathrm{e}^\gamma \cdot 2(\dot{\gamma}Q_{34}^k + \dot{Q_{34}^k}) \langle \nabla f, \dot{x} \rangle \,\mathrm{d}s \\
                &\phantom{=} - \int_{t_0}^t \mathrm{e}^\gamma \cdot 2Q_{34}^k \langle \nabla f, \ddot{x} \rangle \,\mathrm{d}s +\mathrm{const.}\\
                &= \mathrm{e}^\gamma \cdot 2Q_{34}^k \langle v_2, v_3 \rangle - \int_{t_0}^t \mathrm{e}^\gamma \cdot 2(\dot{\gamma}Q_{34}^k + \dot{Q_{34}^k}) \langle v_2, v_3 \rangle \,\mathrm{d}s \\
                &\phantom{=} - \int_{t_0}^t \mathrm{e}^\gamma \cdot 2Q_{34}^k \langle v_2,v_5 \rangle \,\mathrm{d}s +\mathrm{const.}
            \end{split}
        \end{equation}
        Therefore, applying Operation D1 to $(P^k,Q^k)$ gives
        \begin{gather}
            P^{k+1} = P^k + \begin{pmatrix}
                0 & 0 & 0 \\
                0 & 0 & Q_{34}^k \\
                0 & Q_{34}^k & 0
            \end{pmatrix}, \\
            Q^{k+1} = Q^k - \begin{pmatrix}
                0 & 0 & 0 & 0 & 0 \\
                0 & 0 & \dot{\gamma}Q_{34}^k+\dot{Q_{34}^k} & 0 & Q_{34}^k \\
                0 & \dot{\gamma}Q_{34}^k+\dot{Q_{34}^k} & 0 & Q_{34}^k & 0 \\
                0 & 0 & Q_{34}^k & 0 & 0 \\
                0 & Q_{34}^k & 0 & 0 & 0
            \end{pmatrix}.
        \end{gather}

    \item Operation D2: Integration by parts for $\langle v_2, v_5 \rangle$
        \begin{equation}
            \begin{split}
                &\int_{t_0}^t \mathrm{e}^\gamma \cdot 2Q_{25}^k \langle v_2, v_5 \rangle \,\mathrm{d}s \\
                &= \int_{t_0}^t \mathrm{e}^\gamma \cdot 2Q_{25}^k \langle \nabla f, \ddot{x} \rangle \,\mathrm{d}s \\
                &= \int_{t_0}^t \mathrm{e}^\gamma \cdot 2Q_{25}^k (\langle \nabla f, \dot{x} \rangle)^\prime \,\mathrm{d}s - \int_{t_0}^t \mathrm{e}^\gamma \cdot 2Q_{25}^k \langle \dot{x}, \nabla^2 f \dot{x} \rangle \,\mathrm{d}s \\
                &= \mathrm{e}^\gamma \cdot 2 Q_{25}^k \langle \nabla f, \dot{x} \rangle - \int_{t_0}^t \mathrm{e}^\gamma \cdot 2(\dot{\gamma} Q_{25}^k + \dot{Q_{25}^k}) \langle \nabla f, \dot{x} \rangle \,\mathrm{d}s \\
                &\phantom{=} - \int_{t_0}^t \mathrm{e}^\gamma \cdot 2Q_{25}^k \langle \dot{x}, \nabla^2 f \dot{x} \rangle \,\mathrm{d}s +\mathrm{const.}\\
                &= \mathrm{e}^\gamma \cdot 2 Q_{25}^k \langle v_2,v_3 \rangle - \int_{t_0}^t \mathrm{e}^\gamma \cdot 2(\dot{\gamma} Q_{25}^k + \dot{Q_{25}^k}) \langle v_2,v_3 \rangle \,\mathrm{d}s \\
                &\phantom{=} - \int_{t_0}^t \mathrm{e}^\gamma \cdot 2Q_{25}^k \langle v_3,v_4 \rangle \,\mathrm{d}s +\mathrm{const.}
            \end{split}
        \end{equation}
        Therefore, applying Operation D2 to $(P^k,Q^k)$ gives
        \begin{gather}
            P^{k+1} = P^k + \begin{pmatrix}
                0 & 0 & 0 \\
                0 & 0 & Q_{25}^k \\
                0 & Q_{25}^k & 0
            \end{pmatrix}, \\
            Q^{k+1} = Q^k - \begin{pmatrix}
                0 & 0 & 0 & 0 & 0 \\
                0 & 0 & \dot{\gamma}Q_{25}^k+\dot{Q_{25}^k} & 0 & Q_{25}^k \\
                0 & \dot{\gamma}Q_{25}^k+\dot{Q_{25}^k} & 0 & Q_{25}^k & 0 \\
                0 & 0 & Q_{25}^k & 0 & 0 \\
                0 & Q_{25}^k & 0 & 0 & 0
            \end{pmatrix}.
        \end{gather}

    \item Operation D3: Integration by parts for $\langle v_2, v_3 \rangle$
        \begin{equation}
            \begin{split}
                &\int_{t_0}^t \mathrm{e}^\gamma \cdot 2Q_{23}^k \langle v_2,v_3 \rangle \,\mathrm{d}s \\
                &= \int_{t_0}^t \mathrm{e}^\gamma \cdot 2Q_{23}^k \langle \nabla f, \dot{x} \rangle \,\mathrm{d}s \\
                &= \int_{t_0}^t \mathrm{e}^\gamma \cdot 2Q_{23}^k (\langle x-x_*, \nabla f \rangle)^\prime \,\mathrm{d}s - \int_{t_0}^t \mathrm{e}^\gamma \cdot 2Q_{23}^k \langle x-x_*, \nabla^2 f\dot{x} \rangle \,\mathrm{d}s \\
                &= \mathrm{e}^\gamma \cdot 2Q_{23}^k \langle x-x_*, \nabla f \rangle - \int_{t_0}^t \mathrm{e}^\gamma \cdot 2(\dot{\gamma} Q_{23}^k + \dot{Q_{23}^k})\langle x-x_*, \nabla f \rangle \,\mathrm{d}s \\
                &\phantom{=} - \int_{t_0}^t \mathrm{e}^\gamma \cdot 2Q_{23}^k \langle x-x_*, \nabla^2 f \dot{x} \rangle \,\mathrm{d}s +\mathrm{const.}\\
                &= \mathrm{e}^\gamma \cdot 2Q_{23}^k \langle v_1, v_2 \rangle - \int_{t_0}^t \mathrm{e}^\gamma \cdot 2(\dot{\gamma} Q_{23}^k + \dot{Q_{23}^k})\langle v_1, v_2 \rangle \,\mathrm{d}s \\
                &\phantom{=} - \int_{t_0}^t \mathrm{e}^\gamma \cdot 2Q_{23}^k \langle v_1, v_4 \rangle \,\mathrm{d}s +\mathrm{const.}
            \end{split}
        \end{equation}
        Therefore, applying Operation D3 to $(P^k,Q^k)$ gives
        \begin{gather}
            P^{k+1} = P^k + \begin{pmatrix}
                0 & Q_{23}^k & 0 \\
                Q_{23}^k & 0 & 0 \\
                0 & 0 & 0
            \end{pmatrix}, \\
            Q^{k+1} = Q^k - \begin{pmatrix}
                0 & \dot{\gamma}Q_{23}^k+\dot{Q_{23}^k} & 0 & Q_{23}^k & 0 \\
                \dot{\gamma}Q_{23}^k+\dot{Q_{23}^k} & 0 & Q_{23}^k & 0 & 0 \\
                0 & Q_{23}^k & 0 & 0 & 0 \\
                Q_{23}^k & 0 & 0 & 0 & 0 \\
                0 & 0 & 0 & 0 & 0
            \end{pmatrix}.
        \end{gather}

    \item Operation D4: Integration by parts for $\langle v_1, v_4 \rangle$
        \begin{equation}
            \begin{split}
                &\int_{t_0}^t \mathrm{e}^\gamma \cdot 2Q_{14}^k \langle v_1, v_4 \rangle \,\mathrm{d}s \\
                &= \int_{t_0}^t \mathrm{e}^\gamma \cdot 2Q_{14}^k \langle x-x_*, \nabla^2 f \dot{x} \rangle \,\mathrm{d}s \\
                &= \int_{t_0}^t \mathrm{e}^\gamma \cdot 2Q_{14}^k (\langle x-x_*, \nabla f \rangle)^\prime \,\mathrm{d}s - \int_{t_0}^t \mathrm{e}^\gamma \cdot 2Q_{14}^k \langle \nabla f, \dot{x} \rangle \,\mathrm{d}s \\
                &= \mathrm{e}^\gamma \cdot 2Q_{14}^k \langle x-x_*, \nabla f \rangle - \int_{t_0}^t \mathrm{e}^\gamma \cdot 2(\dot{\gamma}Q_{14}^k + \dot{Q_{14}^k}) \langle x-x_*, \nabla f \rangle \,\mathrm{d}s \\
                &\phantom{=} - \int_{t_0}^t \mathrm{e}^\gamma \cdot 2Q_{14}^k \langle \nabla f, \dot{x} \rangle \,\mathrm{d}s +\mathrm{const.}\\
                &= \mathrm{e}^\gamma \cdot 2Q_{14}^k \langle v_1, v_2 \rangle - \int_{t_0}^t \mathrm{e}^\gamma \cdot 2(\dot{\gamma}Q_{14}^k + \dot{Q_{14}^k}) \langle v_1, v_2 \rangle \,\mathrm{d}s \\
                &\phantom{=} - \int_{t_0}^t \mathrm{e}^\gamma \cdot 2Q_{14}^k \langle v_2, v_3 \rangle \,\mathrm{d}s +\mathrm{const.}
            \end{split}
        \end{equation}
        Therefore, applying Operation D4 to $(P^k,Q^k)$ gives
        \begin{gather}
            P^{k+1} = P^k + \begin{pmatrix}
                0 & Q_{14}^k & 0 \\
                Q_{14}^k & 0 & 0 \\
                0 & 0 & 0
            \end{pmatrix}, \\
            Q^{k+1} = Q^k - \begin{pmatrix}
                0 & \dot{\gamma}Q_{14}^k+\dot{Q_{14}^k} & 0 & Q_{14}^k & 0 \\
                \dot{\gamma}Q_{14}^k+\dot{Q_{14}^k} & 0 & Q_{14}^k & 0 & 0 \\
                0 & Q_{14}^k & 0 & 0 & 0 \\
                Q_{14}^k & 0 & 0 & 0 & 0 \\
                0 & 0 & 0 & 0 & 0
            \end{pmatrix}.
        \end{gather}
\end{itemize}

\subsection{Operation Based on the Properties of the Objective Function (1)}
\label{sec:oper_Group_E}

Up to this point, we have only dealt with integration-by-parts operations aimed at eliminating the coefficient of a given term.  
However, it is also possible to make use of the convexity, strong convexity or smoothness of the objective function $f(x)$.  
Here, we introduce an operation studied in Kamijima et al.\cite{TomoyaKamijima2024}.  
%This operation is noteworthy in that it allows a non-diagonal entry to be transferred to a diagonal entry.

\begin{itemize}
    \item Operation E1: An operation that transfers $\langle v_1, v_4 \rangle$ to $\|x-x_*\|^2$ using strong convexity and smoothness.
        \begin{equation}
            \begin{split}
                &\int_{t_0}^t \mathrm{e}^\gamma \cdot 2Q_{14}^k \langle v_1, v_4 \rangle \mathrm{d}s \\
                &= \int_{t_0}^t \mathrm{e}^\gamma \cdot 2Q_{14}^k \langle x-x_*, \nabla^2 f \dot{x} \rangle \mathrm{d}s \\
                &= \int_{t_0}^t \mathrm{e}^\gamma \cdot 2Q_{14}^k (\langle x-x_*, \nabla f \rangle)^\prime \mathrm{d}s - \int_{t_0}^t \mathrm{e}^\gamma \cdot 2Q_{14}^k \langle \nabla f, \dot{x} \rangle \mathrm{d}s \\
                &= \mathrm{e}^\gamma \cdot 2Q_{14}^k \langle x-x_*, \nabla f \rangle - \int_{t_0}^t \mathrm{e}^\gamma \cdot 2(\dot{\gamma}Q_{14}^k + \dot{Q_{14}^k}) \langle x-x_*, \nabla f \rangle \mathrm{d}s \\
                &\phantom{=} - \int_{t_0}^t \mathrm{e}^\gamma \cdot 2Q_{14}^k (f-f_*)^\prime \mathrm{d}s +\mathrm{const.}\\
                &= \mathrm{e}^\gamma \cdot 2Q_{14}^k \langle x-x_*, \nabla f \rangle - \int_{t_0}^t \mathrm{e}^\gamma \cdot 2(\dot{\gamma}Q_{14}^k + \dot{Q_{14}^k}) \langle x-x_*, \nabla f \rangle \mathrm{d}s \\
                &\phantom{=} - \mathrm{e}^\gamma \cdot 2Q_{14}^k (f-f_*) + \int_{t_0}^t \mathrm{e}^\gamma \cdot 2(\dot{\gamma}Q_{14}^k + \dot{Q_{14}^k}) (f-f_*) \mathrm{d}s +\mathrm{const.}\\
                &= \mathrm{e}^\gamma \cdot 2Q_{14}^k (f_* - f - \langle \nabla f, x_*-x \rangle) \\
                &\phantom{=} - \int_{t_0}^t \mathrm{e}^\gamma \cdot 2(\dot{\gamma}Q_{14}^k + \dot{Q_{14}^k}) (f_* - f - \langle \nabla f, x_*-x \rangle)\mathrm{d}s +\mathrm{const.}
            \end{split}
        \end{equation}
        Here, when the objective function $f(x)$ is $\mu$-strongly convex and $L$-smooth,  
        we can use a parameter $\lambda \in [\mu, L]$ (which is the same $\lambda$ that appeared in Operation A1) satisfying
        \begin{equation}
            f_* - f - \langle \nabla f, x_* - x \rangle = \frac{\lambda}{2} \|x-x_*\|^2.
        \end{equation}
        Hence we have
        \begin{equation}
            \begin{split}
                &\int_{t_0}^t \mathrm{e}^\gamma \cdot 2Q_{14}^k \langle v_1, v_4 \rangle \mathrm{d}s \\
                &= \mathrm{e}^\gamma \cdot \lambda Q_{14}^k \|x-x_*\|^2 - \int_{t_0}^t \mathrm{e}^\gamma \cdot \lambda (\dot{\gamma}Q_{14}^k + \dot{Q_{14}^k}) \|x-x_*\|^2 \mathrm{d}s +\mathrm{const.}\\
                &= \mathrm{e}^\gamma \cdot \lambda Q_{14}^k \|v_1\|^2 - \int_{t_0}^t \mathrm{e}^\gamma \cdot \lambda (\dot{\gamma}Q_{14}^k + \dot{Q_{14}^k}) \|v_1\|^2 \mathrm{d}s +\mathrm{const.}
            \end{split}
        \end{equation}
        Therefore, applying Operation E1 to $(P_k,Q_k)$ gives
        \begin{gather}
            P^{k+1} = P^k + \begin{pmatrix}
                \lambda Q_{14}^k & 0 & 0 \\
                0 & 0 & 0 \\
                0 & 0 & 0
            \end{pmatrix},\\
            Q^{k+1} = Q^k - \begin{pmatrix}
                \lambda (\dot{\gamma}Q_{14}^k+\dot{Q_{14}^k}) & 0 & 0 & Q_{14}^k & 0 \\
                0 & 0 & 0 & 0 & 0 \\
                0 & 0 & 0 & 0 & 0 \\
                Q_{14}^k & 0 & 0 & 0 & 0 \\
                0 & 0 & 0 & 0 & 0
            \end{pmatrix}.
        \end{gather}
\end{itemize}

\subsection{Operation Based on the Properties of the Objective Function (2)}
\label{sec:oper_Group_F}

Finally, we introduce another operation that is also based on the strong convexity and smoothness of the objective function $f(x)$. 
%is Similar to Operation E1, this operation has the advantage of being able to transfer off-diagonal terms to diagonal terms.  
\begin{itemize}
    \item Operation F1: An operation that usesusesusesusesuses the properties of the Hessian matrix derived from strong convexity and smoothness. We consider the term
        \begin{equation}
            \int_{t_0}^t \mathrm{e}^\gamma \cdot 2Q_{34}^k \langle v_3, v_4 \rangle \mathrm{d}s .
        \end{equation}
        From Theorem \ref{Hessian}, we have
        \begin{equation}
            \mu I_n \preceq \nabla^2 f(x) \preceq L I_n,
        \end{equation}
        which implies
        \begin{gather}
            \langle \mu I_n \dot{x}, \dot{x} \rangle \leq \langle \nabla^2 f \dot{x}, \dot{x} \rangle \leq \langle L I_n \dot{x}, \dot{x} \rangle \\
            \Leftrightarrow \mu \|\dot{x}\|^2 \leq \langle \nabla^2 f \dot{x}, \dot{x} \rangle \leq L \|\dot{x}\|^2 \\
            \Leftrightarrow \mu \|v_3\|^2 \leq \langle v_3,v_4 \rangle \leq L \|v_3\|^2.
        \end{gather}
        Therefore, when $Q_{34}^k \geq 0$, we obtain
        \begin{equation}
            2\mu Q_{34}^k \|v_3\|^2 \leq 2 Q_{34}^k \langle v_3,v_4 \rangle \leq 2 L Q_{34}^k \|v_3\|^2,
        \end{equation}
        while, conversely, when $Q_{34}^k \leq 0$, we have
        \begin{equation}
            2L Q_{34}^k \|v_3\|^2 \leq 2 Q_{34}^k \langle v_3,v_4 \rangle \leq 2 \mu Q_{34}^k \|v_3\|^2.
        \end{equation}
        Thus, in both cases, there exists a parameter $\theta \in [\mu, L]$ such that
        \begin{equation}
            2Q_{34}^k \langle v_3, v_4 \rangle = 2 \theta Q_{34}^k \|v_3\|^2.
        \end{equation}
        Therefore, we obtain
        \begin{equation}
            \int_{t_0}^t \mathrm{e}^\gamma \cdot 2Q_{34}^k \langle v_3, v_4 \rangle \mathrm{d}s = \int_{t_0}^t \mathrm{e}^\gamma \cdot 2 \theta Q_{34}^k \|v_3\|^2 \mathrm{d}s.
        \end{equation}
        Hence, applying Operation F1 to $(P_k,Q_k)$ yields
        \begin{gather}
            P^{k+1} = P^k, \\
            Q^{k+1} = Q^k - \begin{pmatrix}
                0 & 0 & 0 & 0 & 0 \\
                0 & 0 & 0 & 0 & 0 \\
                0 & 0 & -2\theta Q_{34}^k & Q_{34}^k & 0 \\
                0 & 0 & Q_{34}^k & 0 & 0 \\
                0 & 0 & 0 & 0 & 0
            \end{pmatrix}.
        \end{gather}
\end{itemize}

\section{List of Combinations of the Operations}
\label{sec:list_of_combi}

\subsection{Combinations of the Operations within Each Group}
\label{sec:comb_operations_each_group}

\begin{itemize}
    \item Group A  
        \\As mentioned earlier, Group A only consists of operation A1, which is initially performed once.
    \item Group B  
        \\Group B contains the three operations, B1, B2, and B3, each with the following properties:
        \begin{itemize}
            \item B1: Eliminates $Q_{35}$ and adds some terms to $Q_{33}$.
            \item B2: Eliminates $Q_{15}$ and adds some terms to $Q_{13}$ and $Q_{33}$.
            \item B3: Eliminates $Q_{13}$ and adds some terms to $Q_{11}$.
        \end{itemize}
        Each of them is idempotent because $Q_{35}$, $Q_{15}$, or $Q_{13}$ becomes zero after its application. 
        In addition, since none of them uses the contents of $Q_{33}$, B1 and B2 commute, and B1 and B3 also commute. Therefore we only need to consider combinations of B2 and B3. 
        Because of their idempotency, the possible combinations of B2 and B3 are:
        \begin{itemize}
            \item Do nothing
            \item B2
            \item B3
            \item B2→B3
            \item B3→B2
            \item B3→B2→B3
            \item B2→B3→B2
        \end{itemize}
        yielding six cases. However, the results of B2→B3→B2 and B3→B2→B3 are identical to that of B2→B3. Thus three or more operations can be reduced to two or less operations. In summary, the possible cases of Group B are:
        \begin{enumerate}
            \item Do nothing
            \item B1
            \item B2
            \item B1→B2
            \item B3
            \item B1→B3
            \item B2→B3
            \item B1→B2→B3
            \item B3→B2
            \item B1→B3→B2
        \end{enumerate}
        giving ten cases in total.
    \item Group C  
        \\Group C consists only of operation C1. Since applying C1 once results in $Q_{24}=0$, C1 is idempotent. Hence there are only two possible cases:
        \begin{enumerate}
            \item Do nothing
            \item C1
        \end{enumerate}
    \item Group D  
        \\Group D contains four operations D1, D2, D3, and D4. There are seven related entries $P_{12}$, $P_{23}$, $Q_{12}$, $Q_{14}$, $Q_{23}$, $Q_{25}$, and $Q_{34}$. Here we define
        \begin{equation}
            \label{definition_g}
            g(Q_{ij}^k) = \dot{\gamma}Q_{ij}^k + \dot{Q_{ij}^k},
        \end{equation}
        which is a linear function. 
        Table \ref{D_single} shows the changes in the entries $P_{12}^k$, $P_{23}^k$, $Q_{12}^k$, $Q_{14}^k$, $Q_{23}^k$, $Q_{25}^k$, and $Q_{34}^k$ when a single operation from Group D is applied.

\begin{table}[H]
\centering
\caption{Changes in each entry when a single operation from Group~D is applied once}
\label{D_single}
\resizebox{\textwidth}{!}{
\begin{tabular}{|c|c|c|c|c|c|c|c|}
\hline Operation & $Q_{12}^{k+1}$ & $Q_{14}^{k+1}$ & $Q_{23}^{k+1}$ & $Q_{25}^{k+1}$ & $Q_{34}^{k+1}$ & $P_{12}^{k+1}$ & $P_{23}^{k+1}$ \\ \hline
D1 & $Q_{12}^{k}$ & $Q_{14}^{k}$ & $Q_{23}^{k} - g(Q_{34}^{k})$ & $Q_{25}^{k} - Q_{34}^{k}$ & $0$ & $P_{12}^{k}$ & $P_{23}^{k} + Q_{34}^{k}$ \\ \hline
D2 & $Q_{12}^{k}$ & $Q_{14}^{k}$ & $Q_{23}^{k} - g(Q_{25}^{k})$ & $0$ & $Q_{34}^{k} - Q_{25}^{k}$ & $P_{12}^{k}$ & $P_{23}^{k} + Q_{25}^{k}$ \\ \hline
D3 & $Q_{12}^{k} - g(Q_{23}^{k})$ & $Q_{14}^{k} - Q_{23}^{k}$ & $0$ & $Q_{25}^{k}$ & $Q_{34}^{k}$ & $P_{12}^{k} + Q_{23}^{k}$ & $P_{23}^{k}$ \\ \hline
D4 & $Q_{12}^{k} - g(Q_{14}^{k})$ & $0$ & $Q_{23}^{k} - Q_{14}^{k}$ & $Q_{25}^{k}$ & $Q_{34}^{k}$ & $P_{12}^{k} + Q_{14}^{k}$ & $P_{23}^{k}$ \\ \hline
\end{tabular}
}
\end{table}

From Table \ref{D_single}, it is evident that 
%performing the same operation more than once produces no further change, since the entry to which integration by parts is applied becomes zero after the first operation.
all the operations in Group D are idempotent. 
Therefore, when considering two successive operations in  Group~D, it suffices to examine only the 12 cases in Table~\ref{D_double}.
%excluding the consecutive repetition of the same operation.
%The results of these 12 cases are summarized in the following table.

\begin{table}[H]
    \centering
    \caption{Changes in each entry when operations belonging to Group D are performed twice}
    \label{D_double}
    \resizebox{\textwidth}{!}{
    \begin{tabular}{|c|c|c|c|c|c|c|c|}
         \hline Operation & $Q_{12}^{k+2}$ & $Q_{14}^{k+2}$ & $Q_{23}^{k+2}$ & $Q_{25}^{k+2}$ & $Q_{34}^{k+2}$ & $P_{12}^{k+2}$ & $P_{23}^{k+2}$ \\ \hline
         D1→D2 & $Q_{12}^{k}$ & $Q_{14}^{k}$ & $Q_{23}^{k} - g(Q_{25}^{k})$ & $0$ & $Q_{34}^{k} - Q_{25}^{k}$ & $P_{12}^{k}$ & $P_{23}^{k} + Q_{25}^{k}$ \\ \hline
         D1→D3 & $Q_{12}^{k} - g(Q_{23}^{k}) + g(g(Q_{34}^{k}))$ & $Q_{14}^{k} - Q_{23}^{k} + g(Q_{34}^{k})$ & $0$ & $Q_{25}^{k} - Q_{34}^{k}$ & $0$ & $P_{12}^{k} + Q_{23}^{k} - g(Q_{34}^{k})$ & $P_{23}^{k} + Q_{34}^{k}$ \\ \hline
         D1→D4 & $Q_{12}^{k} - g(Q_{14}^{k})$ & $0$ & $Q_{23}^{k} - Q_{14}^{k} - g(Q_{34}^{k})$ & $Q_{25}^{k} - Q_{34}^{k}$ & $0$ & $P_{12}^{k} + Q_{14}^{k}$ & $P_{23}^{k} + Q_{34}^{k}$ \\ \hline
         D2→D1 & $Q_{12}^{k}$ & $Q_{14}^{k}$ & $Q_{23}^{k} - g(Q_{34}^{k})$ & $Q_{25}^{k} - Q_{34}^{k}$ & $0$ & $P_{12}^{k}$ & $P_{23}^{k} + Q_{34}^{k}$ \\ \hline
         D2→D3 & $Q_{12}^{k} - g(Q_{23}^{k}) + g(g(Q_{25}^{k}))$ & $Q_{14}^{k} - Q_{23}^{k} + g(Q_{25}^{k})$ & $0$ & $0$ & $Q_{34}^{k} - Q_{25}^{k}$ & $P_{12}^{k} + Q_{23}^{k} - g(Q_{25}^{k})$ & $P_{23}^{k} + Q_{25}^{k}$ \\ \hline
         D2→D4 & $Q_{12}^{k} - g(Q_{14}^{k})$ & $0$ & $Q_{23}^{k} - Q_{14}^{k} - g(Q_{25}^{k})$ & $0$ & $Q_{34}^{k} - Q_{25}^{k}$ & $P_{12}^{k} + Q_{14}^{k}$ & $P_{23}^{k} + Q_{25}^{k}$ \\ \hline
         D3→D1 & $Q_{12}^{k} - g(Q_{23}^{k})$ & $Q_{14}^{k} - Q_{23}^{k}$ & $-g(Q_{34}^{k})$ & $Q_{25}^{k} - Q_{34}^{k}$ & $0$ & $P_{12}^{k} + Q_{23}^{k}$ & $P_{23}^{k} + Q_{34}^{k}$ \\ \hline
         D3→D2 & $Q_{12}^{k} - g(Q_{23}^{k})$ & $Q_{14}^{k} - Q_{23}^{k}$ & $-g(Q_{25}^{k})$ & $0$ & $Q_{34}^{k} - Q_{25}^{k}$ & $P_{12}^{k} + Q_{23}^{k}$ & $P_{23}^{k} + Q_{25}^{k}$ \\ \hline
         D3→D4 & $Q_{12}^{k} - g(Q_{14}^{k})$ & $0$ & $Q_{23}^{k} - Q_{14}^{k}$ & $Q_{25}^{k}$ & $Q_{34}^{k}$ & $P_{12}^{k} + Q_{14}^{k}$ & $P_{23}^{k}$ \\ \hline
         D4→D1 & $Q_{12}^{k} - g(Q_{14}^{k})$ & $0$ & $Q_{23}^{k} - Q_{14}^{k} -g(Q_{34}^{k})$ & $Q_{25}^{k} - Q_{34}^{k}$ & $0$ & $P_{12}^{k} + Q_{14}^{k}$ & $P_{23}^{k} + Q_{34}^{k}$ \\ \hline
         D4→D2 & $Q_{12}^{k} - g(Q_{14}^{k})$ & $0$ & $Q_{23}^{k} - Q_{14}^{k} - g(Q_{25}^{k})$ & $0$ & $Q_{34}^{k} - Q_{25}^{k}$ & $P_{12}^{k} + Q_{14}^{k}$ & $P_{23}^{k} + Q_{25}^{k}$ \\ \hline
         D4→D3 & $Q_{12}^{k} - g(Q_{23}^{k})$ & $Q_{14}^{k} - Q_{23}^{k}$ & $0$ & $Q_{25}^{k}$ & $Q_{34}^{k}$ & $P_{12}^{k} + Q_{23}^{k}$ & $P_{23}^{k}$ \\ \hline
    \end{tabular}
    }
\end{table}

We can make two observations from Table~\ref{D_double}. 
\begin{enumerate}
\item 
The result of D1→D2 is exactly the same as that of D2. Similarly, D2→D1 is equivalent to D1, D3→D4 is equivalent to D4, and D4→D3 is equivalent to D3. 
Therefore any number of operations within each set of $\{ \mathrm{D1}, \mathrm{D2} \}$ and $\{\mathrm{D3}, \mathrm{D4} \}$ can be regarded as a single operation in that set. 
\item 
The results of D1→D4 and D4→D1, and D2→D4 and D4→D2 are identical. 
Therefore 
D1 and D4 are interchangeable, and 
D2 and D4 are also interchangeable.  
\end{enumerate}
For the first observation, 
it is sufficient to consider only the 16 cases of three operations from Group D in which the 1st and 3rd operations are chosen from the same set among $\{ \mathrm{D1}, \mathrm{D2} \}$ and $\{\mathrm{D3}, \mathrm{D4} \}$ and the 2nd from the other set. 
Moreover, 
several of the 16 cases can be replaced by shorter operation sequences because of the second observation. 
Therefore we can list the 16 cases as follows: 
\begin{itemize}
    \item D1→D3→D1
    \item D1→D3→D2
    \item D1→D4→D1 = D1→D1→D4 = D1→D4
    \item D1→D4→D2 = D1→D2→D4 = D2→D4
    \item D2→D3→D1
    \item D2→D3→D2
    \item D2→D4→D1 = D2→D1→D4 = D1→D4
    \item D2→D4→D2 = D2→D2→D4 = D2→D4
    \item D3→D1→D3
    \item D3→D1→D4 = D3→D4→D1 = D4→D1
    \item D3→D2→D3
    \item D3→D2→D4 = D3→D4→D2 = D4→D2
    \item D4→D1→D3 = D1→D4→D3 = D1→D3
    \item D4→D1→D4 = D1→D4→D4 = D1→D4
    \item D4→D2→D3 = D2→D4→D3 = D2→D3
    \item D4→D2→D4 = D2→D4→D4 = D2→D4
\end{itemize}

Then we investigate the 6 cases whose lengths remain three.  
Their outcomes are shown in Table \ref{D_triple}.

\begin{table}[H]
    \centering
    \caption{Changes in each entry when operations belonging to Group D are performed three times}
    \label{D_triple}
    \resizebox{\textwidth}{!}{
    \begin{tabular}{|c|c|c|c|c|c|c|c|}
         \hline Operation & $Q_{12}^{k+3}$ & $Q_{14}^{k+3}$ & $Q_{23}^{k+3}$ & $Q_{25}^{k+3}$ & $Q_{34}^{k+3}$ & $P_{12}^{k+3}$ & $P_{23}^{k+3}$ \\ \hline
         D1→D3→D1 & $Q_{12}^{k} - g(Q_{23}^{k}) + g(g(Q_{34}^{k}))$ & $Q_{14}^{k} - Q_{23}^{k} + g(Q_{34}^{k})$ & $0$ & $Q_{25}^{k} - Q_{34}^{k}$ & $0$ & $P_{12}^{k} + Q_{23}^{k} - g(Q_{34}^{k})$ & $P_{23}^{k} + Q_{34}^{k}$ \\ \hline
         D1→D3→D2 & $Q_{12}^{k} - g(Q_{23}^{k}) + g(g(Q_{34}^{k}))$ & $Q_{14}^{k} - Q_{23}^{k} + g(Q_{34}^{k})$ & $g(Q_{34}^{k}) - g(Q_{25}^{k})$ & $0$ & $Q_{34}^{k} - Q_{25}^{k}$ & $P_{12}^{k} + Q_{23}^{k} - g(Q_{34}^{k})$ & $P_{23}^{k} + Q_{25}^{k}$ \\ \hline
         D2→D3→D1 & $Q_{12}^{k} - g(Q_{23}^{k}) + g(g(Q_{25}^{k}))$ & $Q_{14}^{k} - Q_{23}^{k} + g(Q_{25}^{k})$ & $g(Q_{25}^{k}) - g(Q_{34}^{k})$ & $Q_{25}^{k} - Q_{34}^{k}$ & $0$ & $P_{12}^{k} + Q_{23}^{k} - g(Q_{25}^{k})$ & $P_{23}^{k} + Q_{34}^{k}$ \\ \hline
         D2→D3→D2 & $Q_{12}^{k} - g(Q_{23}^{k}) + g(g(Q_{25}^{k}))$ & $Q_{14}^{k} - Q_{23}^{k} + g(Q_{25}^{k})$ & $0$ & $0$ & $Q_{34}^{k} - Q_{25}^{k}$ & $P_{12}^{k} + Q_{23}^{k} - g(Q_{25}^{k})$ & $P_{23}^{k} + Q_{25}^{k}$ \\ \hline
         D3→D1→D3 & $Q_{12}^{k} - g(Q_{23}^{k}) + g(g(Q_{34}^{k}))$ & $Q_{14}^{k} - Q_{23}^{k} + g(Q_{34}^{k})$ & $0$ & $Q_{25}^{k} - Q_{34}^{k}$ & $0$ & $P_{12}^{k} + Q_{23}^{k} - g(Q_{34}^{k})$ & $P_{23}^{k} + Q_{34}^{k}$ \\ \hline
         D3→D2→D3 & $Q_{12}^{k} - g(Q_{23}^{k}) + g(g(Q_{25}^{k}))$ & $Q_{14}^{k} - Q_{23}^{k} + g(Q_{25}^{k})$ & $0$ & $0$ & $Q_{34}^{k} - Q_{25}^{k}$ & $P_{12}^{k} + Q_{23}^{k} - g(Q_{25}^{k})$ & $P_{23}^{k} + Q_{25}^{k}$ \\ \hline
    \end{tabular}
    }
\end{table}

From Table \ref{D_triple}, 
we see that D1→D3→D1 and D3→D1→D3 yield the same result as D1→D3, and D2→D3→D2 and D3→D2→D3 yield the same result as D2→D3. Therefore we have only to consider D1→D3→D2 and D2→D3→D1 as operation sequences with length three. 

Finally, we consider four or more operations from Group D. It is sufficient to consider the cases where operations are chosen alternately from different sets among $\{ \mathrm{D1}, \mathrm{D2} \}$ and $\{\mathrm{D3}, \mathrm{D4} \}$. Then any sequences of such four operations must include sequences with length three other than D1→D3→D2 and D2→D3→D1. 
Because they can be replaced by shorter operation sequences, any sequence of four or more operations from Group D can be reduced to a shorter operation sequence. 

As a result, the possible operation sequences from Group D are listed as follows:

\begin{enumerate}
    \item Do nothing
    \item D1
    \item D2
    \item D3
    \item D4
    \item D1→D3
    \item D1→D4
    \item D2→D3
    \item D2→D4
    \item D3→D1
    \item D3→D2
    \item D1→D3→D2
    \item D2→D3→D1
\end{enumerate}

\item Group E
\\Group E only consists of operation E1. It is idempotent because it sets $Q_{14}=0$. Hence there are two possible combinations as follows:

\begin{enumerate}
    \item Do nothing
    \item E1
\end{enumerate}

\item Group F
\\Group F only consists of operation F1. It is idempotent because it sets $Q_{34}=0$. Hence there are two possible combinations as follows:
%\\There is only one type of operation in Group F, F1. Performing F1 sets $Q_{34}=0$, so F1 is idempotent. Hence, the outcomes are:

\begin{enumerate}
    \item Do nothing
    \item F1
\end{enumerate}
\end{itemize}

From the above, the operation sequences in each group are counted as follows.
%From the above, the sequences of the operations in each group is counted as follows:

\begin{enumerate}
    \item Group A: 1 case
    \item Group B: 10 cases
    \item Group C: 2 cases
    \item Group D: 13 cases
    \item Group E: 2 cases
    \item Group F: 2 cases
\end{enumerate}

\subsection{Combinations of the Operations across Different Groups}
\label{sec:comb_acoss_groups}

\subsubsection{Notation}
\begin{itemize}
    \item The operation ``do nothing'' is denoted by operation $\mathrm{I}$.
    \item Performing operation $y$ after operation $x$ is denoted by $xy$.
    \item A collection consisting of multiple operations is referred to as a set of operations.
    \item Performing a set of operations $Y$ after a set of operations $X$ is denoted by $XY$. For example, when the set of operations $X$ consists of $\mathrm{X1}$ and $\mathrm{X2}$, and the set of operations $Y$ consists of $\mathrm{Y1}$ and $\mathrm{Y2}$, we have
    \[
        XY = \mathrm{X1}\mathrm{Y1},\mathrm{X1}\mathrm{Y2},\mathrm{X2}\mathrm{Y1},\mathrm{X2}\mathrm{Y2}
    \]
    \item For an arbitrary pair of matrices $P$ and $Q$, the set of matrix pairs obtained by applying a set of operations $X$ is denoted by $R_{P,Q}(X)$.
    \item For sets of operations $X$ and $Y$, if $R_{P,Q}(X) = R_{P,Q}(Y)$ holds for any matrices pair $(P,Q)$, we write $X = Y$.
\end{itemize}
For simplicity, we hereafter write $R_{P,Q}(X)$ simply as $R(X)$.  
In the subsequent discussion, the relations $R(X) = R(Y)$ and $R(X) \subseteq R(Y)$ 
should be understood as meaning that these relations hold 
for all matrix pairs $(P, Q)$, i.e., $R_{P,Q}(X) = R_{P,Q}(Y)$ and $R_{P,Q}(X) \subseteq R_{P,Q}(Y)$ hold
for any $(P, Q)$, respectively.

We define the sets of operations as follows.
\begin{table}[H]
    \centering
    \begin{tabular}{|c|c|} \hline
        Set of operations & Operations \\ \hline
        $D$ & $\mathrm{I}$, $\mathrm{D1}$, $\mathrm{D2}$, $\mathrm{D3}$, $\mathrm{D4}$, $\mathrm{D1}\mathrm{D3}$, $\mathrm{D1}\mathrm{D4}$, $\mathrm{D2}\mathrm{D3}$, $\mathrm{D2}\mathrm{D4}$, $\mathrm{D3}\mathrm{D1}$, $\mathrm{D3}\mathrm{D2}$, $\mathrm{D1}\mathrm{D3}\mathrm{D2}$, $\mathrm{D2}\mathrm{D3}\mathrm{D1}$ \\ \hline
        $E$ & $\mathrm{I}$, $\mathrm{E1}$ \\ \hline
        $F$ & $\mathrm{I}$, $\mathrm{F1}$ \\ \hline
        $D^u$ & $\mathrm{I}$, $\mathrm{D1}$, $\mathrm{D2}$ \\ \hline
        $D^d$ & $\mathrm{I}$, $\mathrm{D3}$, $\mathrm{D4}$ \\ \hline
    \end{tabular}
\end{table}

\subsubsection{Lemmas}

\begin{lem}[Idempotence]
    \label{idempotent}
    The identities $DD = D$, $EE = E$, $FF = F$, $D^uD^u = D^u$, and $D^dD^d = D^d$ hold.
\end{lem}
\begin{proof}
    These follow from the arguments for Groups D, E, and F in Appendix~\ref{sec:comb_operations_each_group}.
\end{proof}

\begin{lem}[Decomposition of $D$]
    \label{D-split}
    The identity $D = D^uD^dD^u$ holds.
\end{lem}
\begin{proof}
    The inclusion $R(D) \subseteq R(D^uD^dD^u)$ is immediate.  
    The inclusion $R(D^uD^dD^u) \subseteq R(D)$ follows from the arguments for Group~D in Appendix~\ref{sec:comb_operations_each_group}.
\end{proof}

\begin{lem}[Absorption of $D$]
    \label{D-merge}
    The identities $D = D\mathrm{D1} = D\mathrm{D2} = D\mathrm{D3} = D\mathrm{D4} = \mathrm{D1}D = \mathrm{D2}D = \mathrm{D3}D = \mathrm{D4}D$ hold, and thus $D=DD^u = DD^d = D^uD = D^dD$ also hold.
\end{lem}
\begin{proof}
    These follow from the arguments for Group~D in Appendix~\ref{sec:comb_operations_each_group}.
\end{proof}

\begin{lem}[Commutativity of $E$ and $D^u$]
    \label{E-D^u}
    The identities $\mathrm{E1}\mathrm{D1} = \mathrm{D1}\mathrm{E1}, \mathrm{E1}\mathrm{D2} = \mathrm{D2}\mathrm{E1}$ hold, and thus $\mathrm{E1}D^u = D^u\mathrm{E1}, ED^u = D^uE$ also hold. 
\end{lem}
\begin{proof}
    Operations in $E$ affect the components $P_{11}$, $Q_{11}$, and $Q_{14}$.  
    Operations in $D^u$ affect the components $Q_{23}$, $Q_{25}$, and $Q_{34}$.  
    Since these components are disjoint, $E$ and $D^u$ commute.
\end{proof}

\begin{lem}[Relation between $E$ and $\mathrm{D4}$]
    \label{E-D4}
    The identities $\mathrm{E1}\mathrm{D4} = \mathrm{E1}, \mathrm{D4}\mathrm{E1} = \mathrm{D4}$ hold, and thus $\mathrm{D4}E = \mathrm{D4}$ also holds.
\end{lem}
\begin{proof}
    These can be verified by direct computation of each component.
\end{proof}

\begin{lem}[Commutativity of $F$ and $D^d$]
    \label{F-D^d}
    The identities $\mathrm{F1}\mathrm{D3} = \mathrm{D3}\mathrm{F1}, \mathrm{F1}\mathrm{D4} = \mathrm{D4}\mathrm{F1}$, and thus $\mathrm{F1}D^d = D^d\mathrm{F1},FD^d = D^dF$ also hold.
\end{lem}
\begin{proof}
    Operations in $F$ affect the components $Q_{33}$ and $Q_{34}$.  
    Operations in $D^d$ affect the components $Q_{12}$, $Q_{14}$, and $Q_{23}$.  
    Since these components are disjoint, $F$ and $D^d$ commute.
\end{proof}

\begin{lem}[Relation between $F$ and $\mathrm{D1}$]
    \label{F-D1}
    The identities $\mathrm{F1}\mathrm{D1} = \mathrm{F1},\mathrm{D1}\mathrm{F1} = \mathrm{D1}$ hold and thus $\mathrm{D1}F = \mathrm{D1}$ also holds.
\end{lem}
\begin{proof}
    These can be verified by direct computation of each component.
\end{proof}

\begin{lem}[Commutativity of $E$ and $F$]
    \label{E-F}
    The identity $\mathrm{E1}\mathrm{F1} = \mathrm{F1}\mathrm{E1}$ holds, and thus $\mathrm{E1}F = F\mathrm{E1}, E\mathrm{F1} = \mathrm{F1}E, EF = FE$ also hold.
\end{lem}
\begin{proof}
    Operations in $E$ affect the components $P_{11}$, $Q_{11}$, and $Q_{14}$.  
    Operations in $F$ affect the components $Q_{33}$ and $Q_{34}$.  
    Since these components are disjoint, $E$ and $F$ commute.
\end{proof}

\begin{lem}[Relation between $E$ and $D^d$]
    \label{E-D^d}
    The identity $\mathrm{E1}D^d\mathrm{E1} = D^d\mathrm{E1}$ hold, and thus  $ED^d\mathrm{E1} = D^d\mathrm{E1}$ also holds.
\end{lem}
\begin{proof}
    Since $R(D^d\mathrm{E1}) \subseteq R(ED^d\mathrm{E1})$ is immediate, it suffices to show $R(ED^d\mathrm{E1}) \subseteq R(D^d\mathrm{E1})$.  
    The inclusion $R(\mathrm{I}D^d\mathrm{E1}) \subseteq R(D^d\mathrm{E1})$ is also immediate.  
    Thus, it is enough to show
    \[
        R(\mathrm{E1}\mathrm{I}\mathrm{E1}) \subseteq R(D^d\mathrm{E1}),\quad
        R(\mathrm{E1}\mathrm{D3}\mathrm{E1}) \subseteq R(D^d\mathrm{E1}),\quad
        R(\mathrm{E1}\mathrm{D4}\mathrm{E1}) \subseteq R(D^d\mathrm{E1}).
    \]
    First, since $\mathrm{E1}\mathrm{I}\mathrm{E1} = \mathrm{E1}\mathrm{E1} = \mathrm{E1}$, we have $R(\mathrm{E1}\mathrm{I}\mathrm{E1}) \subseteq R(D^d\mathrm{E1})$.  
    Direct computation shows that $\mathrm{E1}\mathrm{D3}\mathrm{E1} = \mathrm{D3}\mathrm{E1}$, implying $R(\mathrm{E1}\mathrm{D3}\mathrm{E1}) \subseteq R(D^d\mathrm{E1})$.  
    Similarly, direct computation shows that $\mathrm{E1}\mathrm{D4}\mathrm{E1} = \mathrm{E1}$, implying $R(\mathrm{E1}\mathrm{D4}\mathrm{E1}) \subseteq R(D^d\mathrm{E1})$.  
    Therefore, $ED^d\mathrm{E1} = D^d\mathrm{E1}$ holds.
\end{proof}

\begin{lem}[Relation between $F$ and $D^u$]
    \label{F-D^u}
    The identity $\mathrm{F1}\mathrm{D2}\mathrm{F1} = \mathrm{D2}\mathrm{F1}$ hold, and thus $\mathrm{F1}D^u\mathrm{F1} = D^u\mathrm{F1}, FD^u\mathrm{F1} = D^u\mathrm{F1}$ also hold.
\end{lem}
\begin{proof}
    Since $R(D^u\mathrm{F1}) \subseteq R(FD^u\mathrm{F1})$ is immediate, it suffices to show $R(FD^u\mathrm{F1}) \subseteq R(D^u\mathrm{F1})$.  
    The inclusion $R(\mathrm{I}D^u\mathrm{F1}) \subseteq R(D^u\mathrm{F1})$ is also immediate.  
    Thus, we need to show
    \[
        R(\mathrm{F1}\mathrm{I}\mathrm{F1}) \subseteq R(D^u\mathrm{F1}),\quad
        R(\mathrm{F1}\mathrm{D1}\mathrm{F1}) \subseteq R(D^u\mathrm{F1}),\quad
        R(\mathrm{F1}\mathrm{D2}\mathrm{F1}) \subseteq R(D^u\mathrm{F1}).
    \]
    Since $\mathrm{F1}\mathrm{I}\mathrm{F1} = \mathrm{F1}\mathrm{F1} = \mathrm{F1}$, we obtain $R(\mathrm{F1}\mathrm{I}\mathrm{F1}) \subseteq R(D^u\mathrm{F1})$.  
    Direct computation shows $\mathrm{F1}\mathrm{D1}\mathrm{F1} = \mathrm{F1}$, implying $R(\mathrm{F1}\mathrm{D1}\mathrm{F1}) \subseteq R(D^u\mathrm{F1})$.  
    Direct computation also shows $\mathrm{F1}\mathrm{D2}\mathrm{F1} = \mathrm{D2}\mathrm{F1}$, implying $R(\mathrm{F1}\mathrm{D2}\mathrm{F1}) \subseteq R(D^u\mathrm{F1})$.  
    Therefore, $FD^u\mathrm{F1} = D^u\mathrm{F1}$ holds.
\end{proof}

\begin{lem}[Relation between $\mathrm{D2}$ and $D^d, \mathrm{F1}$]
    \label{D2-idempotent}
    The identities $\mathrm{D2}\mathrm{D3}\mathrm{D2} = \mathrm{D2}\mathrm{D3},\mathrm{D2}\mathrm{F1}\mathrm{D2} = \mathrm{D2}\mathrm{F1}$ hold, and thus $\mathrm{D2}D^d\mathrm{D2} = \mathrm{D2}D^d$ also holds.
\end{lem}
\begin{proof}
    These can be verified by direct computation of each component.
\end{proof}

\begin{lem}[Relation between $\mathrm{D3}$ and $\mathrm{E1}$]
    \label{D3-idempotent}
    The identity $\mathrm{D3}\mathrm{E1}\mathrm{D3} = \mathrm{D3}\mathrm{E1}$ holds.
\end{lem}
\begin{proof}
    These can be verified by direct computation of each component.
\end{proof}

\subsubsection{Theorems}

\begin{thm}[Reduction of Seven Operation Sets]
    \label{reduction_7}
    The identity
    \[
        R(DEFDEFD) = R(DEFD) \cup R(D\mathrm{E1}\mathrm{D3}\mathrm{D2}\mathrm{F1}D) \cup R(D\mathrm{F1}\mathrm{D2}\mathrm{D3}\mathrm{E1}D)
    \]
    holds.
\end{thm}

\begin{proof}
    The inclusion 
    \[
        R(DEFD) \cup R(D\mathrm{E1}\mathrm{D3}\mathrm{D2}\mathrm{F1}D) \cup R(D\mathrm{F1}\mathrm{D2}\mathrm{D3}\mathrm{E1}D) \subseteq R(DEFDEFD)
    \]
    is immediate. Therefore, it remains to show
    \[
        R(DEFDEFD) \subseteq R(DEFD) \cup R(D\mathrm{E1}\mathrm{D3}\mathrm{D2}\mathrm{F1}D) \cup R(D\mathrm{F1}\mathrm{D2}\mathrm{D3}\mathrm{E1}D).
    \]
    First, the identity $D\mathrm{I}\mathrm{I}DEFD = DEFD\mathrm{I}\mathrm{I}D = DEFD$ follows directly from Lemma~\ref{idempotent}. Hence, it suffices to establish the following nine inclusions:
    \begin{enumerate}
        \item $R(D\mathrm{E1}D\mathrm{E1}D) \subseteq R(DEFD)$
        \item $R(D\mathrm{E1}D\mathrm{F1}D) \subseteq R(DEFD) \cup R(D\mathrm{E1}\mathrm{D3}\mathrm{D2}\mathrm{F1}D)$
        \item $R(D\mathrm{E1}D\mathrm{E1}\mathrm{F1}D) \subseteq R(DEFD)$
        \item $R(D\mathrm{F1}D\mathrm{E1}D) \subseteq R(DEFD) \cup R(D\mathrm{F1}\mathrm{D2}\mathrm{D3}\mathrm{E1}D)$
        \item $R(D\mathrm{F1}D\mathrm{F1}D) \subseteq R(DEFD)$
        \item $R(D\mathrm{F1}D\mathrm{E1}\mathrm{F1}D) \subseteq R(DEFD) \cup R(D\mathrm{F1}\mathrm{D2}\mathrm{D3}\mathrm{E1}D)$
        \item $R(D\mathrm{E1}\mathrm{F1}D\mathrm{E1}D) \subseteq R(DEFD) \cup R(D\mathrm{F1}\mathrm{D2}\mathrm{D3}\mathrm{E1}D)$
        \item $R(D\mathrm{E1}\mathrm{F1}D\mathrm{F1}D) \subseteq R(DEFD) \cup R(D\mathrm{E1}\mathrm{D3}\mathrm{D2}\mathrm{F1}D)$
        \item $R(D\mathrm{E1}\mathrm{F1}D\mathrm{E1}\mathrm{F1}D) \subseteq R(DEFD) \cup R(D\mathrm{F1}\mathrm{D2}\mathrm{D3}\mathrm{E1}D)$
    \end{enumerate}
    We verify these one by one below.

    \begin{enumerate}
    \item For $R(D\mathrm{E1}D\mathrm{E1}D)$, we have
        \begin{equation}
            \begin{split}
            D\mathrm{E1}D\mathrm{E1}D 
            &= D\mathrm{E1}D^uD^dD^u\mathrm{E1}D \qquad\because\text{Lemma~\ref{D-split}} \\
            &= DD^u\mathrm{E1}D^d\mathrm{E1}D^uD \qquad\because\text{Lemma~\ref{E-D^u}} \\
            &= D\mathrm{E1}D^d\mathrm{E1}D \qquad\because\text{Lemma~\ref{D-merge}} \\
            &= DD^d\mathrm{E1}D \qquad\because\text{Lemma~\ref{E-D^d}} \\
            &= D\mathrm{E1}D. \qquad\because\text{Lemma~\ref{D-merge}}
            \end{split}
        \end{equation}
        Therefore, $R(D\mathrm{E1}D\mathrm{E1}D) \subseteq R(DEFD)$ indeed holds.
    \item For $R(D\mathrm{E1}D\mathrm{F1}D)$, we have
        \begin{equation}
            \begin{split}
            D\mathrm{E1}D\mathrm{F1}D &= D\mathrm{E1}D^uD^dD^u\mathrm{F1}D \qquad\because\text{Lemma~\ref{D-split}} \\
            &= DD^u\mathrm{E1}D^dD^u\mathrm{F1}D \qquad\because\text{Lemma~\ref{E-D^u}} \\
            &= D\mathrm{E1}D^dD^u\mathrm{F1}D. \qquad\because\text{Lemma~\ref{D-merge}}
            \end{split}
        \end{equation}
        Thus, it suffices to consider the three cases  
        $D\mathrm{E1}\mathrm{I}D^u\mathrm{F1}D$, $D\mathrm{E1}\mathrm{D4}D^u\mathrm{F1}D$, and $D\mathrm{E1}\mathrm{D3}D^u\mathrm{F1}D$.

        First, we have
        \begin{equation}
            \begin{split}
            D\mathrm{E1}\mathrm{I}D^u\mathrm{F1}D &= D\mathrm{E1}D^u\mathrm{F1}D \\
            &= DD^u\mathrm{E1}\mathrm{F1}D \qquad\because\text{Lemma~\ref{E-D^u}} \\
            &= D\mathrm{E1}\mathrm{F1}D. \qquad\because\text{Lemma~\ref{D-merge}}
            \end{split}
        \end{equation}
        Therefore, $R(D\mathrm{E1}\mathrm{I}D^u\mathrm{F1}D) \subseteq R(DEFD)$ holds.

        Next, we have
        \begin{equation}
            \begin{split}
            D\mathrm{E1}\mathrm{D4}D^u\mathrm{F1}D &= D\mathrm{E1}D^u\mathrm{F1}D \qquad\because\text{Lemma~\ref{E-D4}} \\
            &= DD^u\mathrm{E1}\mathrm{F1}D \qquad\because\text{Lemma~\ref{E-D^u}} \\
            &= D\mathrm{E1}\mathrm{F1}D. \qquad\because\text{Lemma~\ref{D-merge}}
            \end{split}
        \end{equation}
        Hence, $R(D\mathrm{E1}\mathrm{D4}D^u\mathrm{F1}D) \subseteq R(DEFD)$ holds.

        Finally, for $D\mathrm{E1}\mathrm{D3}D^u\mathrm{F1}D$, we decompose it into $D\mathrm{E1}\mathrm{D3}\mathrm{I}\mathrm{F1}D$, $D\mathrm{E1}\mathrm{D3}\mathrm{D1}\mathrm{F1}D$, and $D\mathrm{E1}\mathrm{D3}\mathrm{D2}\mathrm{F1}D$.

        First, we have
        \begin{equation}
            \begin{split}
            D\mathrm{E1}\mathrm{D3}\mathrm{I}\mathrm{F1}D &= D\mathrm{E1}\mathrm{D3}\mathrm{F1}D \\
            &= D\mathrm{E1}\mathrm{F1}\mathrm{D3}D \qquad\because\text{Lemma~\ref{F-D^d}} \\
            &= D\mathrm{E1}\mathrm{F1}D. \qquad\because\text{Lemma~\ref{D-merge}}
            \end{split}
        \end{equation}
        Thus, $R(D\mathrm{E1}\mathrm{D3}\mathrm{I}\mathrm{F1}D) \subseteq R(DEFD)$.

        Next, we have
        \begin{equation}
            \begin{split}
            D\mathrm{E1}\mathrm{D3}\mathrm{D1}\mathrm{F1}D &= D\mathrm{E1}\mathrm{D3}\mathrm{D1}D \qquad\because\text{Lemma~\ref{F-D1}} \\
            &= DED. \qquad\because\text{Lemma~\ref{D-merge}}
            \end{split}
        \end{equation}
        Hence, $R(D\mathrm{E1}\mathrm{D3}\mathrm{D1}\mathrm{F1}D) \subseteq R(DEFD)$.Then, the only remaining term is $D\mathrm{E1}\mathrm{D3}\mathrm{D2}\mathrm{F1}D$.\\ Therefore, $R(D\mathrm{E1}D\mathrm{F1}D) \subseteq R(DEFD) \cup R(D\mathrm{E1}\mathrm{D3}\mathrm{D2}\mathrm{F1}D)$ holds.
    \item For $R(D\mathrm{E1}D\mathrm{E1}\mathrm{F1}D)$, we have
        \begin{equation}
            \begin{split}
            D\mathrm{E1}D\mathrm{E1}\mathrm{F1}D 
            &= D\mathrm{E1}D^uD^dD^u\mathrm{E1}\mathrm{F1}D \qquad\because\text{Lemma~\ref{D-split}} \\
            &= DD^u\mathrm{E1}D^d\mathrm{E1}D^u\mathrm{F1}D \qquad\because\text{Lemma~\ref{E-D^u}} \\
            &= DD^uD^d\mathrm{E1}D^u\mathrm{F1}D \qquad\because\text{Lemma~\ref{E-D^d}} \\
            &= DD^uD^dD^u\mathrm{E1}\mathrm{F1}D \qquad\because\text{Lemma~\ref{E-D^u}} \\
            &= D\mathrm{E1}\mathrm{F1}D. \qquad\because\text{Lemma~\ref{D-merge}}
            \end{split}
        \end{equation}
        Therefore, $R(D\mathrm{E1}D\mathrm{E1}\mathrm{F1}D) \subseteq R(DEFD)$ indeed holds.
    \item For $R(D\mathrm{F1}D\mathrm{E1}D)$, we have
        \begin{equation}
            \begin{split}
            D\mathrm{F1}D\mathrm{E1}D 
            &= D\mathrm{F1}D^uD^dD^u\mathrm{E1}D \qquad\because\text{Lemma~\ref{D-split}} \\
            &= D\mathrm{F1}D^uD^d\mathrm{E1}D^uD \qquad\because\text{Lemma~\ref{E-D^u}} \\
            &= D\mathrm{F1}D^uD^d\mathrm{E1}D. \qquad\because\text{Lemma~\ref{D-merge}}
            \end{split}
        \end{equation}
        Thus, it suffices to consider the three cases $D\mathrm{F1}\mathrm{I}D^d\mathrm{E1}D$, $D\mathrm{F1}\mathrm{D1}D^d\mathrm{E1}D$, and $D\mathrm{F1}\mathrm{D2}D^d\mathrm{E1}D$.

        First, we have
        \begin{equation}
            \begin{split}
            D\mathrm{F1}\mathrm{I}D^d\mathrm{E1}D &= D\mathrm{F1}D^d\mathrm{E1}D \\
            &= DD^d\mathrm{F1}\mathrm{E1}D \qquad\because\text{Lemma~\ref{F-D^d}} \\
            &= D\mathrm{F1}\mathrm{E1}D \qquad\because\text{Lemma~\ref{D-merge}} \\
            &= D\mathrm{E1}\mathrm{F1}D. \qquad\because\text{Lemma~\ref{E-F}}
            \end{split}
        \end{equation}
        Hence, $R(D\mathrm{F1}\mathrm{I}D^d\mathrm{E1}D) \subseteq R(DEFD)$ holds.

        Next, we have
        \begin{equation}
            \begin{split}
            D\mathrm{F1}\mathrm{D1}D^d\mathrm{E1}D &= D\mathrm{F1}D^d\mathrm{E1}D \qquad\because\text{Lemma~\ref{F-D1}} \\
            &= DD^d\mathrm{F1}\mathrm{E1}D \qquad\because\text{Lemma~\ref{F-D^d}} \\
            &= D\mathrm{F1}\mathrm{E1}D \qquad\because\text{Lemma~\ref{D-merge}} \\
            &= D\mathrm{E1}\mathrm{F1}D. \qquad\because\text{Lemma~\ref{E-F}}
            \end{split}
        \end{equation}
        Thus, $R(D\mathrm{F1}\mathrm{D1}D^d\mathrm{E1}D) \subseteq R(DEFD)$ holds.

        Finally, for $D\mathrm{F1}\mathrm{D2}D^d\mathrm{E1}D$, we decompose it into $D\mathrm{F1}\mathrm{D2}\mathrm{I}\mathrm{E1}D$, $D\mathrm{F1}\mathrm{D2}\mathrm{D4}\mathrm{E1}D$, and $D\mathrm{F1}\mathrm{D2}\mathrm{D3}\mathrm{E1}D$.

        First, we have
        \begin{equation}
            \begin{split}
            D\mathrm{F1}\mathrm{D2}\mathrm{I}\mathrm{E1}D &= D\mathrm{F1}\mathrm{D2}\mathrm{E1}D \\
            &= D\mathrm{F1}\mathrm{E1}\mathrm{D2}D \qquad\because\text{Lemma~\ref{E-D^u}} \\
            &= D\mathrm{F1}\mathrm{E1}D \qquad\because\text{Lemma~\ref{D-merge}} \\
            &= D\mathrm{E1}\mathrm{F1}D. \qquad\because\text{Lemma~\ref{E-F}}
            \end{split}
        \end{equation}
        Therefore, $R(D\mathrm{F1}\mathrm{D2}\mathrm{I}\mathrm{E1}D) \subseteq R(DEFD)$ holds.

        Next, we have
        \begin{equation}
            \begin{split}
            D\mathrm{F1}\mathrm{D2}\mathrm{D4}\mathrm{E1}D &= D\mathrm{F1}\mathrm{D2}\mathrm{D4}D \qquad\because\text{Lemma~\ref{E-D4}} \\
            &= D\mathrm{F1}D. \qquad\because\text{Lemma~\ref{D-merge}}
            \end{split}
        \end{equation}
        Hence, $R(D\mathrm{F1}\mathrm{D2}\mathrm{D4}\mathrm{E1}D) \subseteq R(DEFD)$.

        The only remaining term is $D\mathrm{F1}\mathrm{D2}\mathrm{D3}\mathrm{E1}D$. Thus,$R(D\mathrm{F1}D\mathrm{E1}D) \subseteq R(DEFD) \cup R(D\mathrm{F1}\mathrm{D2}\mathrm{D3}\mathrm{E1}D)$ also holds.
    \item For $R(D\mathrm{F1}D\mathrm{F1}D)$, we have
        \begin{equation}
            \begin{split}
            D\mathrm{F1}D\mathrm{F1}D 
            &= D\mathrm{F1}D^uD^dD^u\mathrm{F1}D. \qquad\because\text{Lemma~\ref{D-split}}
            \end{split}
        \end{equation}
        Thus, it suffices to consider the three cases $D\mathrm{F1}\mathrm{I}D^dD^u\mathrm{F1}D$, $D\mathrm{F1}\mathrm{D1}D^dD^u\mathrm{F1}D$, and $D\mathrm{F1}\mathrm{D2}D^dD^u\mathrm{F1}D$.

        First, we have
        \begin{equation}
            \begin{split}
            D\mathrm{F1}\mathrm{I}D^dD^u\mathrm{F1}D 
            &= D\mathrm{F1}D^dD^u\mathrm{F1}D \\
            &= DD^d\mathrm{F1}D^u\mathrm{F1}D \qquad\because\text{Lemma~\ref{F-D^d}} \\
            &= DD^dD^u\mathrm{F1}D \qquad\because\text{Lemma~\ref{F-D^u}} \\
            &= D\mathrm{F1}D. \qquad\because\text{Lemma~\ref{D-merge}}
            \end{split}
        \end{equation}
        Hence $R(D\mathrm{F1}\mathrm{I}D^dD^u\mathrm{F1}D) \subseteq R(DEFD)$ holds.

        Next, we have
        \begin{equation}
            \begin{split}
            D\mathrm{F1}\mathrm{D1}D^dD^u\mathrm{F1}D 
            &= D\mathrm{F1}D^dD^u\mathrm{F1}D \qquad\because\text{Lemma~\ref{F-D1}} \\
            &= DD^d\mathrm{F1}D^u\mathrm{F1}D \qquad\because\text{Lemma~\ref{F-D^d}} \\
            &= DD^dD^u\mathrm{F1}D \qquad\because\text{Lemma~\ref{F-D^u}} \\
            &= D\mathrm{F1}D. \qquad\because\text{Lemma~\ref{D-merge}}
            \end{split}
        \end{equation}
        Thus $R(D\mathrm{F1}\mathrm{D1}D^dD^u\mathrm{F1}D) \subseteq R(DEFD)$ holds.

        Finally, for $D\mathrm{F1}\mathrm{D2}D^dD^u\mathrm{F1}D$, we decompose it into $D\mathrm{F1}\mathrm{D2}D^d\mathrm{I}\mathrm{F1}D$, $D\mathrm{F1}\mathrm{D2}D^d\mathrm{D1}\mathrm{F1}D$, and \\$D\mathrm{F1}\mathrm{D2}D^d\mathrm{D2}\mathrm{F1}D$.

        First, we have
        \begin{equation}
            \begin{split}
            D\mathrm{F1}\mathrm{D2}D^d\mathrm{I}\mathrm{F1}D
            &= D\mathrm{F1}\mathrm{D2}D^d\mathrm{F1}D \\
            &= D\mathrm{F1}\mathrm{D2}\mathrm{F1}D^dD \qquad\because\text{Lemma~\ref{F-D^d}} \\
            &= D\mathrm{D2}\mathrm{F1}D^dD \qquad\because\text{Lemma~\ref{F-D^u}} \\
            &= D\mathrm{F1}D. \qquad\because\text{Lemma~\ref{D-merge}}
            \end{split}
        \end{equation}
        Thus $R(D\mathrm{F1}\mathrm{D2}D^d\mathrm{I}\mathrm{F1}D) \subseteq R(DEFD)$ holds.

        Next, we have
        \begin{equation}
            \begin{split}
            D\mathrm{F1}\mathrm{D2}D^d\mathrm{D1}\mathrm{F1}D
            &= D\mathrm{F1}\mathrm{D2}D^d\mathrm{D1}D \qquad\because\text{Lemma~\ref{F-D1}} \\
            &= D\mathrm{F1}D. \qquad\because\text{Lemma~\ref{D-merge}}
            \end{split}
        \end{equation}
        Hence $R(D\mathrm{F1}\mathrm{D2}D^d\mathrm{D1}\mathrm{F1}D) \subseteq R(DEFD)$ holds.

        Finally, we have
        \begin{equation}
            \begin{split}
            D\mathrm{F1}\mathrm{D2}D^d\mathrm{D2}\mathrm{F1}D
            &= D\mathrm{F1}\mathrm{D2}D^d\mathrm{F1}D \qquad\because\text{Lemma~\ref{D2-idempotent}} \\
            &= D\mathrm{F1}\mathrm{D2}\mathrm{F1}D^dD \qquad\because\text{Lemma~\ref{F-D^d}} \\
            &= D\mathrm{D2}\mathrm{F1}D^dD \qquad\because\text{Lemma~\ref{F-D^u}} \\
            &= D\mathrm{F1}D. \qquad\because\text{Lemma~\ref{D-merge}}
            \end{split}
        \end{equation}
        Thus $R(D\mathrm{F1}\mathrm{D2}D^d\mathrm{D2}\mathrm{F1}D) \subseteq R(DEFD)$ holds. Combining all cases, we conclude that  $R(D\mathrm{F1}D\mathrm{F1}D) \subseteq R(DEFD)$ indeed holds.

    \item For $R(D\mathrm{F1}D\mathrm{E1}\mathrm{F1}D)$, we have
        \begin{equation}
            \begin{split}
            D\mathrm{F1}D\mathrm{E1}\mathrm{F1}D
            &= D\mathrm{F1}D^uD^dD^u\mathrm{E1}\mathrm{F1}D \qquad\because\text{Lemma~\ref{D-split}} \\
            &= D\mathrm{F1}D^uD^d\mathrm{E1}D^u\mathrm{F1}D. \qquad\because\text{Lemma~\ref{E-D^u}}
            \end{split}
        \end{equation}
        Thus, it suffices to consider the three cases $D\mathrm{F1}\mathrm{I}D^d\mathrm{E1}D^u\mathrm{F1}D$, $D\mathrm{F1}\mathrm{D1}D^d\mathrm{E1}D^u\mathrm{F1}D$, and \\$D\mathrm{F1}\mathrm{D2}D^d\mathrm{E1}D^u\mathrm{F1}D$.

        First, we have
        \begin{equation}
            \begin{split}
            D\mathrm{F1}\mathrm{I}D^d\mathrm{E1}D^u\mathrm{F1}D
            &= D\mathrm{F1}D^d\mathrm{E1}D^u\mathrm{F1}D \\
            &= DD^d\mathrm{F1}\mathrm{E1}D^u\mathrm{F1}D \qquad\because\text{Lemma~\ref{F-D^d}} \\
            &= DD^d\mathrm{E1}\mathrm{F1}D^u\mathrm{F1}D \qquad\because\text{Lemma~\ref{E-F}} \\
            &= DD^d\mathrm{E1}D^u\mathrm{F1}D \qquad\because\text{Lemma~\ref{F-D^u}} \\
            &= DD^dD^u\mathrm{E1}\mathrm{F1}D \qquad\because\text{Lemma~\ref{E-D^u}} \\
            &= D\mathrm{E1}\mathrm{F1}D. \qquad\because\text{Lemma~\ref{D-merge}}
            \end{split}
        \end{equation}
        Hence $R(D\mathrm{F1}\mathrm{I}D^d\mathrm{E1}D^u\mathrm{F1}D) \subseteq R(DEFD)$ holds.

        Next, we have
        \begin{equation}
            \begin{split}
            D\mathrm{F1}\mathrm{D1}D^d\mathrm{E1}D^u\mathrm{F1}D
            &= D\mathrm{F1}D^d\mathrm{E1}D^u\mathrm{F1}D \qquad\because\text{Lemma~\ref{F-D1}} \\
            &= DD^d\mathrm{F1}\mathrm{E1}D^u\mathrm{F1}D \qquad\because\text{Lemma~\ref{F-D^d}} \\
            &= DD^d\mathrm{E1}\mathrm{F1}D^u\mathrm{F1}D \qquad\because\text{Lemma~\ref{E-F}} \\
            &= DD^d\mathrm{E1}D^u\mathrm{F1}D \qquad\because\text{Lemma~\ref{F-D^u}} \\
            &= DD^dD^u\mathrm{E1}\mathrm{F1}D \qquad\because\text{Lemma~\ref{E-D^u}} \\
            &= D\mathrm{E1}\mathrm{F1}D. \qquad\because\text{Lemma~\ref{D-merge}}
            \end{split}
        \end{equation}
        Thus $R(D\mathrm{F1}\mathrm{D1}D^d\mathrm{E1}D^u\mathrm{F1}D) \subseteq R(DEFD)$ holds.

        Finally, for $D\mathrm{F1}\mathrm{D2}D^d\mathrm{E1}D^u\mathrm{F1}D$, we decompose it into $D\mathrm{F1}\mathrm{D2}D^d\mathrm{E1}\mathrm{I}\mathrm{F1}D$, $D\mathrm{F1}\mathrm{D2}D^d\mathrm{E1}\mathrm{D1}\mathrm{F1}D$, and \\$D\mathrm{F1}\mathrm{D2}D^d\mathrm{E1}\mathrm{D2}\mathrm{F1}D$.

        First, we have
        \begin{equation}
            \begin{split}
            D\mathrm{F1}\mathrm{D2}D^d\mathrm{E1}\mathrm{I}\mathrm{F1}D
            &= D\mathrm{F1}\mathrm{D2}D^d\mathrm{E1}\mathrm{F1}D \\
            &= D\mathrm{F1}\mathrm{D2}D^d\mathrm{F1}\mathrm{E1}D \qquad\because\text{Lemma~\ref{E-F}} \\
            &= D\mathrm{F1}\mathrm{D2}\mathrm{F1}D^d\mathrm{E1}D \qquad\because\text{Lemma~\ref{F-D^d}} \\
            &= D\mathrm{D2}\mathrm{F1}D^d\mathrm{E1}D \qquad\because\text{Lemma~\ref{F-D^u}} \\
            &= D\mathrm{D2}D^d\mathrm{F1}\mathrm{E1}D \qquad\because\text{Lemma~\ref{F-D^d}} \\
            &= D\mathrm{F1}\mathrm{E1}D \qquad\because\text{Lemma~\ref{D-merge}} \\
            &= D\mathrm{E1}\mathrm{F1}D. \qquad\because\text{Lemma~\ref{E-F}}
            \end{split}
        \end{equation}
        Hence $R(D\mathrm{F1}\mathrm{D2}D^d\mathrm{E1}\mathrm{I}\mathrm{F1}D) \subseteq R(DEFD)$ holds.

        Next, we have
        \begin{equation}
            \begin{split}
            D\mathrm{F1}\mathrm{D2}D^d\mathrm{E1}\mathrm{D1}\mathrm{F1}D
            &= D\mathrm{F1}\mathrm{D2}D^d\mathrm{E1}\mathrm{D1}D \qquad\because\text{Lemma~\ref{F-D1}} \\
            &= D\mathrm{F1}\mathrm{D2}D^d\mathrm{E1}D. \qquad\because\text{Lemma~\ref{D-merge}}
            \end{split}
        \end{equation}
        Since $R(\mathrm{D2}D^d)\subseteq R(D)$, we have $R(D\mathrm{F1}\mathrm{D2}D^d\mathrm{E1}\mathrm{D1}\mathrm{F1}D) \subseteq R(D\mathrm{F1}D\mathrm{E1}D)$, and from Case 4 we have $R(D\mathrm{F1}\mathrm{D2}D^d\mathrm{E1}\mathrm{D1}\mathrm{F1}D) \subseteq R(D\mathrm{F1}D\mathrm{E1}D) \subseteq R(DEFD) \cup R(D\mathrm{F1}\mathrm{D2}\mathrm{D3}\mathrm{E1}D).$

        Finally, we have
        \begin{equation}
            \begin{split}
            D\mathrm{F1}\mathrm{D2}D^d\mathrm{E1}\mathrm{D2}\mathrm{F1}D
            &= D\mathrm{F1}\mathrm{D2}D^d\mathrm{D2}\mathrm{E1}\mathrm{F1}D \qquad\because\text{Lemma~\ref{E-D^u}} \\
            &= D\mathrm{F1}\mathrm{D2}D^d\mathrm{E1}\mathrm{F1}D \qquad\because\text{Lemma~\ref{D2-idempotent}} \\
            &= D\mathrm{F1}\mathrm{D2}D^d\mathrm{F1}\mathrm{E1}D \qquad\because\text{Lemma~\ref{E-F}} \\
            &= D\mathrm{F1}\mathrm{D2}\mathrm{F1}D^d\mathrm{E1}D \qquad\because\text{Lemma~\ref{F-D^d}} \\
            &= D\mathrm{D2}\mathrm{F1}D^d\mathrm{E1}D \qquad\because\text{Lemma~\ref{F-D^u}} \\
            &= D\mathrm{D2}D^d\mathrm{F1}\mathrm{E1}D \qquad\because\text{Lemma~\ref{F-D^d}} \\
            &= D\mathrm{F1}\mathrm{E1}D \qquad\because\text{Lemma~\ref{D-merge}} \\
            &= D\mathrm{E1}\mathrm{F1}D. \qquad\because\text{Lemma~\ref{E-F}}
            \end{split}
        \end{equation}
        Thus $R(D\mathrm{F1}\mathrm{D2}D^d\mathrm{E1}\mathrm{D2}\mathrm{F1}D) \subseteq R(DEFD)$ holds. Combining all cases, we obtain $R(D\mathrm{F1}D\mathrm{E1}\mathrm{F1}D) \subseteq R(DEFD) \cup R(D\mathrm{F1}\mathrm{D2}\mathrm{D3}\mathrm{E1}D)$.
    \item For $R(D\mathrm{E1}\mathrm{F1}D\mathrm{E1}D)$, we have
        \begin{equation}
            \begin{split}
            D\mathrm{E1}\mathrm{F1}D\mathrm{E1}D
            &= D\mathrm{E1}\mathrm{F1}D^uD^dD^u\mathrm{E1}D \qquad\because\text{Lemma~\ref{D-split}} \\
            &= D\mathrm{F1}\mathrm{E1}D^uD^dD^u\mathrm{E1}D \qquad\because\text{Lemma~\ref{E-F}} \\
            &= D\mathrm{F1}D^u\mathrm{E1}D^d\mathrm{E1}D^uD \qquad\because\text{Lemma~\ref{E-D^u}} \\
            &= D\mathrm{F1}D^uD^d\mathrm{E1}D^uD \qquad\because\text{Lemma~\ref{E-D^d}} \\
            &= D\mathrm{F1}D^uD^dD^u\mathrm{E1}D \qquad\because\text{Lemma~\ref{E-D^u}} \\
            &= D\mathrm{F1}D\mathrm{E1}D. \qquad\because\text{Lemma~\ref{D-split}}
            \end{split}
        \end{equation}
        Thus, we have $R(D\mathrm{E1}\mathrm{F1}D\mathrm{E1}D) = R(D\mathrm{F1}D\mathrm{E1}D)$, and from Case 4 we have \\$R(D\mathrm{E1}\mathrm{F1}D\mathrm{E1}D) = R(D\mathrm{F1}D\mathrm{E1}D) \subseteq R(DEFD) \cup R(D\mathrm{F1}\mathrm{D2}\mathrm{D3}\mathrm{E1}D).$
    \item For $R(D\mathrm{E1}\mathrm{F1}D\mathrm{F1}D)$, we have
        \begin{equation}
            \begin{split}
            D\mathrm{E1}\mathrm{F1}D\mathrm{F1}D
            &= D\mathrm{E1}\mathrm{F1}D^uD^dD^u\mathrm{F1}D. \qquad\because\text{Lemma~\ref{D-split}}
            \end{split}
        \end{equation}
        Thus, it suffices to consider the three cases $D\mathrm{E1}\mathrm{F1}D^uD^d\mathrm{I}\mathrm{F1}D$, $D\mathrm{E1}\mathrm{F1}D^uD^d\mathrm{D1}\mathrm{F1}D$, and \\$D\mathrm{E1}\mathrm{F1}D^uD^d\mathrm{D2}\mathrm{F1}D$. First, we have
        \begin{equation}
            \begin{split}
            D\mathrm{E1}\mathrm{F1}D^uD^d\mathrm{I}\mathrm{F1}D
            &= D\mathrm{E1}\mathrm{F1}D^uD^d\mathrm{F1}D \\
            &= D\mathrm{E1}\mathrm{F1}D^u\mathrm{F1}D^dD \qquad\because\text{Lemma~\ref{F-D^d}} \\
            &= D\mathrm{E1}D^u\mathrm{F1}D^dD \qquad\because\text{Lemma~\ref{F-D^u}} \\
            &= DD^u\mathrm{E1}\mathrm{F1}D^dD \qquad\because\text{Lemma~\ref{E-D^u}} \\
            &= D\mathrm{E1}\mathrm{F1}D. \qquad\because\text{Lemma~\ref{D-merge}}
            \end{split}
        \end{equation}
        Hence $R(D\mathrm{E1}\mathrm{F1}D^uD^d\mathrm{I}\mathrm{F1}D) \subseteq R(DEFD)$ holds. Next, we have
        \begin{equation}
            \begin{split}
            D\mathrm{E1}\mathrm{F1}D^uD^d\mathrm{D1}\mathrm{F1}D
            &= D\mathrm{E1}\mathrm{F1}D^uD^d\mathrm{D1}D \qquad\because\text{Lemma~\ref{F-D1}} \\
            &= D\mathrm{E1}\mathrm{F1}D. \qquad\because\text{Lemma~\ref{D-merge}}
            \end{split}
        \end{equation}
        Thus $R(D\mathrm{E1}\mathrm{F1}D^uD^d\mathrm{D1}\mathrm{F1}D) \subseteq R(DEFD)$ holds. Finally, for $D\mathrm{E1}\mathrm{F1}D^uD^d\mathrm{D2}\mathrm{F1}D$, we decompose it into $D\mathrm{E1}\mathrm{F1}\mathrm{I}D^d\mathrm{D2}\mathrm{F1}D$, $D\mathrm{E1}\mathrm{F1}\mathrm{D1}D^d\mathrm{D2}\mathrm{F1}D$, and $D\mathrm{E1}\mathrm{F1}\mathrm{D2}D^d\mathrm{D2}\mathrm{F1}D$. First, we have
        \begin{equation}
            \begin{split}
            D\mathrm{E1}\mathrm{F1}\mathrm{I}D^d\mathrm{D2}\mathrm{F1}D
            &= D\mathrm{E1}\mathrm{F1}D^d\mathrm{D2}\mathrm{F1}D \\
            &= D\mathrm{E1}D^d\mathrm{F1}\mathrm{D2}\mathrm{F1}D \qquad\because\text{Lemma~\ref{F-D^d}} \\
            &= D\mathrm{E1}D^d\mathrm{D2}\mathrm{F1}D. \qquad\because\text{Lemma~\ref{F-D^u}}
            \end{split}
        \end{equation}
        Thus $R(D\mathrm{E1}\mathrm{F1}\mathrm{I}D^d\mathrm{D2}\mathrm{F1}D) \subseteq R(D\mathrm{E1}D\mathrm{F1}D)$, and from Case 2, \\$R(D\mathrm{E1}\mathrm{F1}\mathrm{I}D^d\mathrm{D2}\mathrm{F1}D) \subseteq R(D\mathrm{E1}D\mathrm{F1}D) \subseteq R(DEFD) \cup R(D\mathrm{E1}\mathrm{D3}\mathrm{D2}\mathrm{F1}D)$ holds. Next, we have
        \begin{equation}
            \begin{split}
            D\mathrm{E1}\mathrm{F1}\mathrm{D1}D^d\mathrm{D2}\mathrm{F1}D
            &= D\mathrm{E1}\mathrm{F1}D^d\mathrm{D2}\mathrm{F1}D \qquad\because\text{Lemma~\ref{F-D1}} \\
            &= D\mathrm{E1}D^d\mathrm{F1}\mathrm{D2}\mathrm{F1}D \qquad\because\text{Lemma~\ref{F-D^d}} \\
            &= D\mathrm{E1}D^d\mathrm{D2}\mathrm{F1}D. \qquad\because\text{Lemma~\ref{F-D^u}}
            \end{split}
        \end{equation}
        Thus $R(D\mathrm{E1}\mathrm{F1}\mathrm{D1}D^d\mathrm{D2}\mathrm{F1}D) \subseteq R(D\mathrm{E1}D\mathrm{F1}D)$, and therefore from Case 2, $R(D\mathrm{E1}\mathrm{F1}\mathrm{D1}D^d\mathrm{D2}\mathrm{F1}D) \subseteq R(D\mathrm{E1}D\mathrm{F1}D) \subseteq R(DEFD) \cup R(D\mathrm{E1}\mathrm{D3}\mathrm{D2}\mathrm{F1}D)$ holds. Finally, we have
        \begin{equation}
            \begin{split}
            D\mathrm{E1}\mathrm{F1}\mathrm{D2}D^d\mathrm{D2}\mathrm{F1}D
            &= D\mathrm{E1}\mathrm{F1}\mathrm{D2}D^d\mathrm{F1}D \qquad\because\text{Lemma~\ref{D2-idempotent}} \\
            &= D\mathrm{E1}\mathrm{F1}\mathrm{D2}\mathrm{F1}D^dD \qquad\because\text{Lemma~\ref{F-D^d}} \\
            &= D\mathrm{E1}\mathrm{D2}\mathrm{F1}D^dD \qquad\because\text{Lemma~\ref{F-D^u}} \\
            &= D\mathrm{E1}\mathrm{D2}\mathrm{F1}D. \qquad\because\text{Lemma~\ref{D-merge}}
            \end{split}
        \end{equation}
        Thus $R(D\mathrm{E1}\mathrm{F1}\mathrm{D2}D^d\mathrm{D2}\mathrm{F1}D) \subseteq R(D\mathrm{E1}D\mathrm{F1}D)$, and hence from Case 2, \\$R(D\mathrm{E1}\mathrm{F1}\mathrm{D2}D^d\mathrm{D2}\mathrm{F1}D)
        \subseteq R(D\mathrm{E1}D\mathrm{F1}D)
        \subseteq R(DEFD) \cup R(D\mathrm{E1}\mathrm{D3}\mathrm{D2}\mathrm{F1}D)$ holds. Therefore, \\$R(D\mathrm{E1}\mathrm{F1}D\mathrm{F1}D) \subseteq R(DEFD) \cup R(D\mathrm{E1}\mathrm{D3}\mathrm{D2}\mathrm{F1}D)$ holds.
    \item For $R(D\mathrm{E1}\mathrm{F1}D\mathrm{E1}\mathrm{F1}D)$, we have
        \begin{equation}
            \begin{split}
            D\mathrm{E1}\mathrm{F1}D\mathrm{E1}\mathrm{F1}D
            &= D\mathrm{E1}\mathrm{F1}D^uD^dD^u\mathrm{E1}\mathrm{F1}D \qquad\because\text{Lemma~\ref{D-split}} \\
            &= D\mathrm{F1}\mathrm{E1}D^uD^dD^u\mathrm{E1}\mathrm{F1}D \qquad\because\text{Lemma~\ref{E-F}} \\
            &= D\mathrm{F1}D^u\mathrm{E1}D^d\mathrm{E1}D^u\mathrm{F1}D \qquad\because\text{Lemma~\ref{E-D^u}} \\
            &= D\mathrm{F1}D^uD^d\mathrm{E1}D^u\mathrm{F1}D \qquad\because\text{Lemma~\ref{E-D^d}} \\
            &= D\mathrm{F1}D^uD^dD^u\mathrm{E1}\mathrm{F1}D \qquad\because\text{Lemma~\ref{E-D^u}} \\
            &= D\mathrm{F1}D\mathrm{E1}\mathrm{F1}D. \qquad\because\text{Lemma~\ref{D-split}}
            \end{split}
        \end{equation}
        Thus, we have $R(D\mathrm{E1}\mathrm{F1}D\mathrm{E1}\mathrm{F1}D) = R(D\mathrm{F1}D\mathrm{E1}\mathrm{F1}D)$, and from Case 6, \\$R(D\mathrm{E1}\mathrm{F1}D\mathrm{E1}\mathrm{F1}D) = R(D\mathrm{F1}D\mathrm{E1}\mathrm{F1}D)\subseteq R(DEFD) \cup R(D\mathrm{F1}\mathrm{D2}\mathrm{D3}\mathrm{E1}D)$ holds.
\end{enumerate}

Combining all the above results, we conclude that
\[
    R(DEFDEFD)
        = R(DEFD) \cup R(D\mathrm{E1}\mathrm{D3}\mathrm{D2}\mathrm{F1}D) \cup R(D\mathrm{F1}\mathrm{D2}\mathrm{D3}\mathrm{E1}D)
\]
indeed holds.
\end{proof}

\begin{thm}[Reduction of Ten Operation Sets]
    \label{reduction_10}
    The identity
    \[
        R(DEFDEFDEFD)
            = R(DEFDEFD) \cup R(D\mathrm{F1}\mathrm{D2}\mathrm{D3}\mathrm{D1}\mathrm{E1}\mathrm{D3}\mathrm{D2}\mathrm{F1}D)
    \]
    holds.
\end{thm}

\begin{proof}
The inclusion 
\[
    R(DEFDEFD) \cup R(D\mathrm{F1}\mathrm{D2}\mathrm{D3}\mathrm{D1}\mathrm{E1}\mathrm{D3}\mathrm{D2}\mathrm{F1}D) \subseteq R(DEFDEFDEFD)
\]
is immediate. Therefore, it suffices to prove the reverse inclusion.  
From Theorem~\ref{reduction_7}, we have
\begin{equation}
    \label{10_split_1}
    R(DEFDEFDEFD)
        = R(DEFDEFD)
        \cup R(D\mathrm{E1}\mathrm{D3}\mathrm{D2}\mathrm{F1}DEFD)
        \cup R(D\mathrm{F1}\mathrm{D2}\mathrm{D3}\mathrm{E1}DEFD).
\end{equation}
Similarly, again by Theorem~\ref{reduction_7}, we also have
\begin{gather}
    \begin{split}
        R(D\mathrm{E1}\mathrm{D3}\mathrm{D2}\mathrm{F1}DEFD)
            &\subseteq R(D\mathrm{E1}DEFDEFD) \\
            &= R(D\mathrm{E1}DEFD)
                \cup R(D\mathrm{E1}D\mathrm{E1}\mathrm{D3}\mathrm{D2}\mathrm{F1}D)
                \cup R(D\mathrm{E1}D\mathrm{F1}\mathrm{D2}\mathrm{D3}\mathrm{E1}D),
    \end{split}
    \label{10_split_2}
    \\
    \begin{split}
        R(D\mathrm{F1}\mathrm{D2}\mathrm{D3}\mathrm{E1}DEFD)
            &\subseteq R(D\mathrm{F1}DEFDEFD) \\
            &= R(D\mathrm{F1}DEFD)
                \cup R(D\mathrm{F1}D\mathrm{E1}\mathrm{D3}\mathrm{D2}\mathrm{F1}D)
                \cup R(D\mathrm{F1}D\mathrm{F1}\mathrm{D2}\mathrm{D3}\mathrm{E1}D).
    \end{split}
    \label{10_split_3}
\end{gather}

First, we have
\begin{equation}
    \begin{split}
        R(D\mathrm{E1}D\mathrm{E1}\mathrm{D3}\mathrm{D2}\mathrm{F1}D)
            &\subseteq R(D\mathrm{E1}D\mathrm{E1}D\mathrm{F1}D) \\
            &\subseteq R(DEFD\mathrm{F1}D)
                \qquad(\text{since } R(D\mathrm{E1}D\mathrm{E1}D)\subseteq R(DEFD)) \\
            &\subseteq R(DEFDEFD),
    \end{split}
\end{equation}
where the second inclusion is due to Case 1 in the proof of Theorem\ref{reduction_7}. Similarly, we also have
\begin{equation}
    \begin{split}
        R(D\mathrm{F1}D\mathrm{F1}\mathrm{D2}\mathrm{D3}\mathrm{E1}D)
            &\subseteq R(D\mathrm{F1}D\mathrm{F1}D\mathrm{E1}D) \\
            &\subseteq R(DEFD\mathrm{E1}D)
                \qquad(\text{since } R(D\mathrm{F1}D\mathrm{F1}D)\subseteq R(DEFD)) \\
            &\subseteq R(DEFDEFD),
    \end{split}
\end{equation}
where the second inclusion is due to Case 5 in the proof of Theorem\ref{reduction_7}.

Thus, it remains to consider the two sets  
\[
    R(D\mathrm{E1}D\mathrm{F1}\mathrm{D2}\mathrm{D3}\mathrm{E1}D), \qquad
    R(D\mathrm{F1}D\mathrm{E1}\mathrm{D3}\mathrm{D2}\mathrm{F1}D).
\]

% -------------------------------
% The D\mathrm{E1}D\mathrm{F1}\mathrm{D2}\mathrm{D3}\mathrm{E1}D branch
% -------------------------------

We begin with $D\mathrm{E1}D\mathrm{F1}\mathrm{D2}\mathrm{D3}\mathrm{E1}D$:
\begin{equation}
    \begin{split}
        D\mathrm{E1}D\mathrm{F1}\mathrm{D2}\mathrm{D3}\mathrm{E1}D
            &= D\mathrm{E1}D^uD^dD^u\mathrm{F1}\mathrm{D2}\mathrm{D3}\mathrm{E1}D
                \qquad\because\text{Lemma~\ref{D-split}} \\
            &= DD^u\mathrm{E1}D^dD^u\mathrm{F1}\mathrm{D2}\mathrm{D3}\mathrm{E1}D
                \qquad\because\text{Lemma~\ref{E-D^u}} \\
            &= D\mathrm{E1}D^dD^u\mathrm{F1}\mathrm{D2}\mathrm{D3}\mathrm{E1}D
                \qquad\because\text{Lemma~\ref{D-merge}}
    \end{split}
\end{equation}
holds. Thus it suffices to consider  
$D\mathrm{E1}\mathrm{I}D^u\mathrm{F1}\mathrm{D2}\mathrm{D3}\mathrm{E1}D$, $D\mathrm{E1}\mathrm{D4}D^u\mathrm{F1}\mathrm{D2}\mathrm{D3}\mathrm{E1}D$, and $D\mathrm{E1}\mathrm{D3}D^u\mathrm{F1}\mathrm{D2}\mathrm{D3}\mathrm{E1}D$.

First, we have
\begin{equation}
    \begin{split}
        D\mathrm{E1}\mathrm{I}D^u\mathrm{F1}\mathrm{D2}\mathrm{D3}\mathrm{E1}D
            &= D\mathrm{E1}D^u\mathrm{F1}\mathrm{D2}\mathrm{D3}\mathrm{E1}D \\
            &= DD^u\mathrm{E1}\mathrm{F1}\mathrm{D2}\mathrm{D3}\mathrm{E1}D \qquad\because\text{Lemma~\ref{E-D^u}} \\
            &= D\mathrm{E1}\mathrm{F1}\mathrm{D2}\mathrm{D3}\mathrm{E1}D, \qquad\because\text{Lemma~\ref{D-merge}}
    \end{split}
\end{equation}
hence $R(D\mathrm{E1}\mathrm{I}D^u\mathrm{F1}\mathrm{D2}\mathrm{D3}\mathrm{E1}D)\subseteq R(DEFDEFD)$.

Next, we have
\begin{equation}
    \begin{split}
        D\mathrm{E1}\mathrm{D4}D^u\mathrm{F1}\mathrm{D2}\mathrm{D3}\mathrm{E1}D
            &= D\mathrm{E1}D^u\mathrm{F1}\mathrm{D2}\mathrm{D3}\mathrm{E1}D \qquad\because\text{Lemma~\ref{E-D4}} \\
            &= DD^u\mathrm{E1}\mathrm{F1}\mathrm{D2}\mathrm{D3}\mathrm{E1}D \qquad\because\text{Lemma~\ref{E-D^u}} \\
            &= D\mathrm{E1}\mathrm{F1}\mathrm{D2}\mathrm{D3}\mathrm{E1}D, \qquad\because\text{Lemma~\ref{D-merge}}
    \end{split}
\end{equation}
so this subset inclusion also holds.

Now, $D\mathrm{E1}\mathrm{D3}D^u\mathrm{F1}\mathrm{D2}\mathrm{D3}\mathrm{E1}D$ splits into  
$D\mathrm{E1}\mathrm{D3}\mathrm{I}\mathrm{F1}\mathrm{D2}\mathrm{D3}\mathrm{E1}D$, $D\mathrm{E1}\mathrm{D3}\mathrm{D1}\mathrm{F1}\mathrm{D2}\mathrm{D3}\mathrm{E1}D$, and \\$D\mathrm{E1}\mathrm{D3}\mathrm{D2}\mathrm{F1}\mathrm{D2}\mathrm{D3}\mathrm{E1}D$.

First, we have
\begin{equation}
    \begin{split}
        D\mathrm{E1}\mathrm{D3}\mathrm{I}\mathrm{F1}\mathrm{D2}\mathrm{D3}\mathrm{E1}D
            &= D\mathrm{E1}\mathrm{D3}\mathrm{F1}\mathrm{D2}\mathrm{D3}\mathrm{E1}D \\
            &= D\mathrm{E1}\mathrm{F1}\mathrm{D3}\mathrm{D2}\mathrm{D3}\mathrm{E1}D,\qquad\because\text{Lemma~\ref{F-D^d}}
    \end{split}
\end{equation}
and since $R(\mathrm{D3}\mathrm{D2})\subseteq R(D)$, we obtain  
$R(D\mathrm{E1}\mathrm{D3}\mathrm{I}\mathrm{F1}\mathrm{D2}\mathrm{D3}\mathrm{E1}D)\subseteq R(DEFDEFD)$.

Next, we have
\begin{equation}
    \begin{split}
            D\mathrm{E1}\mathrm{D3}\mathrm{D1}\mathrm{F1}\mathrm{D2}\mathrm{D3}\mathrm{E1}D
            &= D\mathrm{E1}\mathrm{D3}\mathrm{D1}\mathrm{D2}\mathrm{D3}\mathrm{E1}D,\qquad\because\text{Lemma~\ref{F-D1}}
    \end{split}
\end{equation}
and since $R(\mathrm{D1}\mathrm{D2}\mathrm{D3})\subseteq R(D)$, again it follows that  
$R(D\mathrm{E1}\mathrm{D3}\mathrm{D1}\mathrm{F1}\mathrm{D2}\mathrm{D3}\mathrm{E1}D)\subseteq R(DEFDEFD)$.

Finally, we have
\begin{equation}
    \begin{split}
        D\mathrm{E1}\mathrm{D3}\mathrm{D2}\mathrm{F1}\mathrm{D2}\mathrm{D3}\mathrm{E1}D
            &= D\mathrm{E1}\mathrm{D3}\mathrm{D2}\mathrm{F1}\mathrm{D3}\mathrm{E1}D
                \qquad\because\text{Lemma~\ref{D2-idempotent}} \\
            &= D\mathrm{E1}\mathrm{D3}\mathrm{D2}\mathrm{D3}\mathrm{F1}\mathrm{E1}D
                \qquad\because\text{Lemma~\ref{F-D^d}} \\
            &= D\mathrm{E1}\mathrm{D3}\mathrm{D2}\mathrm{D3}\mathrm{E1}\mathrm{F1}D,
                \qquad\because\text{Lemma~\ref{E-F}}
    \end{split}
\end{equation}
and since $R(\mathrm{D3}\mathrm{D2}\mathrm{D3})\subseteq R(D)$, this case also lies in $R(DEFDEFD)$.

Therefore, we have
\[
    R(D\mathrm{E1}D\mathrm{F1}\mathrm{D2}\mathrm{D3}\mathrm{E1}D) \subseteq R(DEFDEFD).
\]

% -------------------------------
% The D\mathrm{F1}D\mathrm{E1}\mathrm{D3}\mathrm{D2}\mathrm{F1}D branch
% -------------------------------

We now examine $R(D\mathrm{F1}D\mathrm{E1}\mathrm{D3}\mathrm{D2}\mathrm{F1}D)$.

\begin{equation}
    \begin{split}
        D\mathrm{F1}D\mathrm{E1}\mathrm{D3}\mathrm{D2}\mathrm{F1}D
            &= D\mathrm{F1}D^uD^dD^u\mathrm{E1}\mathrm{D3}\mathrm{D2}\mathrm{F1}D
                \qquad\because\text{Lemma~\ref{D-split}} \\
            &= D\mathrm{F1}D^uD^d\mathrm{E1}D^u\mathrm{D3}\mathrm{D2}\mathrm{F1}D
                \qquad\because\text{Lemma~\ref{E-D^u}}
    \end{split}
\end{equation}
holds. Thus, we consider  
$D\mathrm{F1}\mathrm{I}D^d\mathrm{E1}D^u\mathrm{D3}\mathrm{D2}\mathrm{F1}D$, $D\mathrm{F1}\mathrm{D1}D^d\mathrm{E1}D^u\mathrm{D3}\mathrm{D2}\mathrm{F1}D$, and $D\mathrm{F1}\mathrm{D2}D^d\mathrm{E1}D^u\mathrm{D3}\mathrm{D2}\mathrm{F1}D$.

First, we have
\begin{equation}
    \begin{split}
        D\mathrm{F1}\mathrm{I}D^d\mathrm{E1}D^u\mathrm{D3}\mathrm{D2}\mathrm{F1}D
            &= D\mathrm{F1}D^d\mathrm{E1}D^u\mathrm{D3}\mathrm{D2}\mathrm{F1}D \\
            &= DD^d\mathrm{F1}\mathrm{E1}D^u\mathrm{D3}\mathrm{D2}\mathrm{F1}D
                \qquad\because\text{Lemma~\ref{F-D^d}} \\
            &= D\mathrm{F1}\mathrm{E1}D^u\mathrm{D3}\mathrm{D2}\mathrm{F1}D
                \qquad\because\text{Lemma~\ref{D-merge}} \\
            &= D\mathrm{E1}\mathrm{F1}D^u\mathrm{D3}\mathrm{D2}\mathrm{F1}D.
                \qquad\because\text{Lemma~\ref{E-F}}
    \end{split}
\end{equation}
Since $R(D^u\mathrm{D3}\mathrm{D2})\subseteq R(D)$, it follows that  
$R(D\mathrm{F1}\mathrm{I}D^d\mathrm{E1}D^u\mathrm{D3}\mathrm{D2}\mathrm{F1}D)\subseteq R(DEFDEFD)$.

Next, we have
\begin{equation}
    \begin{split}
        D\mathrm{F1}\mathrm{D1}D^d\mathrm{E1}D^u\mathrm{D3}\mathrm{D2}\mathrm{F1}D
            &= D\mathrm{F1}D^d\mathrm{E1}D^u\mathrm{D3}\mathrm{D2}\mathrm{F1}D
                \qquad\because\text{Lemma~\ref{F-D1}} \\
            &= DD^d\mathrm{F1}\mathrm{E1}D^u\mathrm{D3}\mathrm{D2}\mathrm{F1}D
                \qquad\because\text{Lemma~\ref{F-D^d}} \\
            &= D\mathrm{F1}\mathrm{E1}D^u\mathrm{D3}\mathrm{D2}\mathrm{F1}D
                \qquad\because\text{Lemma~\ref{D-merge}} \\
            &= D\mathrm{E1}\mathrm{F1}D^u\mathrm{D3}\mathrm{D2}\mathrm{F1}D.
                \qquad\because\text{Lemma~\ref{E-F}}
    \end{split}
\end{equation}
Hence $R(D\mathrm{F1}\mathrm{D1}D^d\mathrm{E1}D^u\mathrm{D3}\mathrm{D2}\mathrm{F1}D)\subseteq R(DEFDEFD)$.

The third case, $D\mathrm{F1}\mathrm{D2}D^d\mathrm{E1}D^u\mathrm{D3}\mathrm{D2}\mathrm{F1}D$, splits into  
$D\mathrm{F1}\mathrm{D2}\mathrm{I}\mathrm{E1}D^u\mathrm{D3}\mathrm{D2}\mathrm{F1}D$, $D\mathrm{F1}\mathrm{D2}\mathrm{D4}\mathrm{E1}D^u\mathrm{D3}\mathrm{D2}\mathrm{F1}D$, and $D\mathrm{F1}\mathrm{D2}\mathrm{D3}\mathrm{E1}D^u\mathrm{D3}\mathrm{D2}\mathrm{F1}D$.

First, we have
\begin{equation}
    \begin{split}
        D\mathrm{F1}\mathrm{D2}\mathrm{I}\mathrm{E1}D^u\mathrm{D3}\mathrm{D2}\mathrm{F1}D &= D\mathrm{F1}\mathrm{D2}\mathrm{E1}D^u\mathrm{D3}\mathrm{D2}\mathrm{F1}D \\
            &= D\mathrm{F1}\mathrm{E1}\mathrm{D2}D^u\mathrm{D3}\mathrm{D2}\mathrm{F1}D, \qquad\because\text{Lemma~\ref{E-D^u}}
    \end{split}
\end{equation}
and $R(\mathrm{D2}D^u\mathrm{D3})\subseteq R(D)$ gives  
$R(D\mathrm{F1}\mathrm{D2}\mathrm{I}\mathrm{E1}D^u\mathrm{D3}\mathrm{D2}\mathrm{F1}D)\subseteq R(DEFDEFD)$.

Next, we have
\begin{equation}
    \begin{split}
        D\mathrm{F1}\mathrm{D2}\mathrm{D4}\mathrm{E1}D^u\mathrm{D3}\mathrm{D2}\mathrm{F1}D
            &= D\mathrm{F1}\mathrm{D2}\mathrm{D4}D^u\mathrm{D3}\mathrm{D2}\mathrm{F1}D,
                \qquad\because\text{Lemma~\ref{E-D4}}
    \end{split}
\end{equation}
and since $R(\mathrm{D2}\mathrm{D4}D^u\mathrm{D3}\mathrm{D2})\subseteq R(D)$, this also lies in $R(DEFDEFD)$.

Finally, $D\mathrm{F1}\mathrm{D2}\mathrm{D3}\mathrm{E1}D^u\mathrm{D3}\mathrm{D2}\mathrm{F1}D$ splits into three cases:
$D\mathrm{F1}\mathrm{D2}\mathrm{D3}\mathrm{E1}\mathrm{I}\mathrm{D3}\mathrm{D2}\mathrm{F1}D,~ \\D\mathrm{F1}\mathrm{D2}\mathrm{D3}\mathrm{E1}\mathrm{D2}\mathrm{D3}\mathrm{D2}\mathrm{F1}D,~ D\mathrm{F1}\mathrm{D2}\mathrm{D3}\mathrm{E1}\mathrm{D1}\mathrm{D3}\mathrm{D2}\mathrm{F1}D.$

First, we have
\begin{equation}
    \begin{split}
        D\mathrm{F1}\mathrm{D2}\mathrm{D3}\mathrm{E1}\mathrm{I}\mathrm{D3}\mathrm{D2}\mathrm{F1}D
            &= D\mathrm{F1}\mathrm{D2}\mathrm{D3}\mathrm{E1}\mathrm{D3}\mathrm{D2}\mathrm{F1}D \\
            &= D\mathrm{F1}\mathrm{D2}\mathrm{D3}\mathrm{E1}\mathrm{D2}\mathrm{F1}D
                \qquad\because\text{Lemma~\ref{D3-idempotent}} \\
            &= D\mathrm{F1}\mathrm{D2}\mathrm{D3}\mathrm{D2}\mathrm{E1}\mathrm{F1}D.
                \qquad\because\text{Lemma~\ref{E-D^u}}
    \end{split}
\end{equation}
Since $R(\mathrm{D2}\mathrm{D3}\mathrm{D2})\subseteq R(D)$, we obtain  
$R(D\mathrm{F1}\mathrm{D2}\mathrm{D3}\mathrm{E1}\mathrm{I}\mathrm{D3}\mathrm{D2}\mathrm{F1}D)\subseteq R(DEFDEFD)$.

Next, we have
\begin{equation}
    \begin{split}
        D\mathrm{F1}\mathrm{D2}\mathrm{D3}\mathrm{E1}\mathrm{D2}\mathrm{D3}\mathrm{D2}\mathrm{F1}D
            &= D\mathrm{F1}\mathrm{D2}\mathrm{D3}\mathrm{E1}\mathrm{D2}\mathrm{D3}\mathrm{F1}D
                \qquad\because\text{Lemma~\ref{D2-idempotent}} \\
            &= D\mathrm{F1}\mathrm{D2}\mathrm{D3}\mathrm{D2}\mathrm{E1}\mathrm{D3}\mathrm{F1}D
                \qquad\because\text{Lemma~\ref{E-D^u}} \\
            &= D\mathrm{F1}\mathrm{D2}\mathrm{D3}\mathrm{E1}\mathrm{D3}\mathrm{F1}D
                \qquad\because\text{Lemma~\ref{D2-idempotent}} \\
            &= D\mathrm{F1}\mathrm{D2}\mathrm{D3}\mathrm{E1}\mathrm{F1}D.
                \qquad\because\text{Lemma~\ref{D3-idempotent}}
    \end{split}
\end{equation}
Since $R(\mathrm{D2}\mathrm{D3})\subseteq R(D)$, this set also lies in $R(DEFDEFD)$.

Finally, we have
\[
    D\mathrm{F1}\mathrm{D2}\mathrm{D3}\mathrm{E1}\mathrm{D1}\mathrm{D3}\mathrm{D2}\mathrm{F1}D = D\mathrm{F1}\mathrm{D2}\mathrm{D3}\mathrm{D1}\mathrm{E1}\mathrm{D3}\mathrm{D2}\mathrm{F1}D, \qquad\because\text{Lemma~\ref{E-D^u}}
\]
and therefore we have 
\[
    R(D\mathrm{F1}D\mathrm{E1}\mathrm{D3}\mathrm{D2}\mathrm{F1}D)
        = R(DEFDEFD) \cup R(D\mathrm{F1}\mathrm{D2}\mathrm{D3}\mathrm{D1}\mathrm{E1}\mathrm{D3}\mathrm{D2}\mathrm{F1}D).
\]

% -------------------------------
% Final combination
% -------------------------------

Here, from equations \eqref{10_split_1}, \eqref{10_split_2}, and \eqref{10_split_3}, we have
\begin{equation}
    \begin{split}
        R(DEFDEFDEFD)
            &= R(DEFDEFD)
                \cup R(D\mathrm{E1}\mathrm{D3}\mathrm{D2}\mathrm{F1}DEFD)
                \cup R(D\mathrm{F1}\mathrm{D2}\mathrm{D3}\mathrm{E1}DEFD) \\
            &= R(DEFDEFD) \\
            &\phantom{=} \cup R(D\mathrm{E1}DEFD)
                \cup R(D\mathrm{E1}D\mathrm{E1}\mathrm{D3}\mathrm{D2}\mathrm{F1}D)
                \cup R(D\mathrm{E1}D\mathrm{F1}\mathrm{D2}\mathrm{D3}\mathrm{E1}D) \\
            &\phantom{=} \cup R(D\mathrm{F1}DEFD)
                \cup R(D\mathrm{F1}D\mathrm{E1}\mathrm{D3}\mathrm{D2}\mathrm{F1}D)
                \cup R(D\mathrm{F1}D\mathrm{F1}\mathrm{D2}\mathrm{D3}\mathrm{E1}D).
    \end{split}
\end{equation}
From all the results, all the terms except the sixth are subsets of $R(DEFDEFD)$. Hence, we have
\begin{equation}
    R(DEFDEFDEFD)
        \subseteq
        R(DEFDEFD) \cup R(D\mathrm{F1}\mathrm{D2}\mathrm{D3}\mathrm{D1}\mathrm{E1}\mathrm{D3}\mathrm{D2}\mathrm{F1}D).
\end{equation}

Therefore, we have
\begin{equation}
    R(DEFDEFDEFD)
        = R(DEFDEFD) \cup R(D\mathrm{F1}\mathrm{D2}\mathrm{D3}\mathrm{D1}\mathrm{E1}\mathrm{D3}\mathrm{D2}\mathrm{F1}D),
\end{equation}
as claimed.
\end{proof}

\begin{thm}[Reduction of Thirteen Operation Sets]
    \label{reduction_13}
    The identity
    \[
        R(DEFDEFDEFDEFD) = R(DEFDEFDEFD)
    \]
    holds.
\end{thm}

\begin{proof}
The inclusion  
\[
    R(DEFDEFDEFD) \subseteq R(DEFDEFDEFDEFD)
\]
is immediate. Hence it suffices to show the reverse inclusion.  
By Theorem~\ref{reduction_10},
\begin{equation}
    R(DEFDEFDEFDEFD)
        = R(DEFDEFDEFD)
        \cup R(D\mathrm{F1}\mathrm{D2}\mathrm{D3}\mathrm{D1}\mathrm{E1}\mathrm{D3}\mathrm{D2}\mathrm{F1}DEFD).
\end{equation}
We therefore study the remaining term  
\(
    R(D\mathrm{F1}\mathrm{D2}\mathrm{D3}\mathrm{D1}\mathrm{E1}\mathrm{D3}\mathrm{D2}\mathrm{F1}DEFD).
\)

Consider the matrix pair
    \begin{equation}
        P = \begin{pmatrix}
            P_{11} & P_{12} & P_{13} \\
            P_{12} & P_{22} & P_{23} \\
            P_{13} & P_{23} & P_{33} \\
        \end{pmatrix},~Q=\begin{pmatrix}
            Q_{11} & Q_{12} & Q_{13} & Q_{14} & Q_{15} \\
            Q_{12} & Q_{22} & Q_{23} & Q_{24} & Q_{25} \\
            Q_{13} & Q_{23} & Q_{33} & Q_{34} & Q_{35} \\
            Q_{14} & Q_{24} & Q_{34} & Q_{44} & Q_{45} \\
            Q_{15} & Q_{25} & Q_{35} & Q_{45} & Q_{55} \\
        \end{pmatrix}.
    \end{equation}
Let \(P',Q'\) denote the matrices obtained by applying the sequence of operations  
\(\mathrm{F1}\mathrm{D2}\mathrm{D3}\mathrm{D1}\mathrm{E1}\mathrm{D3}\mathrm{D2}\mathrm{F1}\) to \((P,Q)\).

We employ
\begin{equation}
    g(Q_{ij}) = \dot{\gamma} Q_{ij} + \dot{Q}_{ij},
\end{equation}
which is defined in equation \eqref{definition_g}. Then we have
\begin{gather}
    P' =
    \begin{pmatrix}
        P_{11} + \lambda (Q_{14} - Q_{23} + g(Q_{25})) &
        P_{12} - g(Q_{25}) &
        P_{13} \\
        P_{12} - g(Q_{25}) &
        P_{22} &
        P_{23} + Q_{25} \\
        P_{13} &
        P_{23} + Q_{25} &
        P_{33}
    \end{pmatrix},
    \\
    Q' =
    \begin{pmatrix}
        Q_{11} - \lambda(g(Q_{14}) - g(Q_{23}) + g(g(Q_{25}))) &
        Q_{12} - g(Q_{23}) &
        Q_{13} &
        -g(Q_{25}) &
        Q_{15} \\
        Q_{12} - g(Q_{23}) &
        Q_{22} &
        -g(Q_{25}) &
        Q_{24} &
        0 \\
        Q_{13} &
        -g(Q_{25}) &
        Q_{33} + 2\theta(Q_{34} - Q_{25}) &
        0 &
        Q_{35} \\
        -g(Q_{25}) &
        Q_{24} &
        0 &
        Q_{44} &
        Q_{45} \\
        Q_{15} &
        0 &
        Q_{35} &
        Q_{45} &
        Q_{55}
    \end{pmatrix}.
\end{gather}

Observe that the \((2,5)\) and \((3,4)\) entries of \(Q'\) are zero.  
Hence the operations \(\mathrm{D1}, \mathrm{D2}, \mathrm{F1}\) do not alter these entries—and the same remains true after applying \(\mathrm{D3}\) or \(\mathrm{D4}\).  
Thus the admissible \(D\)-operations on \((P',Q')\) are only
\[
    \mathrm{I},\ \mathrm{D3},\ \mathrm{D4}.
\]
Similarly, after applying \(D\) and \(E\), the only admissible \(F\)-operation is
\[
    \mathrm{I}.
\]

Therefore, among  
\[
    \mathrm{F1}\mathrm{D2}\mathrm{D3}\mathrm{D1}\mathrm{E1}\mathrm{D3}\mathrm{D2}\mathrm{F1}DEFD,
\]
we only need to consider the following three cases:
\[
    \mathrm{F1}\mathrm{D2}\mathrm{D3}\mathrm{D1}\mathrm{E1}\mathrm{D3}\mathrm{D2}\mathrm{F1}\mathrm{I}ED,\quad
    \mathrm{F1}\mathrm{D2}\mathrm{D3}\mathrm{D1}\mathrm{E1}\mathrm{D3}\mathrm{D2}\mathrm{F1}\mathrm{D3}ED,\quad
    \mathrm{F1}\mathrm{D2}\mathrm{D3}\mathrm{D1}\mathrm{E1}\mathrm{D3}\mathrm{D2}\mathrm{F1}\mathrm{D4}ED.
\]

For the first case, we have
\begin{equation}
    \begin{split}
        \mathrm{F1}\mathrm{D2}\mathrm{D3}\mathrm{D1}\mathrm{E1}\mathrm{D3}\mathrm{D2}\mathrm{F1}\mathrm{I}ED
            &= \mathrm{F1}\mathrm{D2}\mathrm{D3}\mathrm{D1}\mathrm{E1}\mathrm{D3}\mathrm{D2}\mathrm{F1}ED \\
            &= \mathrm{F1}\mathrm{D2}\mathrm{D3}\mathrm{D1}\mathrm{E1}\mathrm{D3}\mathrm{D2}E\mathrm{F1}D.
                \qquad\because\text{Lemma~\ref{E-F}}.
    \end{split}
\end{equation}
Since \(R(\mathrm{D2}\mathrm{D3}\mathrm{D1})\subseteq R(D)\) and \(R(\mathrm{D3}\mathrm{D2})\subseteq R(D)\), we obtain  
\[
    R(\mathrm{F1}\mathrm{D2}\mathrm{D3}\mathrm{D1}\mathrm{E1}\mathrm{D3}\mathrm{D2}\mathrm{F1}\mathrm{I}ED)
        \subseteq R(EFDEFDEFD).
\]

For the second case, applying \(\mathrm{D3}\) to \((P',Q')\) yields
    \begin{gather}
        \begin{pmatrix}
            P_{11} + \lambda (Q_{14} - Q_{23} + g(Q_{25})) & P_{12} - 2 g(Q_{25}) & P_{13} \\
            P_{12} - 2 g(Q_{25}) & P_{22} & P_{23} + Q_{25} \\
            P_{13} & P_{23} + Q_{25} & P_{33} \\
        \end{pmatrix} \\
        \begin{pmatrix}
            Q_{11} - \lambda(g(Q_{14}) - g(Q_{23})+ g(g(Q_{25})))& Q_{12} -g(Q_{23}) + g(g(Q_{25})) & Q_{13} & 0 & Q_{15} \\
            Q_{12} -g(Q_{23}) + g(g(Q_{25})) & Q_{22} & 0 & Q_{24} & 0 \\
            Q_{13} & 0 & Q_{33} + 2\theta (Q_{34} - Q_{25}) & 0 & Q_{35} \\
            0 & Q_{24} & 0 & Q_{44} & Q_{45} \\
            Q_{15} & 0 & Q_{35} & Q_{45} & Q_{55} \\
        \end{pmatrix}.
    \end{gather}

Since the \((1,4)\) entry is now \(0\), applying \(\mathrm{E1}\) does not change the matrices. Hence, we have
\begin{equation}
    \begin{split}
        \mathrm{F1}\mathrm{D2}\mathrm{D3}\mathrm{D1}\mathrm{E1}\mathrm{D3}\mathrm{D2}\mathrm{F1}\mathrm{D3}ED
            &= \mathrm{F1}\mathrm{D2}\mathrm{D3}\mathrm{D1}\mathrm{E1}\mathrm{D3}\mathrm{D2}\mathrm{F1}\mathrm{D3}D \\
            &= \mathrm{F1}\mathrm{D2}\mathrm{D3}\mathrm{D1}\mathrm{E1}\mathrm{D3}\mathrm{D2}\mathrm{F1}D.
                \qquad\because\text{Lemma~\ref{D-merge}}.
    \end{split}
\end{equation}
Again using \(R(\mathrm{D2}\mathrm{D3}\mathrm{D1})\subseteq R(D)\) and \(R(\mathrm{D3}\mathrm{D2})\subseteq R(D)\), we have
\[
R(\mathrm{F1}\mathrm{D2}\mathrm{D3}\mathrm{D1}\mathrm{E1}\mathrm{D3}\mathrm{D2}\mathrm{F1}\mathrm{D3}ED)
        \subseteq R(EFDEFDEFD).
\]

For the third case, we have 
\begin{equation}
    \begin{split}
        \mathrm{F1}\mathrm{D2}\mathrm{D3}\mathrm{D1}\mathrm{E1}\mathrm{D3}\mathrm{D2}\mathrm{F1}\mathrm{D4}ED
            &= \mathrm{F1}\mathrm{D2}\mathrm{D3}\mathrm{D1}\mathrm{E1}\mathrm{D3}\mathrm{D2}\mathrm{F1}\mathrm{D4}D
                \qquad\because\text{Lemma~\ref{E-D4}} \\
            &= \mathrm{F1}\mathrm{D2}\mathrm{D3}\mathrm{D1}\mathrm{E1}\mathrm{D3}\mathrm{D2}\mathrm{F1}D.
                \qquad\because\text{Lemma~\ref{D-merge}}
    \end{split}
\end{equation}
Thus, we have
\[
R(\mathrm{F1}\mathrm{D2}\mathrm{D3}\mathrm{D1}\mathrm{E1}\mathrm{D3}\mathrm{D2}\mathrm{F1}\mathrm{D4}ED)
        \subseteq R(EFDEFDEFD).
\]

Combining the discussion for the three cases mentioned above, we have
\begin{equation}
    \begin{split}
        R(DEFDEFDEFDEFD)
            &= R(DEFDEFDEFD)
                \cup R(D\mathrm{F1}\mathrm{D2}\mathrm{D3}\mathrm{D1}\mathrm{E1}\mathrm{D3}\mathrm{D2}\mathrm{F1}DEFD) \\
            &\subseteq R(DEFDEFDEFD).
    \end{split}
\end{equation}

Therefore, we have
\[
    R(DEFDEFDEFDEFD) = R(DEFDEFDEFD),
\]
as claimed.
\end{proof}

\begin{thm}[Combinations of $E$, $F$, and $D$]
    \label{combinations}
    For any matrix pair, the set of all matrix pairs obtained by applying 
    the operations in groups $D$, $E$, and $F$ in any order and any number of times 
    is equivalent to the set obtained by first applying one of the $13$ operations in group $D$, 
    then applying one of the following seven combinations of operations, 
    and finally applying one of the $13$ operations in group $D$ again:
    \begin{itemize}
        \item $\mathrm{I}$
        \item $\mathrm{E1}$
        \item $\mathrm{F1}$
        \item $\mathrm{E1}\mathrm{F1}$
        \item $\mathrm{E1}\mathrm{D3}\mathrm{D2}\mathrm{F1}$
        \item $\mathrm{F1}\mathrm{D2}\mathrm{D3}\mathrm{E1}$
        \item $\mathrm{F1}\mathrm{D2}\mathrm{D3}\mathrm{D1}\mathrm{E1}\mathrm{D3}\mathrm{D2}\mathrm{F1}$
    \end{itemize}
\end{thm}

\begin{proof}
    By Lemma~\ref{idempotent}, any sequence of operations obtained 
    by combining those from groups $D$, $E$, and $F$ in arbitrary order and frequency 
    after applying an operation from group $D$ can be represented as
    \begin{equation}
        D\, EF\, D\, EF\, D \cdots D\, EF\, D\, EF\, D.
    \end{equation}
    From Theorem~\ref{reduction_13}, we have 
    $R(DEFDEFDEFDEFD) = R(DEFDEFDEFD)$,
    which implies that it suffices to consider the results obtained by
    \begin{equation}
        D\, EF\, D\, EF\, D\, EF\, D.
    \end{equation}
    Furthermore, by Theorem~\ref{reduction_10}, 
    all combinations except for $D\mathrm{F1}\mathrm{D2}\mathrm{D3}\mathrm{D1}\mathrm{E1}\mathrm{D3}\mathrm{D2}\mathrm{F1}D$ yield results already included in
    \begin{equation}
        D\, EF\, D\, EF\, D.
    \end{equation}
    Similarly, by Theorem~\ref{reduction_7}, 
    all combinations except for $D\mathrm{E1}\mathrm{D3}\mathrm{D2}\mathrm{F1}D$ and $D\mathrm{F1}\mathrm{D2}\mathrm{D3}\mathrm{E1}D$ 
    yield results already included in
    \begin{equation}
        D\, EF\, D.
    \end{equation}
    Therefore, it suffices to consider the following cases:
    \begin{itemize}
        \item $D\, EF\, D$
        \item $D\, \mathrm{E1}\mathrm{D3}\mathrm{D2}\mathrm{F1}\, D$
        \item $D\, \mathrm{F1}\mathrm{D2}\mathrm{D3}\mathrm{E1}\, D$
        \item $D\, \mathrm{F1}\mathrm{D2}\mathrm{D3}\mathrm{D1}\mathrm{E1}\mathrm{D3}\mathrm{D2}\mathrm{F1}\, D$
    \end{itemize}
    Hence, the theorem follows.
\end{proof}

\subsection{All Combinations}

Therefore, combining Groups A, B, C and the first application of Group D operations, the cases to consider are:

\begin{itemize}
    \item Group A: 1 case
    \item Group B: 10 cases
    \item Group C: 2 cases
    \item First Group D: 13 cases
    \item Groups E, F, and second Group D: 7 cases
    \item Third Group D: 13 cases
\end{itemize}

Since these can be considered independent, the total number of possible outcome patterns is 23660.

%\section{Proof in an Expanded Form Using Lyapunov Functions for Each Theorem}
\section{Alternative Proofs of Theorems by Lyapunov Functions for Confirmation}

\subsection{Theorem \ref{Dumped_Newton}}
\label{sec:pr_thm_d_Newton}

In Theorem \ref{Dumped_Newton}, for the continuous dynamical system
\begin{equation}
    \nabla^2 f \dot{x} + \nabla f = 0
    \label{ODE_Dumped_Newton}
\end{equation}
the convergence rate
\begin{equation}
    f(x(t)) - f_* = \mathrm{O}\left( \mathrm{e}^{-t} \right)
\end{equation}
is guaranteed for a convex objective function $f$. Here, we write down the Lyapunov function used in the proof and differentiate it explicitly to support the claim of the theorem.

First, for a constant $k$, define the function
\begin{equation}
    \mathcal{E}(t) = k \mathrm{e}^{kt} \left( f_* - f - \langle \nabla f, x_* - x \rangle \right) + \mathrm{e}^{kt} \left( f - f_* \right)
\end{equation}
following matrix $P^{DN}$ given in equation \eqref{Lyapunov_Dumped_Newton}. It follows from
\begin{equation}
    \mathcal{E}(t) - \mathrm{e}^{kt} \left( f - f_* \right) = k \mathrm{e}^{kt} \left( f_* - f - \langle \nabla f, x_* - x \rangle \right)
\end{equation}
that
\begin{equation}
    \mathcal{E}(t) \geq \mathrm{e}^{kt} \left( f - f_* \right)
\end{equation}
because $f$ is convex and $k \geq 0$.
%holds. 
Considering the time derivative of this function, we have
\begin{equation}
    \begin{split}
        \dot{\mathcal{E}}(t) &= k^2 \mathrm{e}^{kt} \left( f_* - f - \langle \nabla f, x_* - x \rangle \right) + k \mathrm{e}^{kt} \left( - \langle \nabla f, \dot{x} \rangle - \langle \nabla^2 f \dot{x}, x_* - x \rangle - \langle \nabla f, -\dot{x} \rangle \right) \\
        &\phantom{=} + k \mathrm{e}^{kt} \left( f - f_* \right) + \mathrm{e}^{kt} \langle \nabla f, \dot{x} \rangle \\
        &= k^2 \mathrm{e}^{kt} \left( f_* - f - \langle \nabla f, x_* - x \rangle \right) - k \mathrm{e}^{kt} \langle \nabla^2 f \dot{x}, x_* - x \rangle + k \mathrm{e}^{kt} \left( f - f_* \right) + \mathrm{e}^{kt} \langle \nabla f, \dot{x} \rangle.
    \end{split}
\end{equation}
Here, from the continuous dynamical system in equation \eqref{ODE_Dumped_Newton}, we have
\begin{gather}
    - k \mathrm{e}^{kt} \langle \nabla^2 f \dot{x}, x_* - x \rangle = k \mathrm{e}^{kt} \langle \nabla f, x_* - x \rangle, \\
    \mathrm{e}^{kt} \langle \nabla f, \dot{x} \rangle = - \mathrm{e}^{kt} \langle \nabla^2 f \dot{x}, \dot{x} \rangle.
\end{gather}
Therefore we have
\begin{equation}
    \begin{split}
        \dot{\mathcal{E}}(t) &= k^2 \mathrm{e}^{kt} \left( f_* - f - \langle \nabla f, x_* - x \rangle \right) + k \mathrm{e}^{kt} \langle \nabla f, x_* - x \rangle + k \mathrm{e}^{kt} \left( f - f_* \right) - \mathrm{e}^{kt} \langle \nabla^2 f \dot{x}, \dot{x} \rangle \\
        &= k^2 \mathrm{e}^{kt} \left( f_* - f - \langle \nabla f, x_* - x \rangle \right) - k \mathrm{e}^{kt} \left( f_* - f - \langle \nabla f, x_* - x \rangle \right) - \mathrm{e}^{kt} \langle \nabla^2 f \dot{x}, \dot{x} \rangle \\
        &= (k^2 - k) \mathrm{e}^{kt} \left( f_* - f - \langle \nabla f, x_* - x \rangle \right) - \mathrm{e}^{kt} \langle \nabla^2 f \dot{x}, \dot{x} \rangle.
    \end{split}
\end{equation}
Thus, since $f$ is convex,
\begin{equation}
    \dot{\mathcal{E}}(t) \leq 0
\end{equation}
holds for $0 \leq k \leq 1$. 
Therefore $\mathcal{E}(t)$ is indeed a Lyapunov function. Setting $k=1$, we obtain
\begin{equation}
    \mathrm{e}^t (f-f_*) \leq \mathcal{E}(t) \leq \mathcal{E} (0) = \mathrm{const.},
\end{equation}
which gives
\begin{equation}
    f(x(t)) - f_* = \mathrm{O}\left( \mathrm{e}^{-t} \right)
\end{equation}
as claimed.

\subsection{Theorem \ref{First_Order}}
\label{sec:pr_thm_1st_order}

In Theorem \ref{First_Order}, for the continuous dynamical system
\begin{equation}
    \dot{x} - \frac{1}{L} \nabla^2 f \dot{x} + \nabla f = 0,
    \label{ODE_First_Order}
\end{equation}
the convergence rate
\begin{equation}
    f(x(t)) - f_* = \mathrm{O}\left( \mathrm{e}^{- \frac{\mu}{1-\frac{\mu}{L}} t} \right)
\end{equation}
is guaranteed for a $\mu$-strongly convex and $L$-smooth objective function $f$ ($0<\mu<L$). Here, we write down the Lyapunov function used in the proof and differentiate it explicitly to support the claim of the theorem.

First, for a constant $k$, define the function
\begin{equation}
    \mathcal{E}(t) = \frac{k}{2} \mathrm{e}^{kt} \|x-x_*\|^2 - \frac{k}{L} \mathrm{e}^{kt} \left( f_* - f - \langle \nabla f, x_* - x \rangle \right) + \mathrm{e}^{kt} \left( f - f_* \right)
\end{equation}
following matrix $P^{FO}$ given in equation \eqref{Lyapunov_First_Order}. 
If follows from
\begin{equation}
    \begin{split}
        \mathcal{E}(t) - \mathrm{e}^{kt} \left( f - f_* \right) &= \frac{k}{2} \mathrm{e}^{kt} \|x-x_*\|^2 -  \frac{k}{L} \mathrm{e}^{kt} \left( f_* - f - \langle \nabla f, x_* - x \rangle \right)\\
        &= \frac{k}{L} \left( \frac{L}{2} \|x-x_*\|^2 - \left( f_* - f - \langle \nabla f, x_* - x \rangle \right) \right)
    \end{split}
\end{equation}
that 
\begin{equation}
    \mathcal{E}(t) \geq \mathrm{e}^{kt} \left( f - f_* \right)
\end{equation}
because $f$ is $L$-smooth and $k\geq0$.

Considering the time derivative of this function, we have
\begin{equation}
    \begin{split}
        \dot{\mathcal{E}}(t) &= \frac{k^2}{2} \mathrm{e}^{kt} \|x-x_*\|^2 + k \mathrm{e}^{kt} \langle x-x_*, \dot{x} \rangle - \frac{k^2}{L} \mathrm{e}^{kt} \left( f_* - f - \langle \nabla f, x_* - x \rangle \right) \\
        &\phantom{=} - \frac{k}{L} \mathrm{e}^{kt} \left( - \langle \nabla f, \dot{x} \rangle - \langle \nabla^2 f \dot{x}, x_* - x\rangle - \langle \nabla f, -\dot{x} \rangle \right) + k \mathrm{e}^{kt} \left( f - f_* \right) + \mathrm{e}^{kt} \langle \nabla f, \dot{x} \rangle \\
        &= \frac{k^2}{2} \mathrm{e}^{kt} \|x-x_*\|^2 + k \mathrm{e}^{kt} \langle x-x_*, \dot{x} \rangle - \frac{k^2}{L} \mathrm{e}^{kt} \left( f_* - f - \langle \nabla f, x_* - x \rangle \right) \\
        &\phantom{=} + \frac{k}{L} \mathrm{e}^{kt} \langle \nabla^2 f \dot{x}, x_* - x \rangle + k \mathrm{e}^{kt} \left( f - f_* \right) + \mathrm{e}^{kt} \langle \nabla f, \dot{x} \rangle.
    \end{split}
\end{equation}
Here, from the continuous dynamical system in equation \eqref{ODE_First_Order}, we have
\begin{gather}
    \frac{k}{L} \mathrm{e}^{kt} \langle \nabla^2 f \dot{x}, x_* - x \rangle = k \mathrm{e}^{kt} \langle \dot{x} , x_* - x \rangle + k \mathrm{e}^{kt} \langle \nabla f , x_* - x \rangle, \\
    \mathrm{e}^{kt} \langle \nabla f, \dot{x} \rangle = - \mathrm{e}^{kt} \|\dot{x}\|^2 + \frac{1}{L} \mathrm{e}^{kt} \langle \nabla^2 f \dot{x}, \dot{x} \rangle.
\end{gather}

Substituting these into the derivative gives
\begin{equation}
    \begin{split}
        \dot{\mathcal{E}}(t) &= \frac{k^2}{2} e^{kt} \|x-x_*\|^2 + k e^{kt} \langle x-x_*, \dot{x} \rangle - \frac{k^2}{L} e^{kt} \left( f_* - f - \langle \nabla f, x_* - x \rangle \right) \\
        &\phantom{=} + k e^{kt} \langle \dot{x} , x_* - x \rangle + k e^{kt} \langle \nabla f , x_* - x \rangle + k e^{kt} \left( f - f_* \right) - e^{kt} \|\dot{x}\|^2 + \frac{1}{L} e^{kt} \langle \nabla^2 f \dot{x}, \dot{x} \rangle \\
        &= \frac{k^2}{2} e^{kt} \|x-x_*\|^2 - \frac{k^2}{L} e^{kt} \left( f_* - f - \langle \nabla f, x_* - x \rangle \right) \\
        &\phantom{=} - k e^{kt} \left( f_* - f - \langle \nabla f, x_* - x \rangle \right) -e^{kt} \|\dot{x}\|^2 + \frac{1}{L} e^{kt} \langle \nabla^2 f \dot{x}, \dot{x} \rangle \\
        &= \frac{k^2}{2} e^{kt} \|x-x_*\|^2 - \left( \frac{k^2}{L} + k \right) e^{kt} \left( f_* - f - \langle \nabla f, x_* - x \rangle \right) \\
        &\phantom{=} - \left( e^{kt} \|\dot{x}\|^2 - \frac{1}{L} e^{kt} \langle \nabla^2 f \dot{x}, \dot{x} \rangle \right)\\
        &= \left(\frac{k^2}{2} - \frac{\mu}{2} \left( \frac{k^2}{L} + k \right) \right)e^{kt} \|x-x_*\|^2 \\
        &\phantom{=} - \left( \frac{k^2}{L} + k \right) e^{kt} \left( f_* - f - \langle \nabla f, x_* - x \rangle - \frac{\mu}{2} \|x-x_*\|^2 \right) \\
        &\phantom{=} - \frac{1}{L} \left( L e^{kt} \|\dot{x}\|^2 - e^{kt} \langle \nabla^2 f \dot{x}, \dot{x} \rangle \right)\\
        &= \left( 1 - \frac{\mu}{L} \right) \frac{k}{2}\left(k  - \frac{\mu}{1 - \frac{\mu}{L}}\right)e^{kt} \|x-x_*\|^2 \\
        &\phantom{=} - \frac{k}{L} \left( k + L \right) e^{kt} \left( f_* - f - \langle \nabla f, x_* - x \rangle - \frac{\mu}{2} \|x-x_*\|^2 \right) \\
        &\phantom{=} - \frac{1}{L} \left( L e^{kt} \|\dot{x}\|^2 - e^{kt} \langle \nabla^2 f \dot{x}, \dot{x} \rangle \right).\\
    \end{split}
\end{equation}
Therefore, since $f$ is $\mu$-strongly convex and $L$-smooth, we have 
\begin{equation}
    \dot{\mathcal{E}}(t) \leq 0
\end{equation}
for $0\leq k \leq \frac{\mu}{1 - \frac{\mu}{L}}$.
Hence $\mathcal{E}(t)$ is indeed a Lyapunov function. Choosing $k=\frac{\mu}{1 - \frac{\mu}{L}}$ yields
\begin{equation}
    \mathrm{e}^{\frac{\mu}{1 - \frac{\mu}{L}} t} (f-f_*) \leq \mathcal{E}(t) \leq \mathcal{E} (0) = \mathrm{const.},
\end{equation}
which implies
\begin{equation}
    f(x(t)) - f_* = \mathrm{O}\left( \mathrm{e}^{- \frac{\mu}{1 - \frac{\mu}{L}}t} \right).
\end{equation}

\subsection{Theorem \ref{Gradient_Descent}}
\label{sec:pr_thm_GD}

In Theorem \ref{Gradient_Descent}, for the continuous dynamical system
\begin{equation}
    \dot{x} + \nabla f = 0,
    \label{ODE_Gradient_Descent}
\end{equation}
the convergence rate
\begin{equation}
    f(x(t)) - f_* = \mathrm{O}\left( \mathrm{e}^{-2\mu t} \right)
\end{equation}
is guaranteed for a $\mu$-strongly convex objective function $f$. Here, we write down the Lyapunov function used in the proof and differentiate it explicitly to support the claim.

For a constant $k$, we define
\begin{equation}
    \mathcal{E}(t) =  \mathrm{e}^{kt} \left( f - f_* \right)
\end{equation}
following matrix $P^{GD}$ given in equation \eqref{Lyapunov_Gradient_Descent}. Its derivative is given by
\begin{equation}
    \begin{split}
        \dot{\mathcal{E}}(t) &= k \mathrm{e}^{kt} \left( f - f_* \right) + \mathrm{e}^{kt} \langle \nabla f, \dot{x} \rangle \\
        &= -k \mathrm{e}^{kt} \left( f_* - f - \langle \nabla f, x_* - x \rangle - \frac{\mu}{2} \|x-x_*\|^2 \right) + k \mathrm{e}^{kt} \langle \nabla f, x - x_* \rangle - \frac{\mu k}{2} \mathrm{e}^{kt} \|x-x_*\|^2 + \mathrm{e}^{kt} \langle \nabla f, \dot{x} \rangle.
    \end{split}
\end{equation}
Here, from the continuous dynamical system in equation \eqref{ODE_Gradient_Descent}, we have
\begin{gather}
    \mathrm{e}^{kt} \langle \nabla f, \dot{x} \rangle = - \mathrm{e}^{kt} \| \dot{x} \|^2, \\
    k \mathrm{e}^{kt} \langle \nabla f, x - x_* \rangle = - k \mathrm{e}^{kt} \langle \dot{x} , x - x_* \rangle.
\end{gather}
From these, we get
\begin{equation}
    \begin{split}
        \dot{\mathcal{E}}(t) &= -k e^{kt} \left( f_* - f - \langle \nabla f, x_* - x \rangle - \frac{\mu}{2} \|x-x_*\|^2 \right) - k e^{kt} \langle \dot{x} , x - x_* \rangle - \frac{\mu k}{2} e^{kt} \|x-x_*\|^2 - e^{kt} \| \dot{x} \|^2 \\
        &= -k e^{kt} \left( f_* - f - \langle \nabla f, x_* - x \rangle - \frac{\mu}{2} \|x-x_*\|^2 \right) - e^{kt} \left\| \dot{x} - \frac{k}{2}(x-x_*)\right\|^2 + \left( \frac{k^2}{4} - \frac{\mu k}{2} \right) e^{kt} \|x-x_*\|^2 \\
        &= -k e^{kt} \left( f_* - f - \langle \nabla f, x_* - x \rangle - \frac{\mu}{2} \|x-x_*\|^2 \right) - e^{kt} \left\| \dot{x} - \frac{k}{2}(x-x_*)\right\|^2 + \frac{k}{4} \left( k - 2\mu \right) e^{kt} \|x-x_*\|^2.
    \end{split}
\end{equation}

Therefore, since $f$ is $\mu$-strongly convex, we have
\begin{equation}
    \dot{\mathcal{E}}(t) \leq 0
\end{equation}
for $0 \leq k \leq 2\mu$. 
Hence $\mathcal{E}(t)$ is indeed a Lyapunov function. Choosing $k = 2\mu$, we have
\begin{equation}
    \mathrm{e}^{2\mu t} (f-f_*) = \mathcal{E}(t) \leq \mathcal{E}(0) = \mathrm{const.},
\end{equation}
which implies
\begin{equation}
    f(x(t)) - f_* = \mathrm{O}\left( \mathrm{e}^{-2\mu t} \right).
\end{equation}

\subsection{Theorem \ref{SC-NAG}}
\label{sec:pr_thm_SC-NAG}

In Theorem \ref{SC-NAG}, for the continuous dynamical system
\begin{equation}
    \ddot{x} + 2\sqrt{\mu} \dot{x} + \nabla f = 0,
    \label{ODE_SC-NAG}
\end{equation}
the convergence rate
\begin{equation}
    f(x(t)) - f_* = \mathrm{O}\left( \mathrm{e}^{-\sqrt{\mu}t} \right)
\end{equation}
is guaranteed for a $\mu$-strongly convex objective function $f$. Here, we write down the Lyapunov function used in the proof and differentiate it explicitly to support the claim of the theorem.

For a constant $k$, we define
\begin{equation}
    \mathcal{E}(t) = \left( \sqrt{\mu} k - \frac{k^2}{2} \right) \mathrm{e}^{kt} \|x-x_*\|^2 + k \mathrm{e}^{kt} \langle x-x_*, \dot{x} \rangle + \frac{1}{2} \mathrm{e}^{kt} \|\dot{x}\|^2 + \mathrm{e}^{kt} \left( f - f_* \right)
\end{equation}
following matrix $P^{SCNAG}$ given in equation \eqref{Lyapunov_SC-NAG}. Then, the expression
\begin{equation}
    \begin{split}
        \mathcal{E}(t) - e^{kt} \left( f - f_* \right) & = \left( \sqrt{\mu} k - \frac{k^2}{2} \right) e^{kt} \|x-x_*\|^2 + k e^{kt} \langle x-x_*, \dot{x} \rangle + \frac{1}{2} e^{kt} \|\dot{x}\|^2 \\
        &= \frac{1}{2} e^{kt}\|\dot{x} + k (x-x_*)\|^2 + (\sqrt{\mu}k - k^2) e^{kt} \|x-x_*\|^2 \\
        &= \frac{1}{2} e^{kt}\|\dot{x} + k (x-x_*)\|^2 + k (\sqrt{\mu} - k) e^{kt} \|x-x_*\|^2 \\
    \end{split}
\end{equation}
guarantees
\begin{equation}
    \mathcal{E}(t) \geq \mathrm{e}^{kt} \left( f - f_* \right)
\end{equation}
for $0 \le k \le \sqrt{\mu}$.

The time derivative of this function is given by 
\begin{equation}
    \begin{split}
        \dot{\mathcal{E}}(t) &= \left( \sqrt{\mu} k^2 - \frac{k^3}{2} \right) \mathrm{e}^{kt} \|x-x_*\|^2 + \left( 2\sqrt{\mu}k - k^2 \right) \mathrm{e}^{kt} \langle x-x_*, \dot{x} \rangle + k^2 \mathrm{e}^{kt} \langle x-x_*, \dot{x} \rangle \\
        &\phantom{=} + k \mathrm{e}^{kt} \langle x-x_*, \ddot{x} \rangle + k \mathrm{e}^{kt} \|\dot{x}\|^2 + \frac{k}{2} \mathrm{e}^{kt} \|\dot{x}\|^2 + \mathrm{e}^{kt} \langle \dot{x}, \ddot{x} \rangle + k \mathrm{e}^{kt} \left( f-f_* \right) + \mathrm{e}^{kt} \langle \nabla f, \dot{x} \rangle \\
        &= \left( \sqrt{\mu} k^2 - \frac{k^3}{2} \right) \mathrm{e}^{kt} \|x-x_*\|^2 + 2\sqrt{\mu}k \mathrm{e}^{kt} \langle x-x_*, \dot{x} \rangle \\
        &\phantom{=} + k \mathrm{e}^{kt} \langle x-x_*, \ddot{x} \rangle + \frac{3}{2} k \mathrm{e}^{kt} \|\dot{x}\|^2 + \mathrm{e}^{kt} \langle \dot{x}, \ddot{x} \rangle + k \mathrm{e}^{kt} \left( f-f_* \right) + \mathrm{e}^{kt} \langle \nabla f, \dot{x} \rangle.
    \end{split}
\end{equation}
Here, from the continuous dynamical system in equation \eqref{ODE_SC-NAG}, we have
\begin{gather}
    k \mathrm{e}^{kt} \langle x-x_*, \ddot{x} \rangle = - 2\sqrt{\mu} k \mathrm{e}^{kt} \langle x-x_*, \dot{x} \rangle - k \mathrm{e}^{kt} \langle \nabla f, x - x_* \rangle, \\
    \mathrm{e}^{kt} \langle \dot{x}, \ddot{x} \rangle = - 2\sqrt{\mu} \mathrm{e}^{kt} \|\dot{x}\|^2 - \mathrm{e}^{kt} \langle \nabla f, \dot{x} \rangle.
\end{gather}
From these, we have
\begin{equation}
    \begin{split}
        \dot{\mathcal{E}}(t) &= \left( \sqrt{\mu} k^2 - \frac{k^3}{2} \right) e^{kt} \|x-x_*\|^2 + 2\sqrt{\mu}k e^{kt} \langle x-x_*, \dot{x} \rangle - 2\sqrt{\mu} k e^{kt} \langle x-x_*, \dot{x} \rangle\\
        &\phantom{=} - k e^{kt} \langle \nabla f, x - x_* \rangle + \frac{3}{2} k e^{kt} \|\dot{x}\|^2 - 2\sqrt{\mu} e^{kt} \|\dot{x}\|^2 - e^{kt} \langle \nabla f, \dot{x} \rangle + k e^{kt} \left( f-f_* \right) + e^{kt} \langle \nabla f, \dot{x} \rangle \\
        &= \left( \sqrt{\mu} k^2 - \frac{k^3}{2} \right) e^{kt} \|x-x_*\|^2 - k e^{kt} \langle \nabla f, x - x_* \rangle + \left( \frac{3}{2} k - 2\sqrt{\mu} \right) e^{kt} \|\dot{x}\|^2 + k e^{kt} \left( f-f_* \right) \\
        &= \left( \sqrt{\mu} k^2 - \frac{k^3}{2} -\frac{\mu}{2} k\right) e^{kt} \|x-x_*\|^2 - k e^{kt} \left( f_* - f - \langle \nabla f, x_* - x \rangle - \frac{\mu}{2} \|x-x_*\|^2 \right) \\
        &\phantom{=} + \frac{3}{2} \left( k - \sqrt{\mu} \right) e^{kt} \|\dot{x}\|^2 - \frac{\sqrt{\mu}}{2} e^{kt} \|\dot{x}\|^2 \\
        &= - \frac{k}{2} (k - \sqrt{\mu})^2 e^{kt} \|x-x_*\|^2 - k e^{kt} \left( f_* - f - \langle \nabla f, x_* - x \rangle - \frac{\mu}{2} \|x-x_*\|^2 \right) \\
        &\phantom{=} + \frac{3}{2} \left( k - \sqrt{\mu} \right) e^{kt} \|\dot{x}\|^2 - \frac{\sqrt{\mu}}{2} e^{kt} \|\dot{x}\|^2.
    \end{split}
\end{equation}

Since the objective function is $\mu$-strongly convex, we have
\begin{equation}
    \dot{\mathcal{E}}(t) \leq 0,
\end{equation}
for $0\leq k \leq \sqrt{\mu}$.
Hence $\mathcal{E}(t)$ is a Lyapunov function. Choosing $k = \sqrt{\mu}$, we have
\begin{equation}
    \mathrm{e}^{\sqrt{\mu}t} (f-f_*) \leq \mathcal{E}(t) \leq \mathcal{E} (0) = \mathrm{const.},
\end{equation}
which implies
\begin{equation}
    f(x(t)) - f_* = \mathrm{O}\left( \mathrm{e}^{-\sqrt{\mu}t} \right).
\end{equation}

\subsection{Theorem \ref{Second_Order}}
\label{sec:pr_thm_Second_Order}

In Theorem \ref{Second_Order}, for the continuous dynamical system
\begin{equation}
    \ddot{x} + \sqrt{\mu} \dot{x} + \frac{1}{\sqrt{\mu}} \nabla^2 f \dot{x} + \nabla f = 0,
    \label{ODE_Second_Order}
\end{equation}
the convergence rate
\begin{equation}
    f(x(t)) - f_* = \mathrm{O}\left( \mathrm{e}^{-\sqrt{\mu}t} \right)
\end{equation}
is guaranteed for a $\mu$-strongly convex objective function $f$. Here, we write down the Lyapunov function used in the proof and differentiate it explicitly to support the claim of the theorem.

For a constant $k$, we define
\begin{equation}
    \mathcal{E}(t) = \frac{k}{\sqrt{\mu}} \mathrm{e}^{kt} (f_* - f - \langle \nabla f, x_* - x \rangle) + k \mathrm{e}^{kt} \langle x-x_*, \dot{x} \rangle + \frac{1}{2} \mathrm{e}^{kt} \|\dot{x}\|^2 + \mathrm{e}^{kt} \left( f - f_* \right)
\end{equation}
following matrix $P^{SO}$ given in equation \eqref{Lyapunov_Second_Order}. Then, we have
\begin{equation}
    \begin{split}
        \mathcal{E}(t) - \mathrm{e}^{kt} \left( f - f_* \right) 
        &= \frac{k}{\sqrt{\mu}} \mathrm{e}^{kt} \left(f_* - f - \langle \nabla f, x_* - x \rangle \right) + k \mathrm{e}^{kt} \langle x-x_*, \dot{x} \rangle + \frac{1}{2} \mathrm{e}^{kt} \|\dot{x}\|^2 \\
        &= \frac{k}{\sqrt{\mu}} \mathrm{e}^{kt} \left(f_* - f - \langle \nabla f, x_* - x \rangle - \frac{\mu}{2} \|x-x_*\|^2 \right) + \frac{\sqrt{\mu}}{2} k \mathrm{e}^{kt} \|x-x_*\|^2 \\
        &\phantom{=} + \frac{1}{2} \mathrm{e}^{kt} \| \dot{x} + k(x-x_*) \|^2 - \frac{k^2}{2} \mathrm{e}^{kt} \|x-x_*\|^2 \\
        &= \frac{k}{\sqrt{\mu}} \mathrm{e}^{kt} \left(f_* - f - \langle \nabla f, x_* - x \rangle - \frac{\mu}{2} \|x-x_*\|^2 \right) \\
        &\phantom{=} + \frac{1}{2} \mathrm{e}^{kt} \| \dot{x} + k(x-x_*) \|^2 +\frac{k}{2} (\sqrt{\mu} - k) \mathrm{e}^{kt} \|x-x_*\|^2.
    \end{split}
\end{equation}
Thus, for $0 \le k \le \sqrt{\mu}$, we have
\begin{equation}
    \mathcal{E}(t) \geq \mathrm{e}^{kt} \left( f - f_* \right).
\end{equation}

Next, we consider the time derivative of this function:
\begin{equation}
    \begin{split}
        \dot{\mathcal{E}}(t) &= \frac{k^2}{\sqrt{\mu}} \mathrm{e}^{kt} (f_* - f - \langle \nabla f, x_* - x \rangle) - \frac{k}{\sqrt{\mu}} \mathrm{e}^{kt} \langle \nabla f, \dot{x} \rangle - \frac{k}{\sqrt{\mu}} \mathrm{e}^{kt} \langle \nabla^2 f \dot{x}, x_* - x \rangle - \frac{k}{\sqrt{\mu}} \mathrm{e}^{kt} \langle \nabla f, -\dot{x} \rangle \\
        &\phantom{=} + k^2 \mathrm{e}^{kt} \langle x-x_*, \dot{x} \rangle + k \mathrm{e}^{kt} \|\dot{x}\|^2 + k \mathrm{e}^{kt} \langle x-x_*, \ddot{x} \rangle + \frac{k}{2} \mathrm{e}^{kt} \|\dot{x}\|^2 + \mathrm{e}^{kt} \langle \dot{x}, \ddot{x} \rangle + k \mathrm{e}^{kt} (f - f_*) + \mathrm{e}^{kt} \langle \nabla f, \dot{x} \rangle \\
        &= \frac{k^2}{\sqrt{\mu}} \mathrm{e}^{kt} (f_* - f - \langle \nabla f, x_* - x \rangle) + \frac{k}{\sqrt{\mu}} \mathrm{e}^{kt} \langle \nabla^2 f \dot{x}, x - x_* \rangle \\
        &\phantom{=} + k^2 \mathrm{e}^{kt} \langle x-x_*, \dot{x} \rangle + \frac{3}{2}k \mathrm{e}^{kt} \|\dot{x}\|^2 + k \mathrm{e}^{kt} \langle x-x_*, \ddot{x} \rangle + \mathrm{e}^{kt} \langle \dot{x}, \ddot{x} \rangle + k \mathrm{e}^{kt} (f - f_*) + \mathrm{e}^{kt} \langle \nabla f, \dot{x} \rangle.
    \end{split}
\end{equation}
Here, from the continuous dynamical system in equation \eqref{ODE_Second_Order}, we have
\begin{gather}
    k \mathrm{e}^{kt} \langle x-x_*, \ddot{x} \rangle =  -\sqrt{\mu} k \mathrm{e}^{kt} \langle x-x_*,\dot{x} \rangle - \frac{k}{\sqrt{\mu}} \mathrm{e}^{kt} \langle x-x_*, \nabla^2 f \dot{x} \rangle - k \mathrm{e}^{kt} \langle x-x_*, \nabla f \rangle, \\
    \mathrm{e}^{kt} \langle \dot{x}, \ddot{x} \rangle = - \sqrt{\mu} \mathrm{e}^{kt} \|\dot{x}\|^2 - \frac{1}{\sqrt{\mu}} \mathrm{e}^{kt} \langle \dot{x}, \nabla^2 f \dot{x} \rangle - \mathrm{e}^{kt} \langle \dot{x}, \nabla f \rangle.
\end{gather}
From these, we have
\begin{equation}
    \begin{split}
        \dot{\mathcal{E}}(t) &= \frac{k^2}{\sqrt{\mu}} e^{kt} (f_* - f - \langle \nabla f, x_* - x \rangle) + \frac{k}{\sqrt{\mu}} e^{kt} \langle \nabla^2 f \dot{x}, x - x_* \rangle \\
        &\phantom{=} + k^2 e^{kt} \langle x-x_*, \dot{x} \rangle + \frac{3}{2}k e^{kt} \|\dot{x}\|^2 -\sqrt{\mu} k e^{kt} \langle x-x_*,\dot{x} \rangle - \frac{k}{\sqrt{\mu}} e^{kt} \langle x-x_*, \nabla^2 f \dot{x} \rangle - k e^{kt} \langle x-x_*, \nabla f \rangle \\
        &\phantom{=} - \sqrt{\mu} e^{kt} \|\dot{x}\|^2 - \frac{1}{\sqrt{\mu}} e^{kt} \langle \dot{x}, \nabla^2 f \dot{x} \rangle - e^{kt} \langle \dot{x}, \nabla f \rangle + k e^{kt} (f - f_*) + e^{kt} \langle \nabla f, \dot{x} \rangle \\
        &= \left( \frac{k^2}{\sqrt{\mu}} -k \right) e^{kt} (f_* - f - \langle \nabla f, x_* - x \rangle) + \left( k^2 - \sqrt{\mu} k \right) e^{kt} \langle x - x_*, \dot{x} \rangle \\
        &\phantom{=} + \left( \frac{3}{2} k - \sqrt{\mu} \right) e^{kt} \|\dot{x}\|^2 - \frac{1}{\sqrt{\mu}} e^{kt} \langle \nabla^2 f \dot{x}, \dot{x} \rangle \\
        &= \left( \frac{k^2}{\sqrt{\mu}} - k \right) e^{kt} \left( f_* - f - \langle \nabla f, x_* - x \rangle - \frac{\mu}{2} \| x - x_* \|^2 \right) + \left( \frac{\sqrt{\mu}}{2} k^2 - \frac{\mu}{2} k \right) e^{kt} \| x - x_* \|^2 \\
        &\phantom{=} + (k^2 - \sqrt{\mu}k) e^{kt} \langle x - x_*, \dot{x} \rangle - \frac{k}{2} e^{kt} \|\dot{x}\|^2 + 2(k - \sqrt{\mu}) e^{kt} \|\dot{x}\|^2 + \frac{1}{\sqrt{\mu}} e^{kt} (\mu \|\dot{x}\|^2 - \langle \nabla^2 f \dot{x}, \dot{x} \rangle) \\
        &= \frac{k}{\sqrt{\mu}} \left( k - \sqrt{\mu} \right) e^{kt} \left( f_* - f - \langle \nabla f, x_* - x \rangle - \frac{\mu}{2} \| x - x_* \|^2 \right) + \frac{\sqrt{\mu}}{2} k (k - \sqrt{\mu}) e^{kt} \left\| x - x_* + \frac{1}{\sqrt{\mu}} \dot{x} \right\|^2 \\
        &\phantom{=} - \frac{k^2}{2\sqrt{\mu}} e^{kt} \|\dot{x}\|^2 + 2(k - \sqrt{\mu}) e^{kt} \|\dot{x}\|^2 + \frac{1}{\sqrt{\mu}} e^{kt} (\mu \|\dot{x}\|^2 - \langle \nabla^2 f \dot{x}, \dot{x} \rangle).
    \end{split}
\end{equation}

Then, since the objective function is $\mu$-strongly convex, we have
\begin{equation}
    \dot{\mathcal{E}}(t) \leq 0.
\end{equation}
for $0\leq k \leq \sqrt{\mu}$.
Hence $\mathcal{E}(t)$ is a Lyapunov function. Choosing $k=\sqrt{\mu}$, we have
\begin{equation}
    \mathrm{e}^{\sqrt{\mu}t} (f-f_*) \leq \mathcal{E}(t) \leq \mathcal{E} (0) = \mathrm{const.},
\end{equation}
which implies
\begin{equation}
    f(x(t)) - f_* = \mathrm{O}\left( \mathrm{e}^{-\sqrt{\mu}t} \right).
\end{equation}

\subsection{Theorem \ref{NAG_Convex}}
\label{sec:pr_thm_NAG_Convex}

In Theorem \ref{NAG_Convex}, for the continuous dynamical system
\begin{equation}
    \ddot{x} + \frac{3}{t} \dot{x} + \nabla f = 0,
    \label{ODE_NAG_Convex}
\end{equation}
the convergence rate
\begin{equation}
    f(x(t)) - f_* = \mathrm{O}\left( \frac{1}{t^2} \right)
\end{equation}
is guaranteed for a convex objective function $f$. Here, we write down the Lyapunov function used in the proof and differentiate it explicitly to support the claim of the theorem.

For a constant $k$, we define
\begin{equation}
    \mathcal{E}(t) = t^k \biggl( \frac{4k - k^2}{2t^2} \|x-x_*\|^2 + \frac{k}{t} \langle x - x_*, \dot{x} \rangle + \frac{1}{2} \|\dot{x}\|^2 \biggr) + t^k (f-f_*)
\end{equation}
following matrix $P^{NAG}$ given in equation \eqref{Lyapunov_NAG_Convex}. Then, we have
\begin{equation}
    \begin{split}
        \mathcal{E}(t) - t^k \left( f - f_* \right) 
        &= t^k \biggl( \frac{4k - k^2}{2t^2} \|x-x_*\|^2 + \frac{k}{t} \langle x - x_*, \dot{x} \rangle + \frac{1}{2} \|\dot{x}\|^2 \biggr) \\
        &= t^k \biggl( \frac{4k - k^2}{2t^2} \|x-x_*\|^2 + \frac{1}{2} \left\|\dot{x} + \frac{k}{t}(x-x_*)\right\|^2 - \frac{k^2}{2t^2} \|x-x_*\|^2 \biggr) \\
        &= t^k \biggl( \frac{k(2-k)}{t^2} \|x-x_*\|^2 + \frac{1}{2} \left\|\dot{x} + \frac{k}{t}(x-x_*)\right\|^2 \biggr).
    \end{split}
\end{equation}
Then, we have
\begin{equation}
    \mathcal{E}(t) \geq t^k \left( f - f_* \right)
\end{equation}
for $0 \le k \le 2$. 

We consider the time derivative of this function:
\begin{equation}
    \begin{split}
        \dot{\mathcal{E}}(t) &= k t^{k-1} \biggl( \frac{4k - k^2}{2t^2} \|x-x_*\|^2 + \frac{k}{t} \langle x - x_*, \dot{x} \rangle + \frac{1}{2} \|\dot{x}\|^2 \biggr) \\
        &\phantom{=} + t^k \biggl( - \frac{4k - k^2}{t^3} \|x-x_*\|^2 + \frac{4k - k^2}{t^2} \langle x-x_*, \dot{x} \rangle - \frac{k}{t^2} \langle x-x_*, \dot{x} \rangle + \frac{k}{t} \|\dot{x}\|^2 + \frac{k}{t} \langle x-x_*, \ddot{x} \rangle + \langle \dot{x}, \ddot{x} \rangle\biggr) \\
        &\phantom{=} + k t^{k-1} (f-f_*) + t^k \langle \nabla f, \dot{x} \rangle \\
        &= t^k \biggl( \left( \frac{4k^2 - k^3}{2t^3} - \frac{4k - k^2}{t^3} \right) \|x-x_*\|^2 + \left( \frac{k^2}{t^2} + \frac{4k - k^2}{t^2} - \frac{k}{t^2} \right) \langle x-x_*, \dot{x} \rangle + \left( \frac{k}{2t} + \frac{k}{t} \right) \|\dot{x}\|^2 \\
        &\phantom{=} + \frac{k}{t} \langle x-x_*, \ddot{x} \rangle + \langle \dot{x}, \ddot{x} \rangle + \frac{k}{t} (f-f_*) + \langle \nabla f, \dot{x} \rangle \biggr) \\
        &= t^k \left( \frac{-k^3 + 6k^2 - 8k}{2t^3} \|x-x_*\|^2 + \frac{3k}{t^2} \langle x-x_*, \dot{x} \rangle + \frac{3k}{2t} \|\dot{x}\|^2 + \frac{k}{t} \langle x-x_*, \ddot{x} \rangle + \langle \dot{x}, \ddot{x} \rangle + \frac{k}{t} (f-f_*) + \langle \nabla f, \dot{x} \rangle \right)
    \end{split}
\end{equation}
Here, from the continuous dynamical system in equation \eqref{ODE_NAG_Convex}, we have
\begin{gather}
    \frac{k}{t} \langle x-x_*, \ddot{x} \rangle = - \frac{3k}{t^2} \langle x-x_*, \dot{x} \rangle - \frac{k}{t} \langle x-x_*, \nabla f \rangle, \\
    \langle \dot{x}, \ddot{x} \rangle = - \frac{3}{t} \|\dot{x}\|^2 - \langle \nabla f, \dot{x} \rangle.
\end{gather}
Therefore, we have
\begin{equation}
    \begin{split}
        \dot{\mathcal{E}}(t) &= t^k \biggl( \frac{-k^3 + 6k^2 - 8k}{2t^3} \|x-x_*\|^2 + \frac{3k}{t^2} \langle x-x_*, \dot{x} \rangle + \frac{3k}{2t} \|\dot{x}\|^2 \\
        &\phantom{=} - \frac{3k}{t^2} \langle x-x_*, \dot{x} \rangle - \frac{k}{t} \langle x-x_*, \nabla f \rangle - \frac{3}{t} \|\dot{x}\|^2 - \langle \nabla f, \dot{x} \rangle  + \frac{k}{t} (f-f_*) + \langle \nabla f, \dot{x} \rangle \biggr) \\
        &= t^k \biggl( \frac{-k^3 + 6k^2 - 8k}{2t^3} \|x-x_*\|^2 + \left(\frac{3k}{2t} - \frac{3}{t}\right) \|\dot{x}\|^2 - \frac{k}{t} (f_* - f - \langle \nabla f, x_* - x \rangle)\biggr) \\
        &= - t^k \biggl( \frac{k(k-2)(k-4)}{2t^3} \|x-x_*\|^2 + \frac{3(2-k)}{2t} \|\dot{x}\|^2 + \frac{k}{t} (f_* - f - \langle \nabla f, x_* - x \rangle)\biggr).
    \end{split}
\end{equation}

Thus, since the objective function is convex, we have
\begin{equation}
    \dot{\mathcal{E}}(t) \le 0
\end{equation}
for $0\le k \le 2$.
Hence $\mathcal{E}(t)$ is a Lyapunov function. Choosing $k=2$, we have
\begin{equation}
    t^2 (f-f_*) \leq \mathcal{E}(t) \leq \mathcal{E} (0) = \mathrm{const.},
\end{equation}
which implies
\begin{equation}
    f(x(t)) - f_* = \mathrm{O}\left( \frac{1}{t^2} \right).
\end{equation}

\subsection{Theorem \ref{NAG_strong_log}}
\label{sec:pr_thm_NAG_strong_log}

In Theorem \ref{NAG_strong_log}, for the continuous dynamical system
\begin{equation}
    \ddot{x} + \frac{r}{t} \dot{x} + \nabla f = 0 \quad (r>0),
    \label{ODE_NAG_strong_log}
\end{equation}
the convergence rate
\begin{equation}
    f(x(t)) - f_* = \mathrm{O}\left( \frac{1}{t^{\frac{1}{2}r+\frac{1}{2}}} \right)
\end{equation}
is guaranteed for a $\mu$-strongly convex objective function $f$. Here, we write down the Lyapunov function used in the proof and differentiate it explicitly to support the claim of the theorem.

For $r>3$, we define
\begin{equation}
    T = \frac{1}{\sqrt{\mu}} \sqrt{\left( \frac{r-1}{2} \right)^2 -1}. 
\end{equation}
Then $T>0$ holds because $r > 3$.
The number $r$ is given by 
\begin{gather}
    1 + \sqrt{1+\mu T^2} = \frac{1}{2}r + \frac{1}{2}, \\
    \Leftrightarrow r = 1 + 2 \sqrt{1+\mu T^2}.
\end{gather}
Using this $T$ and a constant $k$, 
we define the function
\begin{equation}
    \mathcal{E}(t) = t^k \left( \frac{\left( 2 + 2 \sqrt{1+ \mu T^2}\right)k - k^2}{2t^2} \|x-x_*\|^2 + \frac{k}{t} \langle x - x_*, \dot{x} \rangle + \frac{1}{2} \|\dot{x}\|^2 \right) + t^k (f-f_*)
\end{equation}
following matrix $P^{NAG}$ given in equation \eqref{Lyapunov_NAG_strong_log}.
Then, we have
\begin{equation}
    \begin{split}
        \mathcal{E}(t) - t^k \left( f - f_* \right) &= t^k \left( \frac{\left( 2 + 2 \sqrt{1+ \mu T^2}\right)k - k^2}{2t^2} \|x-x_*\|^2 + \frac{k}{t} \langle x - x_*, \dot{x} \rangle + \frac{1}{2} \|\dot{x}\|^2 \right) \\
        &= t^k \left( \frac{\left( 2 + 2 \sqrt{1+ \mu T^2}\right)k - k^2}{2t^2} \|x-x_*\|^2 + \frac{1}{2} \left\| \dot{x} + \frac{k}{t}(x-x_*) \right\|^2 - \frac{k^2}{2t^2} \|x-x_*\|^2 \right) \\
        &= t^k  \left( \frac{k\left(\left( 1 + \sqrt{1+ \mu T^2}\right) - k\right)}{t^2} \|x-x_*\|^2 + \frac{1}{2} \left\| \dot{x} + \frac{k}{t}(x-x_*) \right\|^2 \right),
    \end{split}
\end{equation}
which implies
\begin{equation}
    \mathcal{E}(t) \ge t^k \left( f - f_* \right)
\end{equation}
for $0 \le k \le 1 + \sqrt{1+\mu T^2}$.

We consider the time derivative of the function:
\begin{equation}
    \begin{split}
        \dot{\mathcal{E}}(t) &= k t^{k-1} \left( \frac{\left( 2 + 2 \sqrt{1+ \mu T^2}\right)k - k^2}{2t^2} \|x-x_*\|^2 + \frac{k}{t} \langle x - x_*, \dot{x} \rangle + \frac{1}{2} \|\dot{x}\|^2 \right) \\
        &\phantom{=} + t^k \left( - \frac{\left( 2 + 2 \sqrt{1+ \mu T^2}\right)k - k^2}{t^3} \|x-x_*\|^2 \right. + \frac{\left( 2 + 2 \sqrt{1+ \mu T^2}\right)k - k^2}{t^2} \langle x-x_*, \dot{x} \rangle \\
        &\phantom{=+ t^k} \left. \phantom{- \frac{\left(\sqrt{T^2}\right)}{t^3}} -\frac{k}{t^2} \langle x-x_*, \dot{x} \rangle + \frac{k}{t} \|\dot{x}\|^2 + \frac{k}{t} \langle x-x_*, \ddot{x} \rangle + \langle \dot{x}, \ddot{x} \rangle  \right) \\
        &\phantom{=} + k t^{k-1} (f-f_*) + t^k \langle \nabla f, \dot{x} \rangle \\
        &= t^k \left( \left( \frac{\left( 2 + 2 \sqrt{1+ \mu T^2}\right)k^2 - k^3}{2t^3} - \frac{\left( 2 + 2 \sqrt{1+ \mu T^2}\right)k - k^2}{t^3} \right) \|x-x_*\|^2\right.\\
        &\phantom{=} \phantom{- \frac{\left(\sqrt{T^2}\right)}{t^3}} + \left( \frac{k^2}{t^2} + \frac{\left( 2 + 2 \sqrt{1+ \mu T^2}\right)k - k^2}{t^2} - \frac{k}{t^2}\right) \langle x-x_*, \dot{x} \rangle + \left (\frac{k}{2t} + \frac{k}{t} \right) \|\dot{x}\|^2 \\
        &\phantom{=} \left. \phantom{- \frac{\left(\sqrt{T^2}\right)}{t^3}} +\frac{k}{t} \langle x-x_*, \ddot{x} \rangle + \langle \dot{x}, \ddot{x} \rangle + \frac{k}{t} (f-f_*) + \langle \nabla f, \dot{x} \rangle \right) \\
        &= t^k \left( \frac{-k^3 + \left(4+2\sqrt{1+\mu T^2}\right) k^2 - \left(4+4\sqrt{1+\mu T^2}\right)k}{2t^3} \|x-x_*\|^2\right.\\
        &\phantom{=} \phantom{- \frac{\left(\sqrt{T^2}\right)}{t^3}} +  \frac{\left( 1 + 2 \sqrt{1+ \mu T^2}\right)k}{t^2} \langle x-x_*, \dot{x} \rangle + \frac{3k}{2t} \|\dot{x}\|^2 \\
        &\phantom{=} \left. \phantom{- \frac{\left(\sqrt{T^2}\right)}{t^3}} +\frac{k}{t} \langle x-x_*, \ddot{x} \rangle + \langle \dot{x}, \ddot{x} \rangle + \frac{k}{t} (f-f_*) + \langle \nabla f, \dot{x} \rangle \right)
    \end{split}
\end{equation}
Here, from the continuous dynamical system in equation \eqref{ODE_NAG_strong_log}, we have
\begin{gather}
    \frac{k}{t} \langle x-x_*, \ddot{x} \rangle = - \frac{\left( 1 + 2 \sqrt{1+ \mu T^2}\right)k}{t^2} \langle x-x_*, \dot{x} \rangle - \frac{k}{t} \langle x-x_*, \nabla f \rangle, \\
    \langle \dot{x}, \ddot{x} \rangle = - \frac{1+2\sqrt{1+\mu T^2}}{t} \|\dot{x}\|^2 - \langle \nabla f, \dot{x} \rangle.
\end{gather}
Therefore we have
\begin{equation}
    \begin{split}
        \dot{\mathcal{E}}(t) &= t^k \left( \frac{-k^3 + \left(4+2\sqrt{1+\mu T^2}\right) k^2 - \left(4+4\sqrt{1+\mu T^2}\right)k}{2t^3} \|x-x_*\|^2\right.\\
        &\phantom{=} \phantom{- \frac{\left(\sqrt{T^2}\right)}{t^3}} +  \frac{\left( 1 + 2 \sqrt{1+ \mu T^2}\right)k}{t^2} \langle x-x_*, \dot{x} \rangle + \frac{3k}{2t} \|\dot{x}\|^2 - \frac{\left( 1 + 2 \sqrt{1+ \mu T^2}\right)k}{t^2} \langle x-x_*, \dot{x} \rangle\\
        &\phantom{=} \left. \phantom{- \frac{\left(\sqrt{T^2}\right)}{t^3}}  - \frac{k}{t} \langle x-x_*, \nabla f \rangle - \frac{1+2\sqrt{1+\mu T^2}}{t} \|\dot{x}\|^2 - \langle \nabla f, \dot{x} \rangle  + \frac{k}{t} (f-f_*) + \langle \nabla f, \dot{x} \rangle \right)\\
        &= t^k \left( \frac{-k^3 + \left(4+2\sqrt{1+\mu T^2}\right) k^2 - \left(4+4\sqrt{1+\mu T^2}\right)k}{2t^3} \|x-x_*\|^2\right.\\
        &\phantom{=} \left. \phantom{- \frac{\left(\sqrt{T^2}\right)}{t^3}}  + \frac{3k}{2t} \|\dot{x}\|^2 - \frac{k}{t} \langle x-x_*, \nabla f \rangle - \frac{1+2\sqrt{1+\mu T^2}}{t} \|\dot{x}\|^2  + \frac{k}{t} (f-f_*) \right)\\
        &= t^k \left( - \frac{k \left( k - \left( 1 + \sqrt{1 + \mu T^2} \right)\right)  \left( k - \left( 3 + \sqrt{1 + \mu T^2} \right)\right)}{2t^3} \|x-x_*\|^2 + \frac{\mu k T^2}{2t^3} \|x-x_*\|^2\right.\\
        &\phantom{=} \left. \phantom{\frac{\left(\sqrt{T^2}\right)}{t^3}}  + \left(\frac{3k}{2t}- \frac{1+2\sqrt{1+\mu T^2}}{t} \right) \|\dot{x}\|^2 - \frac{k}{t} \left(f_* - f - \langle \nabla f, x_* - x \rangle - \frac{\mu}{2} \|x-x_*\|^2 \right) -\frac{\mu k}{2t} \|x-x_*\|^2\right)\\
        &= t^k \left( - \frac{k \left( k - \left( 1 + \sqrt{1 + \mu T^2} \right) \right) \left( k - \left( 3 + \sqrt{1 + \mu T^2} \right)\right)}{2t^3} \|x-x_*\|^2 - \frac{\mu k (t^2 - T^2)}{2t^3} \|x-x_*\|^2\right.\\
        &\phantom{=}  \phantom{\frac{\left(\sqrt{T^2}\right)}{t^3}}  + \frac{3\left(k - \left( 1+\sqrt{1+\mu T^2} \right)\right)}{2t} \|\dot{x}\|^2 - \frac{\sqrt{1+ \mu T^2} -1}{2t} \|\dot{x}\|^2 \\
        &\phantom{=} \left. \phantom{\frac{\left(\sqrt{T^2}\right)}{t^3}}
        - \frac{k}{t} \left(f_* - f - \langle \nabla f, x_* - x \rangle - \frac{\mu}{2} \|x-x_*\|^2 \right)\right).\\
    \end{split}
\end{equation}

Thus, since $f$ is convex, we have
\begin{equation}
    \dot{\mathcal{E}}(t) \le 0
\end{equation}
for $0\le k \le 1 + \sqrt{1+\mu T^2}$ and $t \geq T$.
Hence $\mathcal{E}(t)$ is indeed a Lyapunov function. Setting $k=1 + \sqrt{1+\mu T^2}$, we have
\begin{equation}
    t^{1 + \sqrt{1+\mu T^2}} (f-f_*) \le \mathcal{E}(t) \le \mathcal{E} (T) = \mathrm{const.}
\end{equation}
for $t \geq T$, which implies
\begin{equation}
    f(x(t)) - f_* = \mathrm{O}\left( \frac{1}{t^{1 + \sqrt{1+\mu T^2}}} \right).
\end{equation}
Finally, by using the relation
\begin{equation}
    1 + \sqrt{1+\mu T^2} = \frac{1}{2}r + \frac{1}{2},
\end{equation}
we obtain
\begin{equation}
    f(x(t)) - f_* = \mathrm{O}\left( \frac{1}{t^{\frac{1}{2}r + \frac{1}{2}}} \right),
\end{equation}

\subsection{Theorem \ref{NAG_strong_exp}}
\label{sec:pr_thm_NAG_strong_exp}

In Theorem \ref{NAG_strong_exp}, for the continuous dynamical system
\begin{equation}
    \ddot{x} + \frac{4(k^2+\mu)}{k^2 t} \dot{x} + \nabla f = 0 \quad (k>0),
    \label{ODE_NAG_strong_exp}
\end{equation}
the convergence rate
\begin{equation}
    f(x(t)) - f_* = \mathrm{O}\left( \mathrm{e}^{-kt} \right)
\end{equation}
is guaranteed for a $\mu$-strongly convex objective function $f$ at times $0 < t \leq T = \frac{2(k^2 + \mu)}{k^3}$. Here, we support the claim of the theorem by explicitly writing down the Lyapunov function used in the proof and differentiating it.

For the constant $k$, we define the function
\begin{equation}
    \mathcal{E}(t) = \left(\frac{2(k^2+\mu)}{kt} - \frac{k^2}{2} \right) \mathrm{e}^{kt} \| x-x_* \|^2 + k \mathrm{e}^{kt} \langle x-x_*, \dot{x} \rangle + \frac{1}{2} \mathrm{e}^{kt} \|\dot{x}\|^2 + \mathrm{e}^{kt} (f-f_*)
    \label{eq:E_of_NAG_strong_exp}
\end{equation}
following matrix $P^{NAG}$ given in equation \eqref{Lyapunov_NAG_strong_exp}. Then, we have
\begin{equation}
    \begin{split}
        \mathcal{E}(t) - \mathrm{e}^{kt} \left( f - f_* \right) &= \left(\frac{2(k^2+\mu)}{kt} - \frac{k^2}{2} \right) \mathrm{e}^{kt} \| x-x_* \|^2 + k \mathrm{e}^{kt} \langle x-x_*, \dot{x} \rangle + \frac{1}{2} \mathrm{e}^{kt} \|\dot{x}\|^2 \\
        &= \left(\frac{2(k^2+\mu)}{kt} - \frac{k^2}{2} \right) \mathrm{e}^{kt} \| x-x_* \|^2 + \frac{1}{2} \mathrm{e}^{kt} \|\dot{x} + k(x-x_*)\|^2 - \frac{k^2}{2} \mathrm{e}^{kt}\|x-x_*\|^2 \\
        &= \left(\frac{2(k^2+\mu)}{kt} - k^2 \right) \mathrm{e}^{kt} \| x-x_* \|^2 + \frac{1}{2} \mathrm{e}^{kt} \|\dot{x} + k(x-x_*)\|^2 \\
        &= \frac{2(k^2+\mu)}{k} \left(\frac{1}{t} - \frac{k^3}{2(k^2+\mu)} \right) \mathrm{e}^{kt} \| x-x_* \|^2 + \frac{1}{2} \mathrm{e}^{kt} \|\dot{x} + k(x-x_*)\|^2.
    \end{split}
\end{equation}
Therefore, for $0 < t \leq T = \frac{2(k^2 + \mu)}{k^3}$, we have
\begin{equation}
    \mathcal{E}(t) \geq \mathrm{e}^{kt} \left( f - f_* \right).
\end{equation}

We consider the time derivative of this function:
\begin{equation}
    \begin{split}
        \dot{\mathcal{E}}(t) &= - \frac{2(k^2+\mu)}{kt^2} \mathrm{e}^{kt} \|x-x_*\|^2 + \left( \frac{2(k^2+\mu)}{kt} - \frac{k^2}{2} \right) k \mathrm{e}^{kt} \|x-x_*\|^2 + \left( \frac{4(k^2+\mu)}{kt} - k^2 \right) \langle x-x_*, \dot{x} \rangle \\
        &\phantom{=} + k^2 \mathrm{e}^{kt} \langle x-x_*, \dot{x} \rangle + k \mathrm{e}^{kt} \|\dot{x}\|^2 + k \mathrm{e}^{kt} \langle x-x_*,\ddot{x} \rangle + \frac{k}{2} \mathrm{e}^{kt} \|\dot{x}\|^2 + \mathrm{e}^{kt} \langle \dot{x}, \ddot{x} \rangle \\
        &\phantom{=} + k \mathrm{e}^{kt} (f-f_*) + \mathrm{e}^{kt} \langle \nabla f, \dot{x} \rangle \\
        &= \left( - \frac{k^3}{2} + \frac{2(k^2+\mu)}{t} - \frac{2(k^2+\mu)}{kt^2}\right) \mathrm{e}^{kt} \|x-x_*\|^2 + \frac{4(k^2+\mu)}{kt} \langle x-x_*, \dot{x} \rangle \\
        &\phantom{=} + \frac{3k}{2} \mathrm{e}^{kt} \|\dot{x}\|^2 + k \mathrm{e}^{kt} \langle x-x_*,\ddot{x} \rangle +  \mathrm{e}^{kt} \langle \dot{x}, \ddot{x} \rangle  + k \mathrm{e}^{kt} (f-f_*) + \mathrm{e}^{kt} \langle \nabla f, \dot{x} \rangle.
    \end{split}
\end{equation}
Here, from the continuous dynamical system in equation \eqref{ODE_NAG_strong_exp}, we have
\begin{gather}
    k \mathrm{e}^{kt} \langle x-x_*, \ddot{x} \rangle = - \frac{4(k^2+\mu)}{kt} \mathrm{e}^{kt} \langle x-x_*, \dot{x} \rangle - k \mathrm{e}^{kt} \langle x-x_*, \nabla f \rangle, \\
    \mathrm{e}^{kt} \langle \dot{x}, \ddot{x} \rangle = - \frac{4(k^2+\mu)}{k^2t} \mathrm{e}^{kt} \|\dot{x}\|^2 - \mathrm{e}^{kt} \langle \nabla f, \dot{x} \rangle.
\end{gather}
From these, we have
\begin{equation}
    \begin{split}
        \dot{\mathcal{E}}(t) &= \left( - \frac{k^3}{2} + \frac{2(k^2+\mu)}{t} - \frac{2(k^2+\mu)}{kt^2}\right) e^{kt} \|x-x_*\|^2 + \frac{4(k^2+\mu)}{kt} \langle x-x_*, \dot{x} \rangle \\
        &\phantom{=} + \frac{3k}{2} e^{kt} \|\dot{x}\|^2 - \frac{4(k^2+\mu)}{kt} e^{kt} \langle x-x_*, \dot{x} \rangle - k e^{kt} \langle x-x_*, \nabla f \rangle - \frac{4(k^2+\mu)}{k^2t} e^{kt} \|\dot{x}\|^2 - e^{kt} \langle \nabla f, \dot{x} \rangle \\
        &\phantom{=} + k e^{kt} (f-f_*) + e^{kt} \langle \nabla f, \dot{x} \rangle \\
        &= \left( - \frac{k^3}{2} + \frac{2(k^2+\mu)}{t} - \frac{2(k^2+\mu)}{kt^2}\right) e^{kt} \|x-x_*\|^2 \\
        &\phantom{=} + \left( \frac{3k}{2} - \frac{4(k^2+\mu)}{k^2t}\right)e^{kt}\|\dot{x}\|^2 - k e^{kt} \langle x-x_*, \nabla f \rangle + k e^{kt} (f-f_*) \\
        &= - \frac{k^2+\mu}{2kt^2} (k^2t^2 - 4kt + 4) e^{kt} \|x-x_*\|^2 + \frac{\mu k}{2} e^{kt} \|x-x_*\|^2 + \left( \frac{3k}{2} - \frac{4(k^2+\mu)}{k^2t} \right) \|\dot{x}\|^2 \\
        &\phantom{=} - k e^{kt} \left( f_* - f - \langle \nabla f, x_* - x \rangle - \frac{\mu}{2} \|x-x_*\|^2 \right) - \frac{\mu k}{2} e^{kt} \|x-x_*\|^2 \\
        &= - \frac{k^2+\mu}{2kt^2} (kt-2)^2 e^{kt} \|x-x_*\|^2 - \frac{4(k^2+\mu)}{k^2} \left(  \frac{1}{t} - \frac{3k^3}{8(k^2+\mu)} \right) \|\dot{x}\|^2 \\
        &\phantom{=} - k e^{kt} \left( f_* - f - \langle \nabla f, x_* - x \rangle - \frac{\mu}{2} \|x-x_*\|^2 \right) .
    \end{split}
\end{equation}

Therefore, since the objective function is $\mu$-strongly convex, we have
\begin{equation}
    \dot{\mathcal{E}}(t) \leq 0
\end{equation}
for $0 < t \leq T = \frac{2(k^2 + \mu)}{k^3}$.
Hence $\mathcal{E}(t)$ is indeed a Lyapunov function. Thus we have
\begin{equation}
    \mathrm{e}^{kt} (f-f_*) \leq \mathcal{E}(t) \leq \mathcal{E} (t_0) = \mathrm{const.},
\end{equation}
which implies
\begin{equation}
    f(x(t)) - f_* = \mathrm{O}\left( \mathrm{e}^{-kt} \right).
\end{equation}

\subsection{Theorem \ref{Generalized_NAG}}
\label{sec:pr_thm_Generalized_NAG}

In Theorem \ref{Generalized_NAG}, for the continuous dynamical system
\begin{equation}
    \ddot{x} + \frac{r}{t^\alpha} \dot{x} + \nabla f = 0 ~(r>0, ~0<\alpha<1)
    \label{ODE_Generalized_NAG}
\end{equation}
and for a $\mu$-strongly convex objective function $f$, the convergence rate
\begin{equation}
    f(x(t)) - f_* = \mathrm{O} \left( \mathrm{e}^{-\left(\frac{2}{3}-\epsilon\right)\frac{r}{1-\alpha}t^{1-\alpha}} \right)
\end{equation}
is guaranteed for any $\epsilon>0$. Here, by explicitly writing down the Lyapunov function used in the proof and differentiating it, we support the claim of the theorem.

For a constant $k>0$, we define the function
\begin{equation}
    \mathcal{E}(t) = \frac{r^2k}{2} t^{-2\alpha} \mathrm{e}^{k\frac{r}{1-\alpha}t^{1-\alpha}} \|x-x_*\|^2 + rk t^{-\alpha} \mathrm{e}^{k\frac{r}{1-\alpha}t^{1-\alpha}} \langle x-x_*, \dot{x} \rangle + \frac{1}{2} \mathrm{e}^{k\frac{r}{1-\alpha}t^{1-\alpha}} \|\dot{x}\|^2 + \mathrm{e}^{k\frac{r}{1-\alpha}t^{1-\alpha}} (f-f_*)
\end{equation}
following matrix $P^{G-NAG}$ given in equation \eqref{Lyapunov_Generalized_NAG}. Then, we have
\begin{equation}
    \begin{split}
        \mathcal{E}(t) - \mathrm{e}^{k\frac{r}{1-\alpha}t^{1-\alpha}} (f-f_*) &= \frac{r^2k}{2} t^{-2\alpha} \mathrm{e}^{k\frac{r}{1-\alpha}t^{1-\alpha}} \|x-x_*\|^2 + rk t^{-\alpha} \mathrm{e}^{k\frac{r}{1-\alpha}t^{1-\alpha}} \langle x-x_*, \dot{x} \rangle + \frac{1}{2} \mathrm{e}^{k\frac{r}{1-\alpha}t^{1-\alpha}} \|\dot{x}\|^2 \\
        &= \frac{1}{2} \mathrm{e}^{k\frac{r}{1-\alpha}t^{1-\alpha}} \left( \|\dot{x}\|^2 + 2rkt^{-\alpha} \langle x-x_*, \dot{x} \rangle + r^2k t^{-2\alpha} \|x-x_*\|^2 \right) \\
        &= \frac{1}{2} \mathrm{e}^{k\frac{r}{1-\alpha}t^{1-\alpha}} \left( \|\dot{x} + rk t^{-\alpha} (x-x_*)\|^2 - r^2k^2 t^{-2\alpha} \|x-x_*\|^2 + r^2k t^{-2\alpha} \|x-x_*\|^2 \right) \\
        &= \frac{1}{2} \mathrm{e}^{k\frac{r}{1-\alpha}t^{1-\alpha}} \|\dot{x} + rk t^{-\alpha} (x-x_*)\|^2 + \frac{r^2k(1-k)}{2} t^{-2\alpha} \mathrm{e}^{k\frac{r}{1-\alpha}t^{1-\alpha}} \|x-x_*\|^2.
    \end{split}
\end{equation}
Therefore, we have
\begin{equation}
    \mathcal{E}(t) \geq \mathrm{e}^{k\frac{r}{1-\alpha}t^{1-\alpha}} (f-f_*)
\end{equation}
for $0\leq k \leq r$.

We consider the time derivative of this function:
\begin{equation}
    \begin{split}
        \dot{\mathcal{E}}(t) &= - r^2\alpha k t^{-1-2\alpha} \mathrm{e}^{k\frac{r}{1-\alpha}t^{1-\alpha}} \|x-x_*\|^2 + \frac{r^3k^2}{2} t^{-3\alpha} \mathrm{e}^{k\frac{r}{1-\alpha}t^{1-\alpha}} \|x-x_*\|^2 + r^2k t^{-2\alpha} \mathrm{e}^{k\frac{r}{1-\alpha}t^{1-\alpha}} \langle x-x_*, \dot{x} \rangle \\
        &\phantom{=} - r \alpha k t^{-1-\alpha} \mathrm{e}^{k\frac{r}{1-\alpha}t^{1-\alpha}} \langle x-x_*, \dot{x} \rangle + r^2k^2 t^{-2\alpha} \mathrm{e}^{k\frac{r}{1-\alpha}t^{1-\alpha}} \langle x-x_*, \dot{x} \rangle  + rk t^{-\alpha} \mathrm{e}^{k\frac{r}{1-\alpha}t^{1-\alpha}} \|\dot{x}\|^2 \\
        &\phantom{=} + rk t^{-\alpha} \mathrm{e}^{k\frac{r}{1-\alpha}t^{1-\alpha}} \langle x-x_*, \ddot{x} \rangle + \frac{rk}{2} t^{-\alpha} \mathrm{e}^{k\frac{r}{1-\alpha}t^{1-\alpha}} \|\dot{x}\|^2 + \mathrm{e}^{k\frac{r}{1-\alpha}t^{1-\alpha}} \langle \dot{x}, \ddot{x} \rangle \\
        &\phantom{=} +rk t^{-\alpha} \mathrm{e}^{k\frac{r}{1-\alpha}t^{1-\alpha}} (f-f_*) + \mathrm{e}^{k\frac{r}{1-\alpha}t^{1-\alpha}} \langle \nabla f, \dot{x} \rangle \\
        &= \left( \frac{r^3k^2}{2} t^{-3\alpha} - r^2 \alpha k t^{-1-2\alpha} \right) \mathrm{e}^{k\frac{r}{1-\alpha}t^{1-\alpha}} \|x-x_*\|^2 \\
        &\phantom{=} + \left( r^2k^2 t^{-2\alpha} + r^2k t^{-2\alpha} - r \alpha k t^{-1-\alpha}  \right) \mathrm{e}^{k\frac{r}{1-\alpha}t^{1-\alpha}} \langle x-x_*, \dot{x} \rangle \\
        &\phantom{=} + \frac{3rk}{2} t^{-\alpha}\mathrm{e}^{k\frac{r}{1-\alpha}t^{1-\alpha}} \|\dot{x}\|^2 + rk t^{-\alpha} \mathrm{e}^{k\frac{r}{1-\alpha}t^{1-\alpha}} \langle x-x_*, \ddot{x} \rangle + \mathrm{e}^{k\frac{r}{1-\alpha}t^{1-\alpha}} \langle \dot{x}, \ddot{x} \rangle \\
        &\phantom{=} + rk t^{-\alpha} \mathrm{e}^{k\frac{r}{1-\alpha}t^{1-\alpha}} (f-f_*) + \mathrm{e}^{k\frac{r}{1-\alpha}t^{1-\alpha}} \langle \nabla f, \dot{x} \rangle.
    \end{split}
\end{equation}
Here, from the continuous dynamical system in equation \eqref{ODE_Generalized_NAG}, we have
\begin{gather}
    rk t^{-\alpha} \mathrm{e}^{k\frac{r}{1-\alpha}t^{1-\alpha}} \langle x-x_*, \ddot{x} \rangle = - r^2k t^{-2\alpha} \mathrm{e}^{k\frac{r}{1-\alpha}t^{1-\alpha}} \langle x-x_*, \dot{x} \rangle - rk t^{-\alpha} \mathrm{e}^{k\frac{r}{1-\alpha}t^{1-\alpha}} \langle x-x_*, \nabla f \rangle \\
    \mathrm{e}^{k\frac{r}{1-\alpha}t^{1-\alpha}} \langle \dot{x}, \ddot{x} \rangle = - r t^{-\alpha} \mathrm{e}^{k\frac{r}{1-\alpha}t^{1-\alpha}} \|\dot{x}\|^2 - \mathrm{e}^{k\frac{r}{1-\alpha}t^{1-\alpha}} \langle \dot{x}, \nabla f \rangle.
\end{gather}
From these, we have
\begin{equation}
    \begin{split}
        \dot{\mathcal{E}}(t) &= \left( \frac{r^3k^2}{2} t^{-3\alpha} - r^2 \alpha k t^{-1-2\alpha} \right) \mathrm{e}^{k\frac{r}{1-\alpha}t^{1-\alpha}} \|x-x_*\|^2 \\
        &\phantom{=} + \left( r^2k^2 t^{-2\alpha} + r^2k t^{-2\alpha} - r \alpha k t^{-1-\alpha}  \right) \mathrm{e}^{k\frac{r}{1-\alpha}t^{1-\alpha}} \langle x-x_*, \dot{x} \rangle \\
        &\phantom{=} + \frac{3rk}{2} t^{-\alpha}\mathrm{e}^{k\frac{r}{1-\alpha}t^{1-\alpha}} \|\dot{x}\|^2 - r^2k t^{-2\alpha} \mathrm{e}^{k\frac{r}{1-\alpha}t^{1-\alpha}} \langle x-x_*, \dot{x} \rangle - rk t^{-\alpha} \mathrm{e}^{k\frac{r}{1-\alpha}t^{1-\alpha}} \langle x-x_*, \nabla f \rangle \\
        &\phantom{=} - r t^{-\alpha} \mathrm{e}^{k\frac{r}{1-\alpha}t^{1-\alpha}} \|\dot{x}\|^2 - \mathrm{e}^{k\frac{r}{1-\alpha}t^{1-\alpha}} \langle \dot{x}, \nabla f \rangle \\
        &\phantom{=} + rk t^{-\alpha} \mathrm{e}^{k\frac{r}{1-\alpha}t^{1-\alpha}} (f-f_*) + \mathrm{e}^{k\frac{r}{1-\alpha}t^{1-\alpha}} \langle \nabla f, \dot{x} \rangle \\
        &= - rk t^{-\alpha} \mathrm{e}^{k\frac{r}{1-\alpha}t^{1-\alpha}} \left( f_* - f - \langle \nabla f, x_* - x \rangle - \frac{\mu}{2} \|x-x_*\|^2 \right) \\
        &\phantom{=} + \left( - \frac{\mu rk}{2} t^{-\alpha} + \frac{r^3k^2}{2} t^{-3\alpha} - r^2 \alpha k t^{-1-2\alpha} \right) \mathrm{e}^{k\frac{r}{1-\alpha}t^{1-\alpha}} \|x-x_*\|^2 \\
        &\phantom{=} + \left( r^2k^2 t^{-2\alpha} - r\alpha k t^{-1-\alpha} \right) \mathrm{e}^{k\frac{r}{1-\alpha}t^{1-\alpha}} \langle x-x_*, \dot{x} \rangle + \frac{r}{2}\left( 3k-2 \right) t^{-\alpha} \mathrm{e}^{k\frac{r}{1-\alpha}t^{1-\alpha}} \|\dot{x}\|^2.
    \end{split}
\end{equation}
Furthermore, we can derive its expression for $k \neq \frac{2}{3}$ as follows:
\begin{equation}
    \begin{split}
        \dot{\mathcal{E}}(t) &= - rk t^{-\alpha} \mathrm{e}^{k\frac{r}{1-\alpha}t^{1-\alpha}} \left( f_* - f - \langle \nabla f, x_* - x \rangle - \frac{\mu}{2} \|x-x_*\|^2 \right) \\
        &\phantom{=} - \frac{rk}{2} \left( \mu - \left(r^2k - \frac{r^2k^3}{3k-2} \right) t^{-2\alpha} - \left( \frac{2r\alpha k^2}{3k-2} - 2r\alpha \right)t^{-1-\alpha} + \frac{\alpha^2 k}{3k-2} t^{-2}\right) t^{-\alpha} \mathrm{e}^{k\frac{r}{1-\alpha}t^{1-\alpha}} \|x-x_*\|^2 \\
        &\phantom{=} + \frac{r}{2} \left(3k-2\right) t^{-\alpha} \mathrm{e}^{k\frac{r}{1-\alpha}t^{1-\alpha}} \left\| \dot{x} + \frac{k(rk t^{-\alpha} - \alpha t^{-1})}{3k-2} (x-x_*)\right\|^2.
    \end{split}
\end{equation}
Because $f$ is $\mu$-strongly convex, for $0<k<\frac{2}{3}$ and sufficiently large $t$, 
\begin{equation}
    \dot{\mathcal{E}}(t) \leq 0
\end{equation}
holds. 
Let $T$ be the minimal $t$ satisfying this condition. 
Then $\mathcal{E}(t)$ is a Lyapunov function for any $t \geq T$. 
Therefore we have
\begin{equation}
    \mathrm{e}^{k\frac{r}{1-\alpha}t^{1-\alpha}}(f-f_*) \leq \mathcal{E}(t) \leq \mathcal{E}(T) = \mathrm{const.
    }
\end{equation}
Finally, setting $k = \frac{2}{3} - \epsilon$ with an arbitrarily small positive constant $\epsilon$, 
we have
\begin{equation}
    f(x(t)) - f_* = \mathrm{O} \left( \mathrm{e}^{-\left(\frac{2}{3}-\epsilon\right)\frac{r}{1-\alpha}t^{1-\alpha}} \right).
\end{equation}

\section{Incorporating the method in Su, Boyd, Candes~(2016)~\cite{su2016differential}}
For a $\mu$-strongly convex objective function $f$ and the continuous dynamical system
\begin{equation}
    \ddot{x} + \frac{r}{t} \dot{x} + \nabla f = 0 \quad (r>3),
    \label{eq:ODE_gNAG_in_AppendixD}
\end{equation}
the proposed method in this study finds the convergence rate
\begin{equation}
    f(x(t)) - f_* = \mathrm{O} \left( \frac{1}{t^{\frac{1}{2}r+\frac{1}{2}}} \right)
    \label{eq:rate_r_1_div_2}
\end{equation}
as shown in Theorem~\ref{NAG_strong_log}.
However, Su, Boyd, Candes~(2016)~\cite{su2016differential} obtained
\begin{equation}
    f(x(t)) - f_* = \mathrm{O} \left( \frac{1}{t^{\frac{2}{3}r}} \right),
    \label{eq:SBC_rate_for_gNAG}
\end{equation}
which is strictly better than the convergence rate in \eqref{eq:rate_r_1_div_2} for $r>3$. 
Indeed, they used an alternative rate-giving function that is not a Lyapunov function. Then they obtained the rate in \eqref{eq:SBC_rate_for_gNAG} by an additional technique. 
Taking these facts into consideration, we examine the possibility of incorporating them into our method in this section. To this end, we review the alternative function and the additional technique 
of Su, Boyd, Candes~(2016) in Appendix~\ref{sec:E_not_decrease_by_SBC}. Then, in Appendix~\ref{sec:appl_SBC_to_ours},  
we show that we can incorporate them into our method to derive the better rate in \eqref{eq:SBC_rate_for_gNAG}.

\subsection{Review of the Method in Su, Boyd, Candes~(2016)~\cite{su2016differential}}
\label{sec:E_not_decrease_by_SBC}

The alternative function and its time derivative used in Su, Boyd, Candes~(2016)~\cite{su2016differential} are given by
\begin{gather}
    \mathcal{E}(t) = \frac{2r^2}{9} t^{\frac{2}{3}r-2} \|x-x_*\|^2 + \frac{2r}{3} t^{\frac{2}{3}r-1} \langle x-x_*, \dot{x} \rangle + \frac{1}{2} t^{\frac{2}{3}r} \|\dot{x}\|^2 + t^{\frac{2}{3}r} (f-f_*), \\
    \dot{\mathcal{E}}(t) = \left( \frac{4r^3-12r^2}{27t^3}-\frac{r\lambda}{3t}\right)t^{\frac{2}{3}r} \|x-x_*\|^2 + \frac{2r^2-6r}{9} t^{\frac{2}{3}r-2}\langle x-x_*, \dot{x} \rangle,
    \label{eq:SBC_dot_E_for_gNAG}
\end{gather}
where the parameter $\lambda~(\geq \mu)$ is defined by
\begin{equation}
    \frac{\lambda}{2} \|x-x_*\|^2 = f_* - f - \langle \nabla f, x_* - x \rangle.
\end{equation}
We can write these functions by using the matrices
\begin{equation}
    P = \begin{pmatrix}
        \frac{2r^2}{9t^2} & 0 & \frac{r}{3t} \\
        0 & 0 & 0 \\
        \frac{r}{3t} & 0 & \frac{1}{2}
    \end{pmatrix}, \quad 
    Q = \begin{pmatrix}
        \frac{r\lambda}{3t} - \frac{4r^3-12r^2}{27t^3} & 0 & \frac{3r-r^2}{9t^2} & 0 & 0 \\
        0 & 0 & 0 & 0 & 0 \\
        \frac{3r-r^2}{9t^2} & 0 & 0 & 0 & 0 \\
        0 & 0 & 0 & 0 & 0 \\
        0 & 0 & 0 & 0 & 0
    \end{pmatrix}
\end{equation}
according to our method. 
The matrix $P$ is positive semi-definite for $t>0,~r>0$. 
However, the positive semi-definiteness of $Q$ fails for any $t$ because the $(1,3)$ and $(3,1)$ entries of $Q$ are nonzero for $r>3$. 
Therefore, our method cannot derive the convergence rate from the function $\mathcal{E}(t)$. However, Su, Boyd, Candes~(2016)~\cite{su2016differential} overcame this issue by the additional technique shown below. 

From~\eqref{eq:SBC_dot_E_for_gNAG}, 
% \begin{equation}
%     \dot{\mathcal{E}}(t) = \left( \frac{4r^3-12r^2}{27t^3}-\frac{r\lambda}{3t}\right)t^{\frac{2}{3}r} \|x-x_*\|^2 + \frac{2r^2-6r}{9} t^{\frac{2}{3}r-2}\langle x-x_*, \dot{x} \rangle,
% \end{equation}
we have
\begin{equation}
    \dot{\mathcal{E}}(t) \leq \frac{2r^2-6r}{9} t^{\frac{2}{3}r-2}\langle x-x_*, \dot{x} \rangle
    \label{Inequality_Su_Boyd_Candes}
\end{equation}
for sufficiently large $t$. This derivation corresponds to showing the non-negativity of the $(1,1)$ entry of $Q$ in our method. 
%and is also possible there. 
Integrating both sides of inequality \eqref{Inequality_Su_Boyd_Candes} from time $t_0$ to $t$ gives
\begin{equation}
    \begin{split}
        \mathcal{E}(t)-\mathcal{E}(t_0) &= \int_{t_0}^t \dot{\mathcal{E}}(s) \mathrm{d} s \\
        &\leq \int_{t_0}^t \frac{2r^2-6r}{9} s^{\frac{2}{3}r-2}\langle x-x_*, \dot{x} \rangle \mathrm{d}s \\
        &= \frac{r^2 - 3r}{9} \int_{t_0}^t s^{\frac{2}{3}r-2} \left( \|x-x_*\|^2 \right)^{\prime} \mathrm{d}s \\
        &= \frac{r^2 - 3r}{9} t^{\frac{2}{3}r-2}\|x-x_*\|^2 - \frac{r^2 - 3r}{9} t_0^{\frac{2}{3}r-2}\|x_0-x_*\|^2 - \frac{r^2 - 3r}{9} \int_{t_0}^t \left(\frac{2}{3}r - 2 \right) s^{\frac{2}{3}r - 3}\|x-x_*\|^2 \mathrm{d}s.
    \end{split}
\end{equation}
Because $\frac{r^2 - 3r}{9}>0$ and $\frac{2}{3}r - 2  \geq 0$ hold for $r > 3$, we have
\begin{equation}
    \begin{split}
    \mathcal{E}(t)-\mathcal{E}(t_0) &\leq \frac{r^2 - 3r}{9} t^{\frac{2}{3}r-2}\|x-x_*\|^2 \\
    &\leq \frac{r^2 - 3r}{9} t^{\frac{2}{3}r-2} \cdot \frac{2}{\mu} (f-f_*) \\
    &= \frac{2r^2 - 6r}{9\mu} t^{\frac{2}{3}r-2}(f-f_*).
    \end{split}
    \label{eq:key1_for_Et_Et0}
\end{equation}
In Su, Boyd, Candes~(2016)~\cite{su2016differential}, it has already been shown that
\begin{equation}
    f(x(t)) - f_* = \mathrm{O} \left( \frac{1}{t^{2}} \right)
    \label{eq:key2_for_Et_Et0}
\end{equation}
for the same continuous dynamical system and convex objective function
(Theorem 5 in Su, Boyd, Candes~(2016)). 
Combining
\eqref{eq:key1_for_Et_Et0} and 
\eqref{eq:key2_for_Et_Et0}, 
we have
\begin{equation}
    \begin{split}
        \mathcal{E}(t) - \mathcal{E}(t_0) &\leq \frac{2r^2 - 6r}{9\mu} t^{\frac{2}{3}r-2} \mathrm{O} \left(\frac{1}{t^2}\right)\\
        &= \mathrm{O} \left( t^{\frac{2}{3}r-4} \right).
    \end{split}
\end{equation}

First, we consider the case that $3<r\leq 6$. Then we have $\frac{2}{3}r-4 \leq 0$ and
\begin{equation}
    \mathcal{E}(t) - \mathcal{E}(t_0) = \mathrm{O}(1),
\end{equation}
which implies
\begin{equation}
    t^{\frac{2}{3}r}(f-f_*) \leq \mathcal{E}(t) \leq \mathcal{E} (t_0) \leq \mathrm{const.}
\end{equation}
This shows the desired convergence rate.  
Next, we consider the case that $r>6$. 
In this case, 
we iterate a similar procedure to the above one. That is, we first show a smaller rate, and then iteratively improve the rate to complete the proof for any $r>3$.
% rather than proving $\mathrm{O} \left( \frac{1}{t^{\frac{2}{3}r}} \right)$ directly, 

To incorporate the above methods into ours, we should note the following two features of the methods:
\begin{itemize}
    \item Even if there are entries that break the positive semi-definiteness of $Q$, they can be ignored if the corresponding terms become constant-order after integration.
    \item In the order evaluation above, we use the best convergence rate already obtained, and repeat this until the convergence rate can no longer be improved.
\end{itemize}

\subsection{Incorporating the Method in Su, Boyd, Candes~(2016)~\cite{su2016differential} into ours}
\label{sec:appl_SBC_to_ours}

Besides $P^{NAG}$ and $Q^{NAG}$, 
we can find the matrices
\begin{equation}
    P = 
    \begin{pmatrix}
        \frac{r \dot{\gamma}}{2t} & 0 & \frac{\dot{\gamma}}{2} \\
        0 & 0 & 0 \\
        \frac{\dot{\gamma}}{2} & 0 & \frac{1}{2}
    \end{pmatrix}, \quad
    Q = 
    \begin{pmatrix}
        \frac{-rt (\dot{\gamma}^2 + \ddot{\gamma})  + \dot{\gamma} (r + \lambda t^2)}{2t^2} & 0 & \frac{1}{2} \left( - \dot{\gamma}^2 - \ddot{\gamma} \right) & 0 & 0 \\
        0 & 0 & 0 & 0 & 0 \\
        \frac{1}{2} \left( - \dot{\gamma}^2 - \ddot{\gamma} \right) & 0 & \frac{r}{t} - \frac{3\dot{\gamma}}{2} & 0 & 0 \\
        0 & 0 & 0 & 0 & 0 \\
        0 & 0 & 0 & 0 & 0
    \end{pmatrix}
\end{equation}
among the matrices given by our method for \eqref{eq:ODE_gNAG_in_AppendixD}. 
By substituting $\gamma = \frac{2}{3}\log{t}$ into these, we have
\begin{equation}
    P = 
    \begin{pmatrix}
        \frac{r^2}{3t^2} & 0 & \frac{r}{3t} \\
        0 & 0 & 0 \\
        \frac{r}{3t} & 0 & \frac{1}{2}
    \end{pmatrix},
    \quad
    Q = 
    \begin{pmatrix}
        \frac{r\lambda}{3t} - \frac{2r^3-6r^2}{9t^3} & 0 & \frac{3r-2r^2}{9t^2} & 0 & 0 \\
        0 & 0 & 0 & 0 & 0 \\
        \frac{3r-2r^2}{9t^2} & 0 & 0 & 0 & 0 \\
        0 & 0 & 0 & 0 & 0 \\
        0 & 0 & 0 & 0 & 0
    \end{pmatrix}.
\end{equation}
The matrix $P$ is positive semi-definite for $t>0$ and $r>0$, but $Q$ is not. 

Then we apply the aforementioned method in Su, Boyd, Candes~(2016)~\cite{su2016differential} to these matrices. 
The pair of these matrices represents the function and time derivative given by
\begin{gather}
    \mathcal{E}(t) = \frac{r^2}{3} t^{\frac{2}{3}r-2} \|x-x_*\|^2 + \frac{2r}{3} t^{\frac{2}{3}r-1} \langle x-x_*, \dot{x} \rangle + \frac{1}{2} t^{\frac{2}{3}r} \|\dot{x}\|^2 + t^{\frac{2}{3}r} (f-f_*), \\
    \dot{\mathcal{E}}(t) = \left( \frac{2r^3-6r^2}{9t^3}-\frac{r\lambda}{3t}\right)t^{\frac{2}{3}r} \|x-x_*\|^2 + \frac{4r^2-6r}{9} t^{\frac{2}{3}r-2}\langle x-x_*, \dot{x} \rangle,
    \label{eq:dot_E_for_gNAG_for_ours}
\end{gather}
where the parameter $\lambda (\geq \mu)$ satisfy
\begin{equation}
    \frac{\lambda}{2} \|x-x_*\|^2 = f_* - f - \langle \nabla f, x_* - x \rangle.
\end{equation}
Since $P$ is positive semi-definite, we have
\begin{equation}
    \mathcal{E}(t) \geq t^{\frac{2}{3}r} (f-f_*)
\end{equation}
for all $t>0$. 

From~\eqref{eq:dot_E_for_gNAG_for_ours},
% \begin{equation}
%     \dot{\mathcal{E}}(t) = \left( \frac{2r^3-6r^2}{9t^3}-\frac{r\lambda}{3t}\right)t^{\frac{2}{3}r} \|x-x_*\|^2 + \frac{4r^2-6r}{9} t^{\frac{2}{3}r-2}\langle x-x_*, \dot{x} \rangle,
% \end{equation}
we have
\begin{equation}
    \dot{\mathcal{E}}(t) \leq \frac{4r^2-6r}{9} t^{\frac{2}{3}r-2}\langle x-x_*, \dot{x} \rangle
\end{equation}
for sufficiently large $t$.
Integrating both sides of this inequality from time $t_0$ to $t$ gives
\begin{equation}
    \begin{split}
        \mathcal{E}(t)-\mathcal{E}(t_0) &= \int_{t_0}^t \dot{\mathcal{E}}(s) \mathrm{d} s \\
        &\leq \int_{t_0}^t \frac{4r^2-6r}{9} s^{\frac{2}{3}r-2}\langle x-x_*, \dot{x} \rangle \mathrm{d}s \\
        &= \frac{2r^2 - 3r}{9} \int_{t_0}^t s^{\frac{2}{3}r-2} \left( \|x-x_*\|^2 \right)^{\prime} \mathrm{d}s \\
        &= \frac{2r^2 - 3r}{9} t^{\frac{2}{3}r-2}\|x-x_*\|^2 - \frac{2r^2 - 3r}{9} t_0^{\frac{2}{3}r-2}\|x_0-x_*\|^2 - \frac{2r^2 - 3r}{9} \int_{t_0}^t \left(\frac{2}{3}r - 2 \right) s^{\frac{2}{3}r - 3}\|x-x_*\|^2 \mathrm{d}s.
    \end{split}
\end{equation}
Because $\frac{2r^2 - 3r}{9}>0$ and $\frac{2}{3}r - 2  \geq 0$ hold for $r>3$, we have
\begin{equation}
    \begin{split}
    \mathcal{E}(t)-\mathcal{E}(t_0) &\leq \frac{2r^2 - 3r}{9} t^{\frac{2}{3}r-2}\|x-x_*\|^2 \\
    &\leq \frac{2r^2 - 3r}{9} t^{\frac{2}{3}r-2} \cdot \frac{2}{\mu} (f-f_*) \\
    &= \frac{4r^2 - 6r}{9\mu} t^{\frac{2}{3}r-2}(f-f_*).
    \end{split}
    \label{eq:key1_for_rate_in_ours}
\end{equation}
From Theorem \ref{NAG_Convex}, we know that
\begin{equation}
    f(x(t)) - f_* = \mathrm{O} \left( \frac{1}{t^{2}} \right)
    \label{eq:key2_for_rate_in_ours}
\end{equation}
for the same continuous dynamical system and convex objective function.
Combining
\eqref{eq:key1_for_rate_in_ours} and 
\eqref{eq:key2_for_rate_in_ours}, 
we have
\begin{equation}
    \begin{split}
        \mathcal{E}(t) - \mathcal{E}(t_0) &\leq \frac{2r^2 - 6r}{9\mu} t^{\frac{2}{3}r-2} \mathrm{O} \left(\frac{1}{t^2}\right)\\
        &= \mathrm{O} \left( t^{\frac{2}{3}r-4} \right).
    \end{split}
\end{equation}

In the case that $3<r\leq 6$, we have $\frac{2}{3}r-4 \leq 0$ and
\begin{equation}
    \mathcal{E}(t) - \mathcal{E}(t_0) = \mathrm{O}(1),
\end{equation}
which implies
\begin{equation}
    t^{\frac{2}{3}r}(f-f_*) \leq \mathcal{E}(t) \leq \mathcal{E} (t_0) \leq \mathrm{const.}
\end{equation}
This shows the desired convergence rate. In the case that $r>6$, we can  make use of the aformantioned recursive argument. Therefore, we now have the convergence rate
\begin{equation}
    f(x(t)) - f_* = \mathrm{O}\left(\frac{1}{t^{\frac{2}{3}r}}\right).
\end{equation}
% within the results of our proposed method through the use of the argument given in Su, Boyd, Candes~(2016)~\cite{su2016differential}.

\end{document}